\documentclass{article}

\usepackage[utf8]{inputenc} 
\usepackage[T1]{fontenc}    
\usepackage{hyperref}       
\usepackage{url}            
\usepackage{booktabs}       
\usepackage{amsfonts}       
\usepackage{nicefrac}       
\usepackage{microtype}      
\usepackage{cancel}

\usepackage{amsthm}
\usepackage{amsmath}
\usepackage{amssymb}
\usepackage{graphicx}
\usepackage{color}
\usepackage{ifpdf}
\usepackage{url}
\usepackage{hyperref}
\usepackage{algorithm}
\usepackage{abstract}
\usepackage[usenames,dvipsnames]{xcolor}
\usepackage{xfrac}
\usepackage{comment}

\usepackage[ruled,vlined, linesnumbered, algo2e]{algorithm2e}
\usepackage{amssymb,makeidx,mathrsfs}
\usepackage{fixltx2e}

\usepackage[T1]{fontenc} 

\usepackage{abstract} 

\theoremstyle{plain}
\newtheorem{theorem}{Theorem}[section]
\newtheorem{corollary}{Corollary}[section]
\newtheorem{lemma}{Lemma}[section]

\theoremstyle{definition}

\newtheorem{assumption}{Assumption}[section]
\newtheorem{remark}{Remark}[section]

\makeatletter
\renewcommand\@biblabel[1]{#1.}
\makeatother

\newcommand{\R}{\mathbb{R}}

\newcommand{\set}[1]{\left\{#1\right\}}
\newcommand{\sets}[1]{\{#1\}}
\newcommand{\norm}[1]{\left\Vert#1\right\Vert}
\newcommand{\norms}[1]{\Vert#1\Vert}

\newcommand{\Eproof}{\hfill $\square$}

\newcommand{\zero}[1]{{\boldsymbol{0}}}

\newcommand{\pb}{p}

\newcommand{\Bc}{\mathcal{B}}

\newcommand{\Ub}{\mathbf{U}}

\newcommand{\Sc}{\mathcal{S}}

\newcommand{\Tc}{\mathcal{T}}

\newcommand{\Fc}{\mathcal{F}}

\newcommand{\iprods}[1]{\langle #1\rangle}

\newcommand{\Exp}[1]{\mathbb{E}\left[#1\right]}
\newcommand{\Exps}[2]{\mathbb{E}_{#1}\left[#2\right]}
\newcommand{\Probn}{\mathbb{P}}
\newcommand{\Prob}[1]{\mathbb{P}\left(#1\right)}

\newcommand{\BigO}[1]{\mathcal{O}\left(#1\right)}

\newcommand{\beforesec}{\vspace{-2.5ex}}
\newcommand{\aftersec}{\vspace{-2.25ex}}
\newcommand{\beforesubsec}{\vspace{-2ex}}
\newcommand{\aftersubsec}{\vspace{-1.5ex}}
\newcommand{\beforesubsubsec}{\vspace{-2.0ex}}
\newcommand{\aftersubsubsec}{\vspace{-1.5ex}}
\newcommand{\beforepara}{\vspace{-2.5ex}}

\usepackage{algorithm}
\usepackage{algpseudocode}
\renewcommand{\pb}{\mathbf{p}}

\newcommand{\nhan}[1]{{#1}}

\newcommand{\myeq}[2]{\vspace{-0.5ex}\begin{equation}\label{#1}#2\vspace{-0.5ex}\end{equation}}

\newcommand{\done}[1]{{#1}}

\usepackage{geometry}
\title{Hybrid Stochastic Gradient Descent Algorithms for Stochastic Nonconvex Optimization}

\author{%
 Quoc Tran-Dinh$^{\dagger}$, Nhan H. Pham$^{\dagger}$, Dzung T. Phan$^{\ddagger}$, \textit{and} Lam M. Nguyen$^{\ddagger}$
 \vspace{0.75ex}\\
 $^{\dagger}$Department of Statistics and Operations Research\\
The University of North Carolina at Chapel Hill, Chapel Hill, NC27599, USA.\\
\texttt{quoctd@email.unc.edu, ~nhanph@live.unc.edu} \\
\and
 $^{\ddagger}$IBM Research, Thomas J. Watson Research Center\\
Yorktown Heights, NY10598, USA.\\
\texttt{phandu@us.ibm.com, ~LamNguyen.MLTD@ibm.com}
 }

\begin{document}

\maketitle

\begin{abstract}
\vspace{-2ex}
We introduce a hybrid stochastic estimator to design stochastic gradient algorithms for solving stochastic optimization problems.
Such a hybrid estimator is a convex combination of two existing biased and unbiased estimators and leads to some useful property on its variance.
We limit our consideration to a hybrid SARAH-SGD for nonconvex expectation problems.
However, our idea can be extended to handle a broader class of estimators in both convex and nonconvex settings.
We propose a new single-loop stochastic gradient descent algorithm that can achieve $\BigO{\max\set{\sigma^3\varepsilon^{-1},\sigma\varepsilon^{-3}}}$-complexity bound to obtain an $\varepsilon$-stationary point under smoothness and $\sigma^2$-bounded variance assumptions. This complexity is better than $\BigO{\sigma^2\varepsilon^{-4}}$ often obtained in state-of-the-art SGDs when $\sigma < \BigO{\varepsilon^{-3}}$.
We also consider different extensions of our method, including constant and adaptive step-size with single-loop, double-loop, and mini-batch variants.
We compare our algorithms with existing methods on several datasets using two nonconvex models.
\end{abstract}

\beforesec
\section{Introduction}\label{sec:intro}
\aftersec
Consider the following stochastic nonconvex optimization problem of the form:
\begin{equation}\label{eq:ncvx_prob}
\min_{x\in\R^p}\Big\{ f(x) := \Exps{\xi}{f(x;\xi)} \Big\},
\end{equation}
where $f(\cdot;\cdot) : \R^p\times \Omega \to \R$ is a stochastic function defined such that for each $x\in\R^p$, $f(x;\cdot)$ is a random variable in a given probability space $(\Omega, \Probn)$, while for each realization $\xi\in\Omega$, $f(\cdot;\xi)$ is smooth on $\R^p$; and $\Exps{\xi}{f(x;\xi)}$ is the expectation of $f(x;\xi)$ w.r.t. $\xi$ over $\Omega$.

\beforepara
\paragraph{Our goals and assumptions:}
Since \eqref{eq:ncvx_prob} is nonconvex, our goal in this paper is to develop a new class of stochastic gradient algorithms to find an $\varepsilon$-approximate stationary point $\widetilde{x}_T$ of \eqref{eq:ncvx_prob} such that $\Exp{\norms{\nabla{f}(\widetilde{x}_T)}^2} \leq \varepsilon^2$ under mild assumptions as stated in Assumption~\ref{as:A1}.

\begin{assumption}\label{as:A1} The objective function $f$ of \eqref{eq:ncvx_prob} satisfies the following conditions:
\vspace{-1ex}
\begin{itemize}
\item[$\mathrm{(a)}$] (\textbf{Boundedness from below}) There exists a finite lower bound $f^{\star} := \inf_{x\in\R^p}f(x) > -\infty$.

\vspace{-1.5ex}
\item[$\mathrm{(b)}$] (\textbf{$L$-average smoothness})
The function $f(\cdot;\xi)$ is $L$-average smooth on $\R^p$, i.e. there exists $L\in (0, +\infty)$ such that
\begin{equation}\label{eq:L_smooth}
\Exps{\xi}{\norm{\nabla{f}(x;\xi) - \nabla{f}(y;\xi)}^2} \leq L^2\norms{x - y}^2,~~\forall x, y\in\R^p.
\vspace{-1.5ex}
\end{equation}
\item[$\mathrm{(c)}$] (\textbf{Bounded variance})
There exists $\sigma \in (0, \infty)$ such that 
\vspace{1ex}
\begin{equation}\label{eq:bounded_variance2}
\Exps{\xi}{\norms{\nabla f(x;\xi) - \nabla{f}(x)}^2} \leq \sigma^2, ~~~\forall x\in\R^p.
\end{equation}
\end{itemize}
\end{assumption}
%
These assumptions are very standard in stochastic optimization methods \cite{ghadimi2013stochastic,Nemirovski2009a}. 
The $L$-average smoothness of $f$ is weaker than the smoothness of $f$ for each realization $\xi\in\Omega$.
Note that our methods described in the sequel are also applicable to the finite-sum problem $\min_{x} \set{f(x) = \frac{1}{n}\sum_{i=1}^nf_i(x)}$ as long as the above assumptions hold.
However, we do not specify our methods to solve this problem.
In this case, $\sigma$ in \eqref{eq:bounded_variance2} can be replaced by other alternatives, e.g., $\sigma_n^2 := \frac{1}{n}\sum_{i=1}^n\left[\norms{\nabla{f}_i(x)}^2 - \norms{\nabla{f}(x)}^2\right]$. 

\beforepara
\paragraph{Our key idea:}
Different from existing methods, we introduce a convex combination of  a biased and unbiased estimator of the gradient $\nabla{f}$ of $f$, which we call a \textit{hybrid stochastic gradient} estimator. 
While the biased estimator exploited in this paper is SARAH in \cite{Nguyen2017_sarah}, the unbiased one can be any unbiased estimator. 
SARAH is a recursive biased and variance reduced estimator for $\nabla{f}$.
Combining it with  an unbiased estimator allows us to reduce the bias and variance of the hybrid estimator.
In this paper, we only focus on the standard stochastic estimator as an unbiased candidate.

\beforepara
\paragraph{Related work:}
Under Assumption~\ref{as:A1}, problem \eqref{eq:ncvx_prob} covers a large number of applications in machine learning and data sciences. 
The stochastic gradient descent (SGD) method was first studied in \cite{RM1951}, and becomes extremely popular in recent years.
\cite{Nemirovski2009a} seems to be the first work showing the convergence rates of robust SGD variants in the convex setting, while \cite{Nemirovskii1983} provides an intensive complexity analysis for many optimization algorithms, including stochastic methods.
Variance reduction methods have also been widely studied, see, e.g. \cite{allen2016katyusha,chambolle2017stochastic,SAGA,johnson2013accelerating,nitanda2014stochastic,reddi2016stochastic,schmidt2017minimizing,shalev2013stochastic,Xiao2014}. 
In the nonconvex setting, \cite{ghadimi2013stochastic} seems to be the first algorithm achieving $\BigO{\sigma^2\varepsilon^{-4}}$-complexity bound.
Other researchers have also made significant progress in this direction, including \cite{allen2017natasha2,allen2018neon2,SVRG++,fang2018spider,lei2017non,Nguyen2018_iSARAH,Nguyen2019_SARAH,Nguyen2017_sarahnonconvex,reddi2016proximal,wang2018spiderboost,zhou2018stochastic}.
A majority of these works, including \cite{allen2017natasha2,allen2018neon2,SVRG++,lei2017non,reddi2016proximal}, rely on SVRG estimator in order to obtain better complexity bounds.
Hitherto, the complexity of SVRG-based methods remains worse than the best-known results, which is obtained in \cite{fang2018spider,pham2019proxsarah,wang2018spiderboost} via the SARAH estimator. 
However, as discussed in  \cite{pham2019proxsarah,wang2018spiderboost}, the method called SPIDER in  \cite{fang2018spider,lei2017non} does not practically perform well due to small step-size and its dependence on the reciprocal of the estimator's norm.
\cite{wang2018spiderboost} amends this issue by using a large constant step-size, but requires large mini-batch and does not consider  the single sample case and single loop variants.
\cite{pham2019proxsarah} provides a more general framework to treat composite problems where it covers \eqref{eq:ncvx_prob} as special case, but it does not consider the single loop as in SGDs.

\beforepara
\paragraph{Our contribution:}
To this end, our contribution can be summarized as follows:
\begin{itemize}
\vspace{-1.5ex}
\item[$\mathrm{(a)}$] 
We propose a hybrid stochastic estimator for a stochastic gradient of a nonconvex function $f$ in \eqref{eq:ncvx_prob} by combining the SARAH estimator from \cite{Nguyen2017_sarah} and any unbiased stochastic estimator such as SGD and SVRG. 
However, we only focus on the SGD estimator in this paper.
We prove some key properties of this hybrid estimator that can be used to design new algorithms.

\vspace{-0.75ex}
\item[$\mathrm{(b)}$] 
We exploit our hybrid estimator to develop a single-loop SGD algorithm that can achieve an $\varepsilon$-stationary point $\widetilde{x}_m$ such that $\Exp{\norms{\nabla{f}(\widetilde{x}_m)}^2} \leq \varepsilon^2$ in at most $\BigO{\sigma\varepsilon^{-3} + \sigma^3\varepsilon^{-1}}$ stochastic gradient evaluations.
This complexity significantly improves $\BigO{\sigma^2\varepsilon^{-4}}$ of SGD if $\sigma < \BigO{\varepsilon^{-3}}$. 
We extend our algorithm to a double loop variant, which requires $\BigO{\max\set{ \sigma\varepsilon^{-3}, \sigma^2\varepsilon^{-2}}}$ stochastic gradient evaluations.
This is the best-known complexity in the literature for stochastic gradient-type methods for solving \eqref{eq:ncvx_prob}.

\vspace{-0.75ex}
\item[$\mathrm{(c)}$] 
We also investigate other variants of our method, including adaptive step-sizes, and mini-batches.
In all these cases, our methods achieve the best-known complexity bounds.

\vspace{-1.5ex}
\end{itemize}
Let us emphasize the following points of our contribution.
Firstly, although our single-loop method requires three gradients per iteration compared to standard SGDs, it can achieves better complexity bound.
Secondly, it can be cast into a variance reduction method where it starts from a ``good'' approximation $v_0$ of $\nabla{f}(x^0)$, and aggressively reduces the variance.
Thirdly, our step-size is $\eta = \BigO{m^{-1/3}}$ which is larger than $\eta = \BigO{m^{-1/2}}$ in SGDs.
Fourthly, the step-size of the adaptive variant is increasing instead of diminishing as in SGDs.
Finally, our method achieves the same best-known complexity as in variance reduction methods studied in \cite{fang2018spider,pham2019proxsarah,wang2018spiderboost}.
We believe that our approach can be extended to other estimators such as SVRG \cite{johnson2013accelerating} and SAGA \cite{SAGA}, and can be used for Hessians to develop second-order methods as well as to solve convex and composite problems.

\beforepara
\paragraph{Paper organization:}
The rest of this paper is organized as follows.
Section~\ref{sec:stochastic_estimators} introduces our new hybrid stochastic estimator for the gradient of $f$ and investigates its properties.
Section~\ref{sec:algorithms} proposes a single-loop hybrid SGD-SARAH algorithm and its complexity analysis. 
It also considers a double-loop and mini-batch variants with rigorous complexity analysis.
Section~\ref{sec:num_experiments} provides two numerical examples to illustrate our methods and compares them with state-of-the-art methods.
All the proofs and additional experiments can be found in the Supplementary Document.

\textbf{Notation:} We work with Euclidean spaces, $\R^p$, equipped with standard inner product $\iprods{\cdot,\cdot}$ and norm $\norm{\cdot}$.
For a smooth function $f$ (i.e., $f$ is continuously differentiable), $\nabla{f}$ denotes its gradient.
We use $\Ub_{\pb}(\Sc)$ to denote a distribution on $\Sc$ with probability $\pb$.
If $\pb$ is uniform, then we simply use $\Ub(\Sc)$.
We also use $\BigO{\cdot}$ to present big-O notion in complexity theory, and $\sigma(\cdot)$ to denote a $\sigma$-field.

\beforesec
\section{Hybrid stochastic gradient estimators}\label{sec:stochastic_estimators}
\aftersec
%
In this section, we propose new stochastic estimators for the gradient of a smooth function $f$.

Let $u_t$ be an unbiased estimator of $\nabla{f}(x_t)$ formed by a realization $\zeta_t$ of $\xi$, i.e. $\Exps{\zeta_t}{u_t} = \nabla{f}(x_t)$.
We attempt to develop the following stochastic estimator for $\nabla{f}(x_t)$ in \eqref{eq:ncvx_prob}:
\begin{equation}\label{eq:v_estimator}
v_t :=  \beta_{t-1}v_{t-1} + \beta_{t-1}(\nabla{f}(x_t;\xi_t) - \nabla{f}(x_{t-1};\xi_t)) + (1-\beta_{t-1})u_t,
\end{equation}
where $\xi_t$ and $\zeta_t$ are two independent realizations of $\xi$ on $\Omega$.
Clearly, if $\beta_t = 0$, then we obtain a simple unbiased stochastic estimator, and $\beta_t = 1$, we obtain the SARAH estimator in \cite{Nguyen2017_sarah}.
We are interested in the case $\beta_t \in (0, 1)$, in which we call $v_t$ in \eqref{eq:v_estimator} a \textbf{hybrid stochastic} estimator.

Note that  we can rewrite $v_t$ as
\begin{equation*} 
v_t :=   \beta_{t-1}\nabla{f}(x_t;\xi_t) + (1-\beta_{t-1})u_t + \beta_{t-1}(v_{t-1} - \nabla{f}(x_{t-1};\xi_t)).
\end{equation*}
The first two terms are two stochastic gradients estimated at $x_t$, while the third term is the difference $v_{t-1} - \nabla{f}(x_{t-1};\xi_t)$ of the previous estimator and a stochastic gradient at the previous iterate.
Here, since $\beta_{t-1} \in (0, 1)$, the main idea is to exploit more recent information than the old ones.
In fact, the hybrid estimator $v_t$ covers many other estimators, including SGD, SVRG, and SARAH.
We can use one of the following two concrete unbiased estimators $u_t$ of $\nabla{f}(x_t)$ as follows:
\begin{itemize}
\vspace{-1.5ex}
\item \textbf{The SGD estimator:} $u_t := u_t^{\mathrm{sgd}} = \nabla{f}(x_t;\zeta_t)$. 
\item \textbf{The SVRG estimator:} $u_t := u_t^{\mathrm{svrg}} = \nabla{f}(\tilde{x}) + \nabla{f}(x_t;\zeta_t) - \nabla{f}(\tilde{x};\zeta_t)$, where $\nabla{f}(\tilde{x})$ is a full gradient evaluated at a given snapshot point $\tilde{x}$.
\vspace{-1.5ex}
\end{itemize}
However, for the sake of presentation, we only focus on the SGD estimator $u_t :=  u_t^{\mathrm{sgd}}$.

We first prove the following property of  the estimator $v_t$ showing how the variance is estimated.

\begin{lemma}\label{le:key_estimate10}
Let $v_t$ be defined by \eqref{eq:v_estimator}.
Then
\begin{equation}\label{eq:biased_estimator}
\Exps{(\xi_t,\zeta_t)}{v_t} = \nabla{f}(x_t) + \beta_{t-1}(v_{t-1} - \nabla{f}(x_{t-1})).
\end{equation}
If $\beta_{t-1} \neq 0$, then $v_t$ is a biased estimator.
Moreover, we have
\begin{equation}\label{eq:key_estimate10} 
\begin{array}{ll}
\Exps{(\xi_t,\zeta_t)}{\norms{v_t - \nabla{f}(x_t)}^2} &=  \beta_{t-1}^2\norms{v_{t-1}  - \nabla{f}(x_{t-1})}^2 - \beta_{t-1}^2\norms{\nabla{f}(x_{t-1}) - \nabla{f}(x_t)}^2 \vspace{1ex}\\
& + {~} \beta_{t-1}^2\Exps{\xi_t}{\norms{\nabla{f}(x_t;\xi_t) - \nabla{f}(x_{t-1};\xi_t)}^2} \vspace{1ex}\\
& + {~} (1-\beta_{t-1})^2\Exps{\zeta_t}{\norms{u_t - \nabla{f}(x_t)}^2}.
\end{array}
\end{equation}
\end{lemma}

\begin{remark}
From \eqref{eq:v_estimator}, we can see that $v_t$ remains a biased estimator as long as $\beta_{t-1} \in (0, 1]$. 
Its biased term is 
\begin{equation*}
\mathrm{Bias}(v_t) = \norms{\Exps{(\xi_t,\zeta_t)}{v_t - \nabla{f}(x_t) \mid \Fc_t}} = \beta_{t-1}\norms{v_{t-1} - \nabla{f}(x_{t-1})} \leq \norms{v_{t-1} - \nabla{f}(x_{t-1})}.
\end{equation*}
%
This shows that the bias $v_t$ estimator is  smaller than the one in the SARAH estimator $v_t^{\textrm{sarah}} := v_{t-1}^{\textrm{sarah}} + \nabla{f}(x_t;\xi_t) - \nabla{f}(x_{t-1};\xi_t)$ from \cite{Nguyen2017_sarah}, which is $\mathrm{Bias}(v^{\mathrm{sarah}}_t) = \norms{v_{t-1}^{\mathrm{sarah}} - \nabla{f}(x_{t-1})}$.
\end{remark}
The following lemma bounds the second moment of $v_t - \nabla{f}(x_t)$ with $v_t$ defined in \eqref{eq:v_estimator}.

\begin{lemma}\label{le:upper_bound_new}
Assume that $f(\cdot,\cdot)$ is $L$-smooth and $u_t$ is an SGD estimator.
Then, we have the following upper bound on the variance $\Exp{\norms{v_t - \nabla{f}(x_t)}^2}$ of $v_t$:
\begin{equation}\label{eq:vt_variance_bound_new}
\Exp{\norms{v_t - \nabla{f}(x_t)}^2} \leq \omega_t\Exp{\norms{v_0 - \nabla{f}(x^0)}^2} + L^2\sum_{i=0}^{t-1}\omega_{i,t}\Exp{\norms{x_{i+1} - x_{i}}^2} + S_t,
\end{equation}
where the expectation is taking over all the randomness $\Fc_t := \sigma(v_0, v_1, \cdots, v_t)$, 
$\omega_{t} := \prod_{i=1}^{t}\beta_{i-1}^2$, $\omega_{i, t} := \prod_{j=i+1}^{t}\beta_{j-1}^2$ for $i=0,\cdots, t$, and $S_{t} := \sum_{i=0}^{t-1}\big(\prod_{j=i+2}^{t}\beta_{j-1}^2\big)(1-\beta_i)^2\sigma_{i+1}^2$ for $t \geq 0$.
\end{lemma}
Lemmas~\ref{le:key_estimate10} and \ref{le:upper_bound_new} provides two key properties to develop stochastic algorithm in Section~\ref{sec:algorithms}.

\beforesec
\section{Hybrid SARAH-SGD algorithms}\label{sec:algorithms}
\aftersec
In this section, we utilize our hybrid stochastic estimator $v_t$ in \eqref{eq:v_estimator} to develop stochastic gradient methods for solving \eqref{eq:ncvx_prob}. 
We consider three different variants using the hybrid SARAH-SGD estimator.

\beforesubsec
\subsection{The generic algorithm framework}
\aftersubsec
Using  $v_t$ defined by \eqref{eq:v_estimator}, we can develop a new algorithm for solving \eqref{eq:ncvx_prob} as in Algorithm~\ref{alg:A1}.

\begin{algorithm}[ht!]\caption{(Hybrid  stochastic gradient descent (Hybrid-SGD) algorithm)}\label{alg:A1}
\normalsize
\begin{algorithmic}[1]
   \State{\bfseries Initialization:} An initial point $x^0$ and parameters $b$, $\beta_t$, and $\eta_t$ (will be specified).
   \vspace{0.5ex}
   \State\hspace{0ex}\label{step:o2} Generate an unbiased estimator $v_0 := \frac{1}{b}\sum_{\hat{\xi}_i\in\Bc}\nabla{f}(x_0;\hat{\xi}_i)$ at $x_0$ using a mini-batch $\Bc$.
   \vspace{0.5ex}   
   \State\hspace{0ex}\label{step:o3} Update $x_1 := x_0 - \eta_0v_0$.
   \vspace{0.5ex}   
   \State\hspace{0ex}\label{step:o4}{\bfseries For $t := 1,\cdots,m$ do}
   \vspace{0.5ex}   
   \State\hspace{3ex}\label{step:i1} Generate a proper sample pair $(\xi_t, \zeta_t)$ independently (single sample or mini-batch).
   \vspace{0.5ex}   
   \State\hspace{3ex}\label{step:i2} Evaluate $v_{t}  := \beta_{t-1}v_{t-1}  + \beta_{t-1}\big(\nabla{f}(x_{t};\xi_t) - \nabla{f}(x_{t-1};\xi_t)\big) + (1-\beta_{t-1})\nabla{f}(x_{t}; \zeta_t)$.
   \State\hspace{3ex}\label{step:i3} Update $x_{t+1} :=x_{t} - \eta_t v_{t}$.
   \vspace{0.5ex}   
   \State\hspace{0ex}{\bfseries EndFor}
   \vspace{0.5ex}   
   \State\hspace{0ex}\label{step:o5} Choose $\widetilde{x}_m$ from $\set{x_0, x_1, \cdots, x_m}$ (at random or deterministic, specified later). 
\end{algorithmic}
\end{algorithm}
\vspace{-1ex}

Algorithm~\ref{alg:A1} looks essentially the same as any SGD scheme with only one loop.
The differences are at Step~\ref{step:o2} with a mini-batch estimator $v_0$ and at Step~\ref{step:i2}, where we use our hybrid gradient estimator $v_t$.
In addition, we will show in the sequel that it uses different step-sizes and leads to different variants.
Unlike the inner loop of SARAH or SVRG, each iteration of Algorithm~\ref{alg:A1} requires three individual gradient evaluations instead of two as in these methods.
The snapshot at Step~\ref{step:o2} of Algorithm~\ref{alg:A1} relies on a mini-batch $\Bc$ of the size $b$, which is independent of $(\xi_t,\zeta_t)$ in the loop $t$.

\beforesubsec
\subsection{Convergence analysis}
\aftersubsec
We analyze two cases: constant step-size and adaptive step-size. In both cases, $\beta_t$ is fixed for all $t$.

\beforesubsubsec
\subsubsection{Convergence of Algorithm~\ref{alg:A1} with constant step-size $\eta$ and constant $\beta$}
\aftersubsubsec
Assume that we run Algorithm~\ref{alg:A1} within $m$ iterations $m\geq 1$.
In this case, given $0 < c_1 < \sqrt{b(m+1)}$, we choose $\eta$ and $\beta$ in Algorithm~\ref{alg:A1} as follows:
\begin{equation}\label{eq:step_size1}
\eta := \frac{2}{L(\sqrt{1 + 4\alpha_m^2} + 1)}~~~~~~\text{with}~~~~~\beta := 1 - \frac{c_1}{\sqrt{b(m+1)}}~~~~\text{and}~~~~\alpha_m^2 :=  \frac{\beta^2(1-\beta^{2m})}{1-\beta^2}.
\end{equation}
The following theorem estimates the complexity of Algorithm~\ref{alg:A1} to approximate an $\varepsilon$-stationary point of \eqref{eq:ncvx_prob}, whose proof is given in Subsection~\ref{apdx:th:singe_loop_const_step} of the supplementary document.

\begin{theorem}\label{th:singe_loop_const_step}
Let $\sets{x_t}$ be the sequence generated by Algorithm~\ref{alg:A1} using the step-size $\eta$ defined by \eqref{eq:step_size1}.
Let us choose $\widetilde{x}_m\sim\Ub(\sets{x_t}_{t=0}^m)$.
Then
\begin{itemize}
\vspace{-1ex}
\item[$(\mathrm{a})$] The step-size $\eta$ satisfies $\eta \geq \underline{\eta} :=  \frac{2\sqrt{c_1}}{3\nhan{L}\big[ b(m+1)\big]^{1/4}}$. 
In addition, we have
\begin{equation}\label{eq:single_SGD} 
\Exp{\norms{\nabla{f}(\widetilde{x}_m)}^2} \leq \frac{3b^{1/4}L\big[f(x^0) - f^{\star}\big]}{\sqrt{c_1}(m+1)^{3/4}} + \left(c_1 + \frac{1}{c_1}\right)\frac{ \sigma^2}{\sqrt{b(m+1)}}.
\end{equation}
\item[$(\mathrm{b})$] If we choose  $b :=  c_2\sigma^{8/3} (m + 1)^{1/3}$ for any $c_2 > 0$, then to guarantee $\Exp{\norms{\nabla{f}(\widetilde{x}_m)}^2} \leq \varepsilon^2$, we need to choose 
\begin{equation}\label{eq:choice_of_m}
m := \bigg\lfloor\tfrac{\sigma}{\varepsilon^3}\Big[\tfrac{3 L c_2^{1/4}}{\sqrt{c_1}}\big[f(x^0) - f^{\star}\big] + \left(c_1 + \tfrac{1}{c_1}\right)\tfrac{1}{\sqrt{c_2}}\Big]^{3/2}\bigg\rfloor = \BigO{\frac{\sigma}{\varepsilon^3}}.
\end{equation}
In particular, if we choose $c_1 = 1$, then the number of oracle calls is $\Tc_{ge}$ is
\begin{equation}\label{eq:overall_complexity2}
{\!\!\!\!\!\!\!\!}\begin{array}{ll}
\Tc_{ge} &:=  \dfrac{\sigma^3}{\varepsilon}\left[ 3 L c_2^{9/4}\big[f(x^0) - f^{\star}\big] + 2c_2^{3/2}\right]^{1/2} {\!\!\!} +  \dfrac{3\sigma}{\varepsilon^3}\left[ 3 L c_2^{1/4}\big[f(x^0) - f^{\star}\big] + \frac{2}{\sqrt{c_2}}\right]^{3/2} \vspace{1ex}\\
&= \BigO{\dfrac{\sigma^3}{\varepsilon}+ \dfrac{\sigma}{\varepsilon^3}}.
\end{array}{\!\!\!\!\!\!\!\!\!}
\end{equation}
Moreover, the step-size $\eta$ satisfies $\eta \geq \underline{\eta} := \frac{2}{3\nhan{L}c_2^{1/4}\sigma^{2/3}(m+1)^{1/3}} = \BigO{m^{-1/3}}$.
\end{itemize}
\end{theorem}
%
%
Here, $\Tc_{ge}$ stands for the number of stochastic gradient evaluations of $f$ in \eqref{eq:ncvx_prob}.
The complexity $\Tc_{ge}$ in \eqref{eq:overall_complexity2} can be written as $\Tc_{ge} = \BigO{\max\set{\sigma\varepsilon^{-3}, \sigma^3\varepsilon^{-1}}}$.
If $\sigma < \BigO{\frac{1}{\varepsilon}}$, then our complexity is $\Tc_{ge} = \BigO{\sigma\varepsilon^{-3}}$.
Even if $\sigma < \BigO{\frac{1}{\varepsilon^3}}$, then our complexity is still better than $\BigO{\sigma^2\varepsilon^{-4}}$ in SGD.

\beforesubsubsec
\subsubsection{Convergence of Algorithm~\ref{alg:A1} with adaptive step-size $\eta_t$ and constant $\beta$}
\aftersubsubsec
Let $\beta := 1 - \frac{c_1}{\sqrt{b(m+1)}} \in (0, 1)$ be fixed for some $0 < c_1 < \sqrt{b(m+1)}$.
Instead of fixing step-size $\eta_t$ as in \eqref{eq:step_size1}, we can update it adaptively as 
\begin{equation}\label{eq:update_of_eta_t}
\eta_m := \frac{1}{L},~~\text{and}~~\eta_t := \frac{1}{L + L^2\big[\beta^2\eta_{t+1} + \beta^4\eta_{t+2} + \cdots + \beta^{2(m-t)}\eta_m\big]}~~\text{for}~t=0,\cdots, m-1.
\end{equation}
It can be shown that $0 < \eta_0 < \eta_1 < \cdots < \eta_m$.
Interestingly, our step-size is updated in an increasing manner instead of diminishing as in existing SGD-type methods.
Moreover, given $m$, we can pre-compute the sequence of these step-sizes $\set{\eta_t}_{t=0}^m$ in advance within $\BigO{m}$ basic operations.
Therefore, it does not significantly incur the computational cost of our method.

The following theorem states the convergence of Algorithm~\ref{alg:A1} under the adaptive update \eqref{eq:update_of_eta_t}, whose proof is given in Subsection~\ref{apdx:th:singe_loop_adapt_step} of the supplementary document.

\begin{theorem}\label{th:singe_loop_adapt_step}
Let $\sets{x_t}$ be the sequence generated by Algorithm~\ref{alg:A1} using the step-size $\eta_t$ defined by \eqref{eq:update_of_eta_t}.
Let $\Sigma_m := \sum_{t=0}^m\eta_t$, and $\widetilde{x}_m\sim\Ub_{\pb}(\sets{x_t}_{t=0}^m)$ with $\pb_t := \Prob{\widetilde{x}_m = x_t} = \frac{\eta_t}{\Sigma_m}$.
Then
\begin{itemize}
\vspace{-1ex}
\item[$(\mathrm{a})$] The sum $\Sigma_m$ is bounded from below as $\Sigma_m \geq  \frac{\sqrt{c_1}(m+1)^{3/4}}{2L b^{1/4}}$.
\vspace{-1ex}
\item[$(\mathrm{b})$] If we choose  $b :=  c_2\sigma^{8/3} (m + 1)^{1/3}$ for any $c_2 > 0$, then to guarantee $\Exp{\norms{\nabla{f}(\widetilde{x}_m)}^2} \leq \varepsilon^2$, we need to choose $m := \Big\lfloor\frac{\sigma}{\varepsilon^3}\Big[\frac{3 L c_2^{1/4}}{\sqrt{c_1}}\big[f(x^0) - f^{\star}\big] + \big(c_1 + \frac{1}{c_1}\big)\frac{1}{\sqrt{c_2}}\Big]^{3/2}\Big\rfloor = \BigO{\frac{\sigma}{\varepsilon^3}}$.
Therefore, the number of stochastic gradient evaluations $\Tc_{ge}$ is at most the same as in \eqref{eq:overall_complexity2}.
\end{itemize}
\end{theorem}
%
%
%
Note that in the finite sum case, i.e. $\vert\Omega\vert = n$, we set $b := \min \{n, c_2\sigma^{8/3} (m + 1)^{1/3}\}$ in both Theorems~\ref{th:singe_loop_const_step} and \ref{th:singe_loop_adapt_step}.
This complexity remains the same as in Theorem~\ref{th:singe_loop_const_step}.
However, the adaptive stepsize $\eta_t$ potentially gives a better performance in practice as we will see in Section~\ref{sec:num_experiments}.

Algorithm~\ref{alg:A1} can be considered as a single-loop variance reduction method, which is similar to SAGA \cite{SAGA}, but Algorithm~\ref{alg:A1} aims at solving the nonconvex problem \eqref{eq:ncvx_prob}.
It is different from standard SGD methods, where it can be initialized by a mini-batch and then update the estimator using three individual gradients.
Therefore, it has the same cost as SGD with mini-batch of size $3$.
As a compensation, we obtain an improvement on the complexity bound as in Theorems~\ref{th:singe_loop_const_step} and \ref{th:singe_loop_adapt_step}.

\beforesubsec
\subsection{Convergence analysis of the double loop variant}
\aftersubsec
Since the step-size $\eta_t$ depends on $m$, it is natural to run Algorithm~\ref{alg:A1} with multiple stages.
This leads to a double-loop algorithm as SVRG, SARAH, and SPIDER, where Algorithm~\ref{alg:A1} is restarted at each outer iteration $s$.
The detail of this variant is described in Algorithm~\ref{alg:A2}.

\vspace{-1ex}
\begin{algorithm}[ht!]\caption{(Double-loop HSGD algorithm)}\label{alg:A2}
\normalsize
\begin{algorithmic}[1]
   \State{\bfseries Initialization:} An initial point $\widetilde{x}^0$  and parameters $b$, $m$, $\beta_t$, and $\eta_t$ (will be specified).
   \vspace{0.5ex}
   \State{\bfseries OuterLoop:}~{\bfseries For $s := 1, 2, \cdots, S$ do}
   \vspace{0.5ex}   
   \State\hspace{3ex}\label{step:o2} Run Algorithm~\ref{alg:A1} with an initial point $x_0^{(s)} := \widetilde{x}^{(s-1)}$.
   \vspace{0.5ex}   
   \State\hspace{3ex}\label{step:o5} Set $\widetilde{x}^{(s)} := x_{m+1}^{(s)}$ as the last iterate of Algorithm~\ref{alg:A1}. 
   \vspace{0.5ex}   
   \State{\bfseries EndFor}
\end{algorithmic}
\end{algorithm}
\vspace{-1ex}

To analyze Algorithm~\ref{alg:A2}, we use $x^{(s)}_t$ to represent the iterate of Algorithm~\ref{alg:A1} at  the $t$-th inner iteration within each stage $s$.
From  \eqref{eq:single_loop2}, we can see that each stage $s$, the following estimate holds
\begin{equation*} 
\frac{\eta}{2}\displaystyle\sum_{t=0}^m\Exp{\norms{\nabla{f}(x_t^{(s)})}^2} \leq  \Exp{f(x_0^{(s)})} - \Exp{f(x_{m+1}^{(s)})}  + \frac{\eta\sigma^2\sqrt{m+1}}{(1+\beta)\sqrt{b}}.
\end{equation*}
Here, we assume that we fix the step-size $\eta_t = \eta > 0$ for simplicity of analysis.
The complexity of Algorithm~\ref{alg:A2} is given in the following theorem, whose proof is in Supplementary Document~\ref{apdx:th:double_loop_convergence}.

\begin{theorem}\label{th:double_loop_convergence}
Let $\sets{x^{(s)}_t}_{t=0\to m}^{s=1\to S}$ be the sequence generated by Algorithm~\ref{alg:A2} using constant step-size $\eta$ in \eqref{eq:step_size1}.
Then, the following estimate holds
\begin{equation}\label{eq:double_loop_est} 
\frac{1}{S(m+1)}\displaystyle\sum_{s=1}^S\sum_{t=0}^m\Exp{\norms{\nabla{f}(x_t^{(s)})}^2} \leq  \frac{3 L b^{1/4}}{S (m+1)^{3/4}} \big[f(\widetilde{x}^0) - f^{\star}\big] + \frac{2 \sigma^2}{\sqrt{b (m+1)}}.
\end{equation}
Let $\widetilde{x}_T \sim \Ub(\sets{x^{(s)}_t}_{t=0\to m}^{s=1\to S})$.
If we choose $b := \frac{c_1\sigma^2}{\varepsilon^2}$ and $m + 1 :=  \frac{c_2\sigma^2}{\varepsilon^2}$ for some constants $c_1 > 0$ and $c_2 > 0$ and $c_1c_2 > 4$, then, to guarantee $\Exp{\norms{\nabla{f}(\widetilde{x}_T)}^2} \leq \varepsilon^2$, we require at most
\begin{equation}\label{eq:S_iterations}
S := \Bigg\lfloor \frac{3 L c_1^{1/4} \big[f(\widetilde{x}^0) - f^{\star}\big]}{c_2^{3/4} \sigma \left( 1 - \frac{2}{\sqrt{c_1 c_2}} \right) \varepsilon}\Bigg\rfloor ~~~\text{outer iterations}.
\end{equation}
Consequently,  the total number of stochastic gradient evaluations $\Tc_{ge}$ does not exceed 
\begin{equation}\label{eq:Toc_double_loop}
\Tc_{ge} := (b + 3m)S  = \frac{3 L (c_1 + 3 c_2)c_1^{1/4} \big[f(\widetilde{x}^0) - f^{\star}\big]\sigma}{c_2^{3/4}\left( 1 - \frac{2}{\sqrt{c_1 c_2}} \right) \varepsilon^3}  = \BigO{\frac{\sigma}{\varepsilon^3}}. 
\end{equation}
\end{theorem}

%
Note that the complexity \eqref{eq:Toc_double_loop} only holds if $\BigO{\frac{\sigma}{\varepsilon^3}} > \frac{c_1\sigma^2}{\varepsilon^2}$.
Otherwise, the total complexity is $\BigO{\max\set{\frac{\sigma}{\varepsilon^3}, \frac{\sigma^2}{\varepsilon^2}}}$, where other constants independent of $\sigma$ and $\varepsilon$, and are hidden.
Practically, if $\beta$ is very close to $1$, one can remove the unbiased SGD term to save one stochastic gradient evaluation. 
In this case, our estimator reduces to SARAH but using different step-size.
We observed empirically that when $\beta\approx 0.999$, the performance of our methods is not affected if we do so.

\beforesubsec
\subsection{Extensions to  mini-batch cases}\label{subsec:extensions}
\aftersubsec
We consider a mini-batch hybrid stochastic estimator $\hat{v}_t$ for the gradient $\nabla{f}(x_t)$ defined as:
\begin{equation}\label{eq:vhat_t}
\hat{v}_t := \beta_{t-1}\hat{v}_{t-1} + \frac{\beta_{t-1}}{\hat{b}_t}\sum_{i\in\hat{\Bc}_t}(\nabla{f}(x_t;\xi_i) - \nabla{f}(x_{t-1};\xi_i)) +  (1-\beta_{t-1})u_t,
\end{equation}
where $\beta_{t-1}\in [0, 1]$, and $\hat{\Bc}_t$ is a mini-batch of the size $\hat{b}_t$ and independent of the unbiased estimator $u_t$.
Note that $u_t$ can also be a mini-batch unbiased estimator.
For example, $u_t := \frac{1}{\tilde{b}_t}\sum_{j\in\tilde{\Bc}_t}\nabla{f}(x_t;\zeta_j)$ is a mini-batch SGD estimator with a mini-batch $\tilde{\Bc}_t$ of size $\tilde{b}_t$, where $\tilde{\Bc}_t$ is independent of $\hat{\Bc}_t$.

Using $\hat{v}_t$ defined by \eqref{eq:vhat_t}, we can design a mini-batch variant of Algorithms~\ref{alg:A1} to solve \eqref{eq:ncvx_prob}.
The following corollary is obtained as a result of Theorems~\ref{th:singe_loop_const_step} for the mini-batch variant  of Algorithm~\ref{alg:A1}, whose proof is in Subsection~\ref{apdx:co:mini_batch} of the supplementary document.

\begin{corollary}\label{co:mini_batch}
Let Algorithm~\ref{alg:A1} be applied to solve \eqref{eq:ncvx_prob} using mini-batch update \eqref{eq:vhat_t} for $v_t$ with $\hat{b}_t = \tilde{b}_t = \hat{b} \geq 1$ fixed, $0 < c_1 < \sqrt{b(m+1)}$, and the step-size 
\begin{equation}\label{eq:step_size_batch}
\eta := \frac{2}{L\left(1 + \sqrt{1 + 4\rho\alpha_m^2}\right)}~~~~\text{with}~~~~\alpha^2_m := \frac{\beta^2(1-\beta^{2m})}{1 - \beta^2}~~~\text{and}~~~\beta := 1 - \frac{c_1}{\sqrt{\rho b (m-1)}}.
\end{equation}
If we choose  $b := c_2\sigma^{8/3}\left[\rho(m+1)\right]^{1/3}$ for any $c_2 > 0$, then to guarantee $\Exp{\norms{\nabla{f}(\widetilde{x}_m)}^2} \leq \varepsilon^2$, we need to choose 
\begin{equation}\label{eq:choice_of_m2}
m := \left\lfloor \frac{\rho^{1/2}\sigma}{\varepsilon^3}\left[\frac{3 Lc_2^{1/4}}{2{c_1}}\left( f(x^0) - f^{\star}\right) +  \left(c_1 + \frac{1}{c_1}\right)\frac{1}{2\sqrt{c_2}}\right]^{3/2}\right\rfloor = \BigO{\frac{\rho^{1/2}\sigma}{\varepsilon^3}}.
\end{equation}
Therefore, the number of oracle calls is $\Tc_{ge}$ is
\begin{equation}\label{eq:T_ge3}
\Tc_{ge}  := \BigO{\frac{\rho^{1/2}\sigma^3}{\varepsilon} + \frac{\sigma}{\rho^{1/2}\varepsilon^3}},
\end{equation}
where $\rho = \rho(\hat{b}) := \frac{n - \hat{b}}{(n-1)\hat{b}}$ if $n := \vert\Omega\vert$ is finite, and $\rho(\hat{b}) := \frac{1}{\hat{b}}$, otherwise.
In particular, if we choose $\hat{b} := \frac{\varepsilon^2\sigma^2}{c_3^2}$ for some $0 < c_3 \leq \varepsilon\sigma$, then, the overall complexity $\Tc_{ge}$ is $\Tc_{ge} :=  \BigO{\left(c_3 + \frac{1}{c_3}\right)\frac{\sigma^2}{\varepsilon^2}}$.
\end{corollary}

We can also develop a mini-batch variant of Algorithm~\ref{alg:A2} and estimate its complexity as in Theorem~\ref{th:double_loop_convergence}.
For more details, we refer to Subsection~\ref{apdx:th:double_loop_convergence_mini_batch} in the supplementary document due to space limit.


\beforesec
\section{Numerical experiments}\label{sec:num_experiments}
\aftersec
We verify our algorithms on two numerical examples and compare them with several existing methods: SVRG \cite{nonconvexSVRG}, SVRG+ \cite{li2018simple}, SPIDER \cite{fang2018spider}, SpiderBoost \cite{wang2018spiderboost}, and SGD \cite{ghadimi2013stochastic}. 
Due to space limit, the detailed configuration of our experiments as well as more numerical experiments can be found in Supplementary Document~\ref{apdx:subsec:experiments}.
Our numerical experiments are implemented in Python and running on a MacBook Pro. Laptop with 2.7GHz Intel Core i5 and 16Gb memory.

\beforesubsec
\subsection{Logistic regression with nonconvex regularizer}\label{subsec:exam1}
\aftersubsec
Our first example is the following well-known problem used in may papers including \cite{wang2018spiderboost}:
\myeq{eq:exam1}{
\min_{x\in\R^p}\bigg\{ f(x) := \frac{1}{n}\sum_{i=1}^n \Big[ f_i(x) := \log(1 + \exp(-a_i^Tx)) + \lambda\sum_{j=1}^p\frac{x_i^2}{1 + x_i^2}\Big] \bigg\},
}
where $a_i\in\R^p$ are given for $i=1,\cdots, n$, and $\lambda > 0$ is a regularization parameter. 
Clearly, problem \eqref{eq:exam1} fits \eqref{eq:ncvx_prob} well with $L_{f_i} = \frac{\Vert A\Vert^2}{4} + 2\lambda$.
In this experiment, we choose $\lambda = 0.1$ and normalize the data.
One can also verify Assumption~\ref{as:A1} due to the bounded Hessian of $f_i$.

We use three datasets from LibSVM for \eqref{eq:exam1}: \texttt{w8a} ($n=49,749, p=300$), \texttt{rcv1.binary} ($n=20,242, p = 47,236$), and \texttt{real-sim} $(n=72,309, p = 20,958$).
We run $8$ different algorithms as follows.
Algorithm~\ref{alg:A1} with constant step-size (Hybrid-SGD-SL) and adaptive step-size (Hybrid-SGD-ASL) using our theoretical step-sizes \eqref{eq:step_size1} and \eqref{eq:update_of_eta_t}, respectively without tuning.
Hybrid-SGD-DL is Algorithm~\ref{alg:A2}. 
SGD1 is SGD with constant step-size $\eta_t = \frac{0.1}{L}$, and SGD2 is SGD with adaptive step-size $\eta_t = \frac{0.1}{L(1 + \lfloor t/n\rfloor)}$.
Since the stepsize of SPIDER depends on an accuracy $\varepsilon$, we choose $\varepsilon = 10^{-1}$ to get a larger step-size.
Our first result in the single-sample case (i.e. when $\hat{b} = 1$, not using mini-batch) is plotted in Fig.~\ref{fig:logistic_reg} after $20$ epochs.


\begin{figure}[htp!]
\begin{center}
\includegraphics[width = 0.9\textwidth]{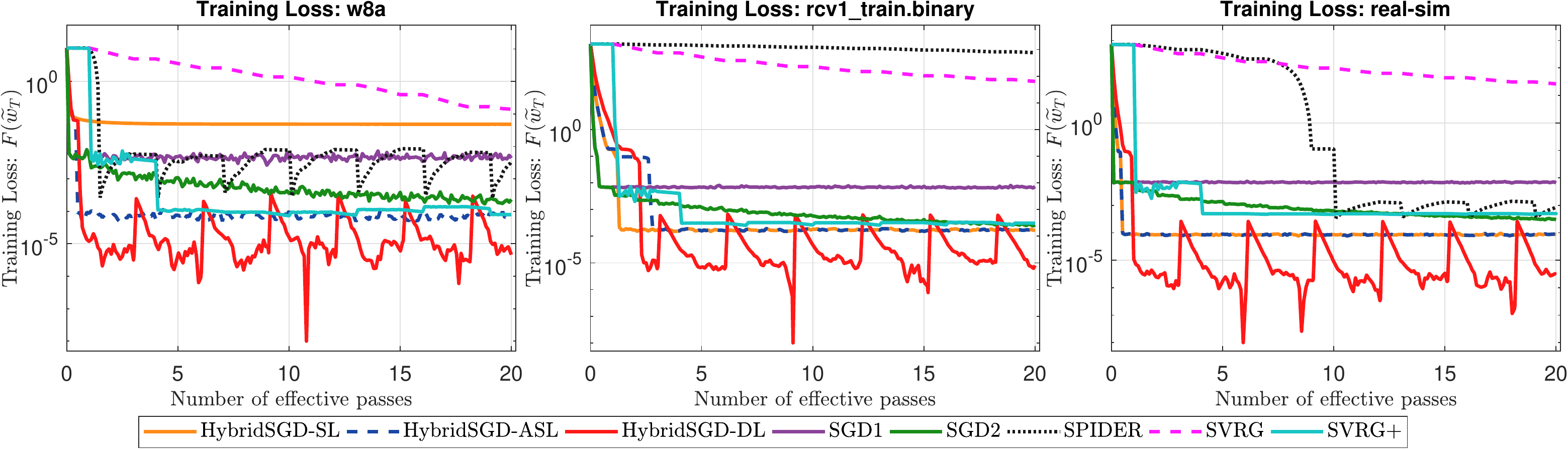}\vspace{1ex}\\
\includegraphics[width = 0.9\textwidth]{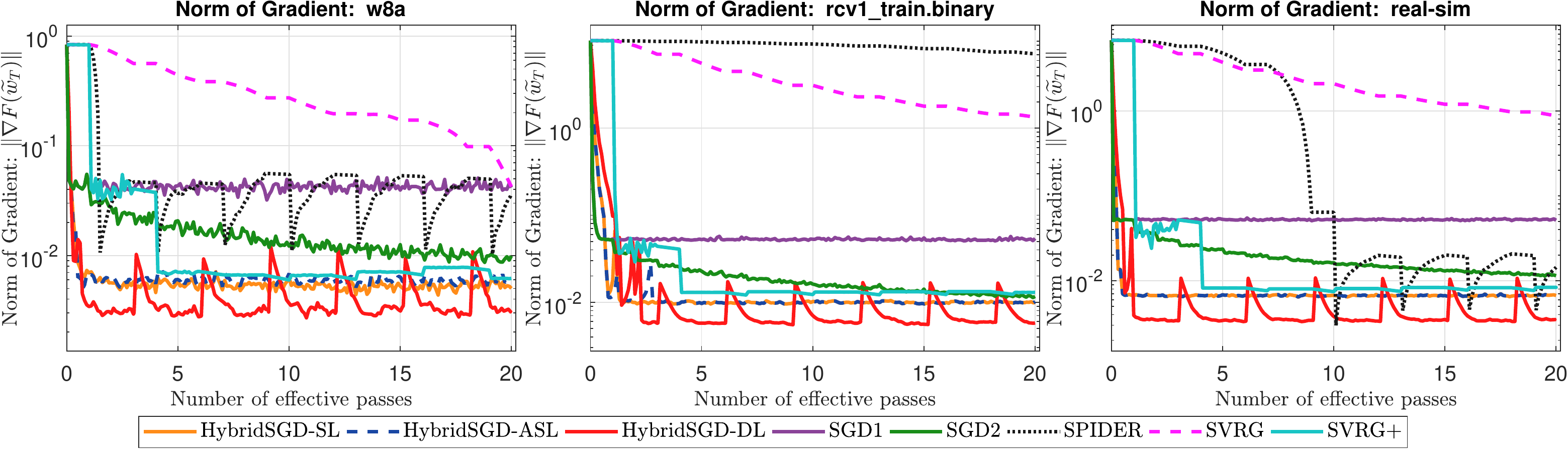}
\vspace{-1.5ex}
\caption{The training loss and gradient norms of \eqref{eq:exam1}: Single sample case $\hat{b} =1$.}\label{fig:logistic_reg}
\end{center}
\vspace{-2ex}
\end{figure}
From Fig. \ref{fig:logistic_reg}, we observe that  Hybrid-SGD-SL has similar convergence behavior as SGD1, but Hybrid-SGD-ASL works better.
Hybrid-SGD-DL is the best but has some oscillation.
SGD2 works better than SGD1 and is comparable with Hybrid-SGD-SL/ASL in the two last datasets.
SVRG performs very poorly due to its small step-size.
SVRG+ works much better than SVRG, and is comparable with our methods.
SPIDER is also slow even when we have increased its step-size.

Now, we run 3 single-loop algorithms with mini-batch of the size $\hat{b} :=300$. The result is shown in Fig.~\ref{fig:logistic_reg2} after $20$ epochs.
\begin{figure}[htp!]
\begin{center}
\includegraphics[width = 1\textwidth]{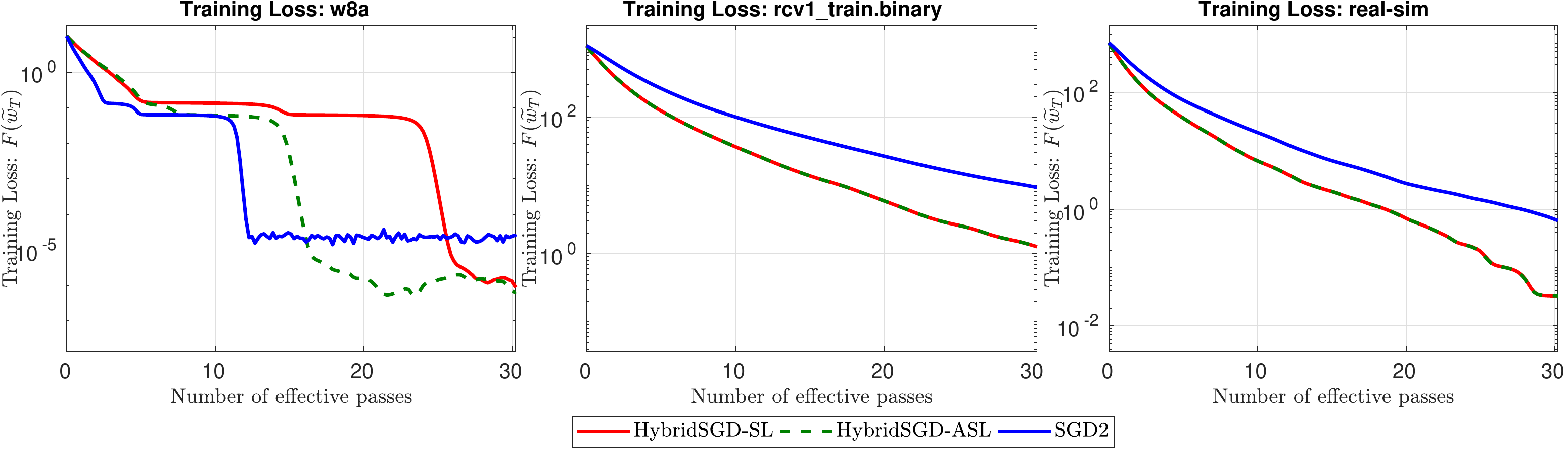}
\includegraphics[width = 1\textwidth]{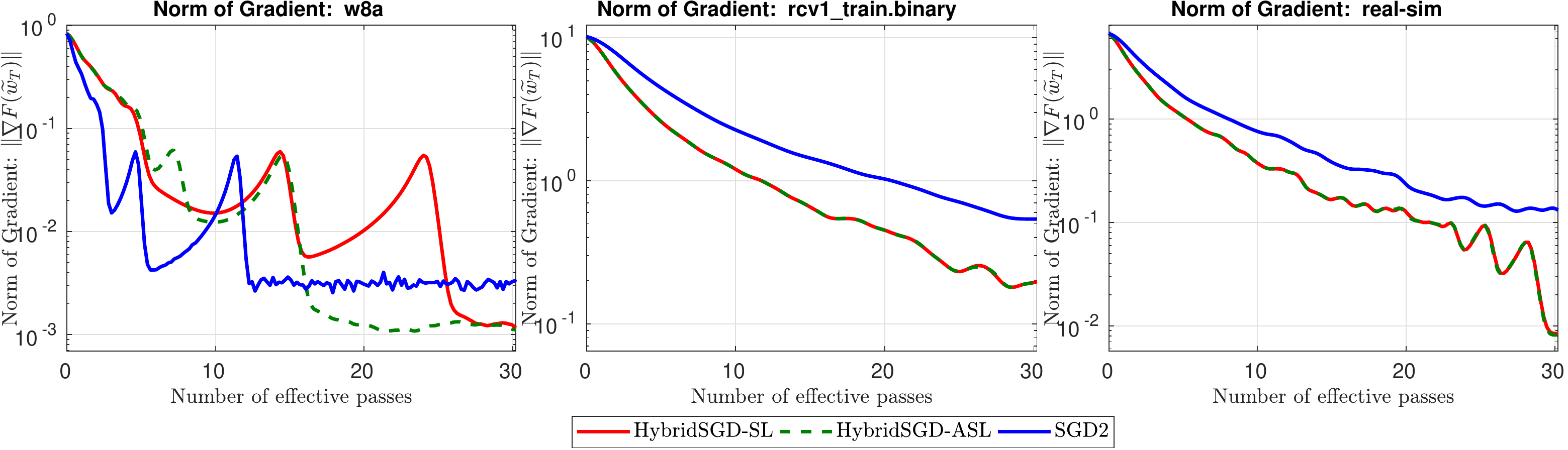}
\vspace{-1.5ex}
\caption{The training loss and gradient norms of \eqref{eq:exam1}: Mini-batch case $\hat{b} > 1$.}\label{fig:logistic_reg2}
\end{center}
\vspace{-4ex}
\end{figure}
Fig.~\ref{fig:logistic_reg2} shows similar performance between Hybrid-SGD-SL and ASL and SGD2.
Clearly, these theoretical variants of Algorithm~\ref{alg:A1} are slightly better than the adaptive SGD variant (SGD2), where a careful step-size is used.


\beforesubsec
\subsection{Binary classification involving nonconvex loss and Tikhonov's regularizer}\label{subsec:exam2}
\aftersubsec
We consider the following binary classification problem studied in \cite{zhao2010convex} involving nonconvex loss: 
\myeq{eq:exam2}{
f^{\star} := \min_{x\in\R^p}\Big\{ f(x) := \frac{1}{n}\sum_{i=1}^n\ell(a_i^{\top}x, b_i) + \frac{\lambda}{2}\norms{x}^2 \Big\},
}
where $a_i\in\R^p$ and $b_i \in\set{-1,1}$ are given data for $i=1,\cdots, n$, $\lambda > 0$ is a regularization parameter, and  $\ell$ is a nonconvex loss of the forms: $\ell(\tau, s) = \left(1 - \frac{1}{1+\exp(-\tau s)}\right)^2$ (using in two-layer neural networks).
One can check that \eqref{eq:exam2} satisfies Assumption~\ref{as:A1} with $L \approx 0.15405\max_i\norms{a_i}^2 + \lambda$.
We choose $\lambda := 0.01$, and test three variants of Algorithm~\ref{alg:A2}: Hybrid-SGD-DL and compare them with SpiderBoost, SVRG, and SVRG+.
Due to space limit, we only plot one experiment in Fig.~\ref{fig:logistic_reg3} after $20$ epochs. 
Additional experiments can be found in Supplementary Document~\ref{apdx:subsec:experiments}.


\begin{figure}[htp!]
\vspace{-1.3ex}
\begin{center}
\includegraphics[width = 1\textwidth]{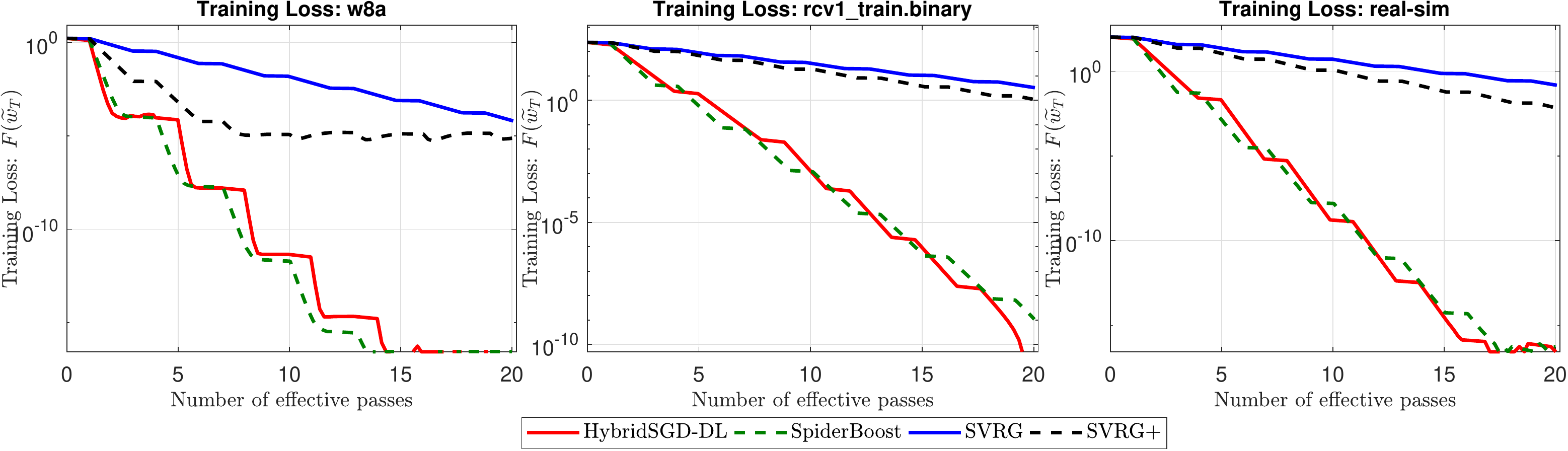}
\includegraphics[width = 1\textwidth]{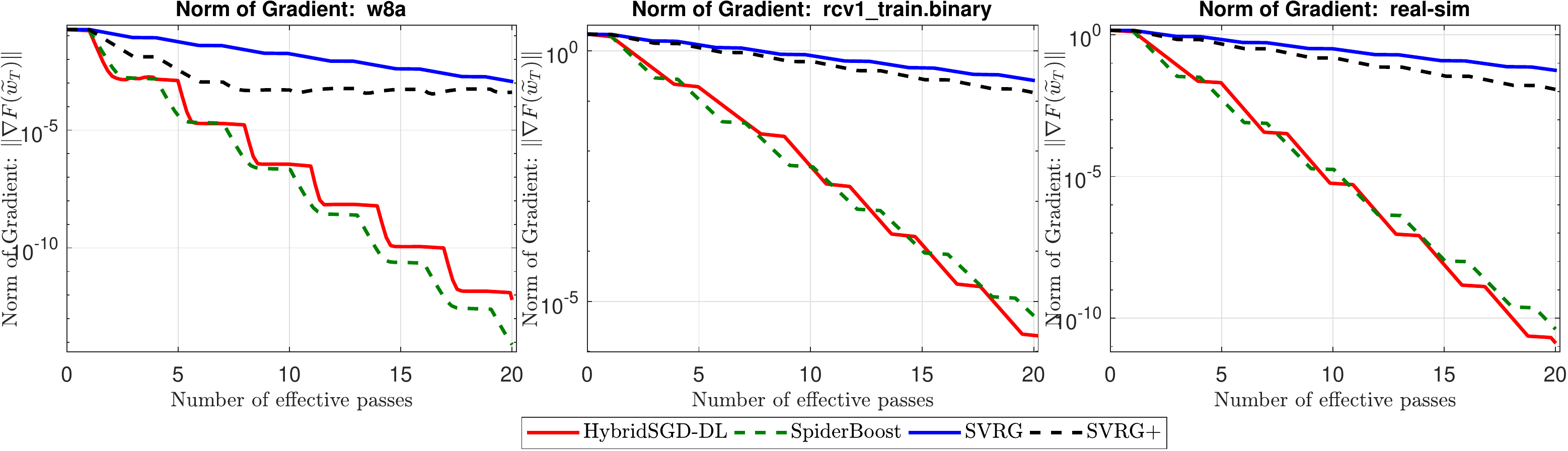}
\vspace{-1.5ex}
\caption{The training loss and gradient norms of \eqref{eq:exam2}: Mini-batch case $\hat{b} > 1$.}\label{fig:logistic_reg3}
\end{center}
\vspace{-3ex}
\end{figure}

As we can see from Fig.~\ref{fig:logistic_reg3} that Algorithm~\ref{alg:A2} performs well and is slightly better than SpiderBoost.
Note that SpiderBoost simply uses SARAH estimator with constant stepsize $\eta = \frac{1}{2L}$ but with mini-batch of the size $\lfloor\sqrt{n}\rfloor$.
It is not surprise that SpiderBoost makes very good progress to decrease the gradient norms. 
Both SVRG and SVRG+ perform much worse than Hybrid-SGD-DL and SpiderBoost in this test, but SVRG+ is slightly better than SVRG.
In our methods, due to the aid of SARAH part, they also make similar progress as SpiderBoost but using different step-sizes.

\beforesec
\section{Conclusion}
\aftersec
We have introduced a new hybrid SARAH-SGD estimator for the objective gradient of expectation optimization problems.
Under standard assumptions, we have shown that this estimator has a better variance reduction property than SARAH.
By exploiting such an estimator, we have developed a new Hybrid-SGD algorithm, Algorithm~\ref{alg:A1}, that has better complexity bounds than state-of-the-art SGDs.
Our algorithm works with both constant and adaptive step-sizes.
We have also studied its double-loop and mini-batch variants.
We believe that our approach can be extended to other choices of unbiased estimators, Hessian estimators for second-order stochastic methods, and adaptive $\beta$. 

%
\appendix
\beforesec
\section{Appendix: Properties of the hybrid stochastic estimator}\label{sec:appendix1}
\aftersec
This supplementary document provides the full proof of all the results in the main text.
First, we need the following lemma in the sequel.

\begin{lemma}\label{le:adaptive_step_size}
Given $L > 0$ and $\omega \in (0, 1)$. 
Let $\sets{\eta_t}_{t=0}^m$ be the sequence updated by
\begin{equation}\label{eq:update_of_eta_t_proof}
\eta_m := \frac{1}{L},~~~\text{and}~~\eta_t := \frac{1}{L + L^2\big[\omega\eta_{t+1} + \omega^2\eta_{t+2} + \cdots + \omega^{(m-t)}\eta_m\big]},
\end{equation}
for $t=0,\cdots, m-1$.
Then  
\begin{equation}\label{eq:stepsize_pros}
0 < \eta_0 < \eta_1 < \cdots < \eta_m = \frac{1}{L},~~~\text{and}~~~\Sigma_m := \sum_{t=0}^m\eta_t \geq \frac{(m+1)\sqrt{1-\omega}}{2L}.
\end{equation}
\end{lemma}

\begin{proof}
First, from \eqref{eq:update_of_eta_t_proof} it is obvious to show that $0 < \eta_0 < \cdots < \eta_{m-1} = \frac{1}{L(1+\omega)} < \eta_m = \frac{1}{L}$. 
At the same time, since $\omega \in (0, 1)$, we have $1 \geq \omega \geq \omega^2 \geq \cdots \geq \omega^{m}$.
By Chebyshev's sum inequality, we have
\begin{equation}\label{eq:estimate1}
\begin{array}{ll}
(m-t)\big(\omega\eta_{t+1} + \omega^2\eta_{t+2} + \cdots + \omega^{m-t}\eta_m\big) &\leq \big(\sum_{j=t+1}^m\eta_i\big)\left(\omega + \omega^2 + \cdots + \omega^{m-t}\right) \vspace{1ex}\\
&\leq \frac{\omega}{1-\omega}\big(\sum_{j=t+1}^m\eta_i\big).
\end{array}
\end{equation}
From the update \eqref{eq:update_of_eta_t_proof}, we also have
\begin{equation}\label{eq:estimate2}
\left\{\begin{array}{ll}
L^2\eta_0(\omega\eta_1 + \omega^2\eta_2 + \cdots + \omega^{m}\eta_m) &= 1 - L\eta_0 \\
L^2\eta_1(\omega\eta_2 + \omega^2\eta_3 + \cdots + \omega^{m-1}\eta_{m}) &= 1 - L\eta_1 \\
\cdots & \cdots \\
L^2\eta_{m-1}\omega\eta_m &= 1 - L\eta_{m-1} \\
0  &= 1 - L\eta_{m}.
\end{array}\right.
\end{equation}
Using \eqref{eq:estimate1} into \eqref{eq:estimate2}, we get
\begin{equation*} 
\left\{\begin{array}{ll}
\frac{\omega L^2}{1-\omega}\eta_0(\eta_0 + \eta_1   + \cdots + \eta_m) &\geq m  - mL\eta_0 + \frac{\omega L^2}{1-\omega}\eta_0^2\\
\frac{\omega L^2}{1-\omega}\eta_1(\eta_0 + \eta_1  + \cdots + \eta_{m}) &\geq (m-1) - (m-1)L\eta_1 + \frac{\omega L^2}{1-\omega}(\eta_1\eta_0 + \eta_1^2) \\
\cdots & \cdots \\
\frac{\omega L^2}{1-\omega}\eta_{m-1}(\eta_0 + \eta_1 + \cdots + \eta_m) &\geq 1 - L\eta_{m-1} + \frac{\omega L^2}{1-\omega}(\eta_{m-1}\eta_0 + \cdots + \eta_{m-1}^2) \\
\frac{\omega L^2}{1-\omega}\eta_{m}(\eta_0 + \eta_1 + \cdots + \eta_m) &\geq 1 - L\eta_{m} + \frac{\omega L^2}{1-\omega}(\eta_{m}\eta_0 + \cdots + \eta_{m}^2).
\end{array}\right.
\end{equation*}
Let $\Sigma_m := \sum_{t=0}^m\eta_t$ and $S_m := \sum_{t=0}^m\eta_t^2$.
Summing up both sides of the above inequalities, we get 
\begin{equation*}
\frac{\omega L^2}{1-\omega}\Sigma_m^2 \geq  \frac{m^2 + m + 2}{2} - L(m\eta_0 + (m-1)\eta_1 + \cdots + \eta_{m-1} + \eta_m) + \frac{\omega L^2}{2(1-\omega)}\big(S_m + \Sigma_m^2\big).
\end{equation*}
Using again Chebyshev's sum inequality, we have
\begin{equation*}
m\eta_0 + (m-1)\eta_1 + \cdots + \eta_{m-1} + \eta_m \leq \frac{m^2 + m + 2}{2(m+1)}\left(\sum_{t=0}^m\eta_t\right) = \frac{(m^2 + m + 2)\Sigma_m}{2(m+1)}.
\end{equation*}
Note that $(m+1)S_m \geq \Sigma_m^2$ by Cauchy-Schwarz's inequality, which shows that $S_m + \Sigma_m^2 \geq \big(\frac{m+2}{m+1}\big)\Sigma_m^2$.
Combining three last inequalities, we obtain the following quadratic inequation in $\Sigma_m$
\begin{equation*}
\frac{m\omega L^2}{(1-\omega)}\Sigma_m^2 + L(m^2 + m + 2)\Sigma_m -  (m+1)(m^2 + m + 2) \geq 0.
\end{equation*}
Solving this inequation with respect to $\Sigma_m > 0$, we obtain
\begin{equation*}
\begin{array}{ll}
\Sigma_m &\geq \frac{(1-\omega)\big[\sqrt{(m^2 + m + 2)^2 + \frac{4m(m+1)(m^2 + m + 2)\omega}{1-\omega}} - (m^2 + m + 2)\big]}{2\omega m L} \vspace{1ex}\\
&= \frac{2(m+1)}{L\left[1 + \sqrt{1 + \frac{4m(m+1)\omega}{(1-\omega)(m^2 + m+2)}}\right]} \vspace{1ex}\\
&\geq \frac{2(m+1)\sqrt{1-\omega}}{L\left[\sqrt{1-\omega} + \sqrt{1 + 3\omega}\right]} ~~~~~\text{since}~~\frac{m(m+1)}{m^2+m+2} < 1\vspace{1ex}\\
&\geq \frac{2(m+1)\sqrt{1-\omega}}{L(2 + \sqrt{3\omega})} ~~~~~\text{since}~~\sqrt{1+3\omega} + \sqrt{1-\omega} \leq 2 + \sqrt{3\omega}.
\end{array}
\end{equation*}
Since $\omega \in (0, 1)$, we can overestimate this as $\Sigma_m \geq  \frac{(m+1)\sqrt{1-\omega}}{2L}$, which proves \eqref{eq:stepsize_pros}.
\end{proof}

\beforesubsec
\subsection{The proof of Lemma~\ref{le:key_estimate10}: Properties of the hybrid SARAH estimator}\label{apdx:le:key_estimate10}
\aftersubsec
By taking the expectation of both sides in \eqref{eq:v_estimator} and using the fact that $\xi_t$ and $\zeta_t$ are independent, we can easily obtain \eqref{eq:biased_estimator}.

To prove \eqref{eq:key_estimate10}, we first write
\begin{equation*} 
\begin{array}{ll}
v_t - \nabla{f}(x_t) &=  \beta_{t-1}(v_{t-1}  - \nabla{f}(x_{t-1})) + \beta_{t-1}(\nabla{f}(x_t;\xi_t) - \nabla{f}(x_{t-1};\xi_t)) \vspace{1ex}\\
& + {~} (1-\beta_{t-1})\big[u_t - \nabla{f}(x_t)\big] + \beta_{t-1}\big[\nabla{f}(x_{t-1}) - \nabla{f}(x_t)\big].
\end{array}
\end{equation*}
In this case, we have
\begin{equation*} 
\begin{array}{ll}
\norms{v_t - \nabla{f}(x_t)}^2 &=  \beta_{t-1}^2\norms{v_{t-1}  - \nabla{f}(x_{t-1})}^2 + \beta_{t-1}^2\norms{\nabla{f}(x_t;\xi_t) - \nabla{f}(x_{t-1};\xi_t)}^2 \vspace{1ex}\\
& + {~} (1-\beta_{t-1})^2\norms{u_t - \nabla{f}(x_t)}^2 + \beta_{t-1}^2\norms{\nabla{f}(x_{t-1}) - \nabla{f}(x_t)}^2 \vspace{1ex}\\
&+  {~} 2\beta_{t-1}^2(v_{t-1}  - \nabla{f}(x_{t-1}))^{\top}(\nabla{f}(x_t;\xi_t) - \nabla{f}(x_{t-1};\xi_t))\vspace{1ex}\\
&+  {~} 2\beta_{t-1}^2(v_{t-1}  - \nabla{f}(x_{t-1}))^{\top}(\nabla{f}(x_{t-1}) - \nabla{f}(x_t))\vspace{1ex}\\
&+  {~} 2\beta_{t-1}(1-\beta_{t-1})(v_{t-1}  - \nabla{f}(x_{t-1}))^{\top}(u_t - \nabla{f}(x_t))\vspace{1ex}\\
&+  {~} 2\beta_{t-1}(1-\beta_{t-1})(\nabla{f}(x_t;\xi_t) - \nabla{f}(x_{t-1};\xi_t))^{\top}(u_t - \nabla{f}(x_t))\vspace{1ex}\\
&+  {~} 2\beta_{t-1}^2(\nabla{f}(x_t;\xi_t) - \nabla{f}(x_{t-1};\xi_t))^{\top}(\nabla{f}(x_{t-1}) - \nabla{f}(x_t))\vspace{1ex}\\
&+  {~} 2\beta_{t-1}(1-\beta_{t-1})(u_t - \nabla{f}(x_t))^{\top}(\nabla{f}(x_{t-1}) - \nabla{f}(x_t)).
\end{array}
\end{equation*}
Let us first take expectation w.r.t. $\xi_t$ conditioned on $\zeta_t$ to obtain
\begin{equation*}
\begin{array}{ll}
\Exps{\xi_t}{\norms{v_t - \nabla{f}(x_t)}^2 \mid \zeta_t} &=  \beta_{t-1}^2\norms{v_{t-1}  - \nabla{f}(x_{t-1})}^2 + \beta_{t-1}^2\Exps{\xi_t}{\norms{\nabla{f}(x_t;\xi_t) - \nabla{f}(x_{t-1};\xi_t)}^2 \mid \zeta_t} \vspace{1ex}\\
& + {~} (1-\beta_{t-1})^2\norms{u_t - \nabla{f}(x_t)}^2 - \beta_{t-1}^2\norms{\nabla{f}(x_{t-1}) - \nabla{f}(x_t)}^2 \vspace{1ex}\\
&+  {~} 2\beta_{t-1}(1-\beta_{t-1})(v_{t-1}  - \nabla{f}(x_{t-1}))^{\top}(u_t - \nabla{f}(x_t))\vspace{1ex}\\
&+  {~} 2\beta_{t-1}(1-\beta_{t-1})(\nabla{f}(x_t) - \nabla{f}(x_{t-1}))^{\top}(u_t - \nabla{f}(x_t))\vspace{1ex}\\
&+  {~} 2\beta_{t-1}(1-\beta_{t-1})(u_t - \nabla{f}(x_t))^{\top}(\nabla{f}(x_{t-1}) - \nabla{f}(x_t)).
\end{array}
\end{equation*}
Now, taking the expectation over $\zeta_t$, and noting that $\Exps{(\xi_t,\zeta_t)}{\cdot} = \Exps{\zeta_t}{\Exps{\xi_t}{\cdot\mid \zeta_t}}$ and $\Exps{\zeta_t}{u_t - \nabla{f}(x_t)} = 0$, we get
\begin{equation*} 
\begin{array}{ll}
\Exps{(\xi_t,\zeta_t)}{\norms{v_t - \nabla{f}(x_t)}^2} &=  \beta_{t-1}^2\norms{v_{t-1}  - \nabla{f}(x_{t-1})}^2 + \beta_{t-1}^2\Exps{\xi_t}{\norms{\nabla{f}(x_t;\xi_t) - \nabla{f}(x_{t-1};\xi_t)}^2} \vspace{1ex}\\
& +  {~} (1-\beta_{t-1})^2\Exps{\zeta_t}{\norms{u_t - \nabla{f}(x_t)}^2} - \beta_{t-1}^2\norms{\nabla{f}(x_{t-1}) - \nabla{f}(x_t)}^2,
\end{array}
\end{equation*}
which is exactly \eqref{eq:key_estimate10}.
\Eproof

\beforesubsec
\subsection{The proof of Lemma~\ref{le:upper_bound_new}: Bound on the variance of the hybrid estimator}\label{apdx:le:upper_bound_new}
\aftersubsec
We first upper bound \eqref{eq:key_estimate10} by using $\sigma_t^2 := \Exps{\zeta_t}{\norms{u_t - \nabla{f}(x_t)}^2}$  and then taking the full expectation over $\Fc_t := \sigma(v_0, v_1, \cdots, v_t)$ as
\begin{equation*}
\begin{array}{ll}
\Exp{\norms{v_t - \nabla{f}(x_t)}^2} &\leq  \beta_{t-1}^2\Exp{\norms{v_{t-1}  - \nabla{f}(x_{t-1})}^2} + \beta_{t-1}^2\Exp{\norms{\nabla{f}(x_t;\xi_t) - \nabla{f}(x_{t-1};\xi_t)}^2} \vspace{1ex}\\
& +{~} (1-\beta_{t-1})^2\sigma_t^2 \vspace{1ex}\\
&\overset{\tiny\eqref{eq:L_smooth}}{\leq}  \beta_{t-1}^2\Exp{\norms{v_{t-1}  - \nabla{f}(x_{t-1})}^2} + \beta_{t-1}^2L^2\Exp{\norms{x_t - x_{t-1}}^2} + (1-\beta_{t-1})^2\sigma_t^2.
\end{array}
\end{equation*}
If we define $a_t^2 := \Exp{\norms{ v_t - \nabla{f}(x_t)}^2}$, then\ the above inequality can lead to
\begin{equation*} 
a_t^2 \leq \beta_{t-1}^2a_{t-1}^2 + \beta_{t-1}^2L^2\Exp{\norms{x_{t} - x_{t-1}}^2} + (1-\beta_{t-1})^2\sigma^2_t.
\end{equation*}
Denote $b_{t-1}^2 := \Exp{\norms{x_t-x_{t-1}}^2}$. 
Then, we have from the last inequality that
\begin{equation*} 
a_t^2 \leq \beta_{t-1}^2a_{t-1}^2 +  L^2\beta_{t-1}^2b_{t-1}^2 + (1-\beta_{t-1})^2\sigma^2_t.
\end{equation*}
By induction, this inequality implies 
\begin{equation*} 
\begin{array}{ll}
a_t^2 &\leq {~} \beta_{t-1}^2a_{t-1}^2 +  L^2\beta_{t-1}^2b_{t-1}^2 + (1-\beta_{t-1})^2\sigma^2_t \vspace{1ex}\\
&\leq {~} \beta_{t-1}^2\big[\beta_{t-2}^2a_{t-2}^2 +  L^2\beta_{t-2}^2b_{t-2}^2 + (1-\beta_{t-2})^2\sigma^2\big] +  L^2\beta_{t-1}^2b_{t-1}^2 + (1-\beta_{t-1})^2\sigma^2_t \vspace{1ex}\\
&= {~} \beta_{t-1}^2\beta_{t-2}^2a_{t-2}^2 +  L^2\beta_{t-1}^2\beta_{t-2}^2b_{t-2}^2 + L^2\beta_{t-1}^2b_{t-1}^2 + \big[(1-\beta_{t-1})^2\sigma^2_t + \beta_{t-1}^2(1-\beta_{t-2})^2\sigma^2_{t-1}\big] \vspace{1ex}\\
&\leq {~} \beta_{t-1}^2\beta_{t-2}^2\big[\beta_{t-3}^2a_{t-3}^2 +  L^2\beta_{t-3}^2b_{t-3}^2 + (1-\beta_{t-3})^2\sigma^2_{t-2}\big] \vspace{1ex}\\
& + {~} L^2\beta_{t-1}^2\beta_{t-2}^2b_{t-2}^2 + L^2\beta_{t-1}^2b_{t-1}^2 + \big[(1-\beta_{t-1})^2\sigma^2_t + \beta_{t-1}^2(1-\beta_{t-2})^2\sigma^2_{t-1}\big] \vspace{1ex}\\
&= {~} \beta_{t-1}^2\beta_{t-2}^2\beta_{t-3}^2a_{t-3}^2 +  L^2\beta_{t-1}^2\beta_{t-2}^2\beta_{t-3}^2b_{t-3}^2 +   L^2\beta_{t-1}^2\beta_{t-2}^2b_{t-2}^2 \vspace{1ex}\\
& + {~} L^2\beta_{t-1}^2b_{t-1}^2 + \big[(1-\beta_{t-1})^2\sigma^2_t + \beta_{t-1}^2(1-\beta_{t-2})^2\sigma^2_{t-1} + \beta_{t-1}^2\beta_{t-2}^2(1-\beta_{t-3})^2\sigma^2_{t-2}\big] \vspace{1ex}\\
& \cdots \cdots \vspace{1ex}\\
&\leq {~} (\beta_{t-1}^2\cdots\beta_0^2)a_0^2 + L^2(\beta_{t-1}^2\cdots\beta_0^2)b_0^2 + L^2(\beta_{t-1}^2\cdots\beta_1^2)b_1^2 + \cdots + L^2\beta_{t-1}^2b_{t-1}^2 \vspace{1ex}\\
&+ {~} \big[(1-\beta_{t-1})^2\sigma^2_t + \beta_{t-1}^2(1-\beta_{t-2})^2\sigma^2_{t-1} + \beta_{t-1}^2\beta_{t-2}^2(1-\beta_{t-3})^2\sigma^2_{t-2} + \cdots \vspace{1ex}\\
&{~~~}  + {~} \beta_{t-1}^2\beta_{t-2}^2\cdots\beta_1^2(1-\beta_0)^2\sigma^2_1\big].
\end{array}
\end{equation*}
Here, we use a convention that $\prod_{i = t+1}^t \beta_i^2 = 1$.
As a result, it can be rewritten in a compact form as
\begin{equation}\label{eq:key_estimate1d}
a_t^2 \leq \Big(\prod_{i=1}^{t}\beta_{i-1}^2\Big)a_0^2 + L^2\sum_{i=0}^{t-1}\Big(\prod_{j=i+1}^{t}\beta_{j-1}^2\Big)b_i^2 + \sum_{i=0}^{t-1}\Big(\prod_{j=i+2}^{t}\beta_{j-1}^2\Big)(1-\beta_i)^2\sigma^2_{i+1}.
\end{equation}
Define $\omega_{t} := \prod_{i=1}^{t}\beta_{i-1}^2$, $\omega_{i, t} := \prod_{j=i+1}^{t}\beta_{j-1}^2$, and $S_{t} := \sum_{i=0}^{t-1}s_i = \sum_{i=0}^{t-1}\big(\prod_{j=i+2}^{t}\beta_{j-1}^2\big)(1-\beta_i)^2\sigma^2_{i+1}$ with $s_i := (1-\beta_i)^2\sigma^2_{i+1}\big(\prod_{j=i+2}^{t}\beta_{j-1}^2\big)$.
Then, we can rewrite \eqref{eq:key_estimate1d} as
\begin{equation*}
a_t^2 \leq \omega_{t}a_0^2 + L^2\sum_{i=0}^{t-1}\omega_{i,t}b_i^2 + S_{t},
\end{equation*}
which is exactly \eqref{eq:vt_variance_bound_new}.
\Eproof

\beforesec
\section{Appendix: Convergence analysis of Algorithm~\ref{alg:A1} and Algorithm~\ref{alg:A2}}\label{apdx:sec:convergence-analysis}
\aftersec
We provide the full convergence analysis for  Algorithm~\ref{alg:A1} and Algorithm~\ref{alg:A2} in the single-sample case.

\beforesubsec
\subsection{The proof of Lemma~\ref{le:key_estimate_of_convergence}: One-iteration analysis}\label{apdx:le:key_estimate_of_convergence}
\aftersubsec
The following lemma provides a key estimate for convergence analysis of Algorithm~\ref{alg:A1}.  

\begin{lemma}\label{le:key_estimate_of_convergence}
Let $\sets{x_t}$ be the sequence generated by Algorithm~\ref{alg:A1}.
Then, under Assumption~\ref{as:A1}, we have the following estimate:
\begin{equation}\label{eq:key_estimate_31}
\begin{array}{ll}
\Exp{f(x_{m+1})} &\leq \Exp{f(x_0)} - \displaystyle\frac{1}{2}\displaystyle\sum_{t=0}^m\eta_t\Exp{\norms{\nabla{f}(x_t)}^2} \\
& + {~} \dfrac{1}{2}\Big(\displaystyle\sum_{t=0}^m\eta_t\omega_t\Big)\Exp{\norms{v_0 - \nabla{f}(x_0)}^2}  + \frac{1}{2}\Big(\displaystyle\sum_{t=0}^m\eta_tS_t\Big) + \frac{1}{2}\Tc_m,
\end{array}
\end{equation}
where 
\begin{equation}\label{eq:T_m}
\Tc_m :=   L^2\sum_{t=1}^m \eta_t \sum_{i=0}^{t-1}\omega_{i,t}\eta_i^2\Exp{\norms{v_i}^2} - \sum_{t=0}^m\big(\eta_t -  L\eta_t^2\big)\Exp{\norms{v_t}^2},
\end{equation}
and $\omega_t$, $\omega_{i,t}$, and $S_t$ are defined in Lemma~\ref{le:upper_bound_new}.
\end{lemma}

\begin{proof}
First, from the $L$-smoothness of $f$, we have
\begin{equation*} 
\begin{array}{ll}
f(x_{t+1}) &\leq f(x_t) - \eta_t\iprods{\nabla{f}(x_t), v_t} + \frac{L\eta_t^2}{2}\norms{v_t}^2 \vspace{1ex}\\
&=  f(x_t) - \frac{\eta_t}{2}\norms{\nabla{f}(x_t)}^2 - \big(\frac{\eta_t}{2} - \frac{L\eta_t^2}{2}\big)\norms{v_t}^2 + \frac{\eta_t}{2}\norms{v_t - \nabla{f}(x_t)}^2.
\end{array}
\end{equation*}
Taking the expectation over the randomness  $(\xi_t, \zeta_t)$ of this estimate, we obtain
\begin{equation*} 
\begin{array}{ll}
\Exps{(\xi_t,\zeta_t)}{f(x_{t+1})} &\leq  f(x_t) - \frac{\eta_t}{2}\Exps{(\xi_t,\zeta_t)}{\norms{\nabla{f}(x_t)}^2} - \frac{\eta_t}{2}\big(1 - L\eta_t\big)\Exps{(\xi_t,\zeta_t)}{\norms{v_t}^2} \vspace{1ex}\\
& + {~} \frac{\eta_t}{2}\Exps{(\xi_t,\zeta_t)}{\norms{v_t - \nabla{f}(x_t)}^2}.
\end{array}
\end{equation*}
Taking the full expectation over the entire history up to the $t$-th iteration, and then using \eqref{eq:vt_variance_bound_new} and noting that $x_t - x_{t-1} = -\eta_{t-1}v_{t-1}$, we obtain
\begin{equation}\label{eq:est5b}
\begin{array}{ll}
\Exp{f(x_{t+1})} &\leq  \Exp{f(x_t)} - \frac{\eta_t}{2}\Exp{\norms{\nabla{f}(x_t)}^2} - \frac{\eta_t}{2}\big(1 - L\eta_t\big)q_t^2 + \frac{\eta_t}{2}a_t^2 \vspace{1ex}\\
&\leq \Exp{f(x_t)} - \frac{\eta_t}{2}\Exp{\norms{\nabla{f}(x_t)}^2} - \frac{\eta_t}{2}\big(1 - L\eta_t\big)q_t^2 \vspace{1ex}\\
& + {~} \frac{\eta_t}{2}\Big[\omega_{t}a_0^2 +  L^2\sum_{i=0}^{t-1}\omega_{i,t}\eta_i^2q_i^2 + S_{t}\Big],
\end{array}
\end{equation}
%
where $q_t^2 := \Exp{\norms{v_t}^2}$ and $a_t^2 := \Exp{\norms{v_t - \nabla{f}(x_t)}^2}$.
Here, we use  $b_{t-1}^2 := \Exp{\norms{x_t - x_{t-1}}^2} = \eta_{t-1}^2\Exp{\norms{v_{t-1}}^2} = \eta_{t-1}^2q_{t-1}^2$ in the last inequality.

Summing up \eqref{eq:est5b} from $t=0$ to $t=m$, we obtain
\begin{equation}\label{eq:est5c}
\begin{array}{ll}
\Exp{f(x_{m+1})} &\leq \Exp{f(x_0)} - \sum_{t=0}^m\frac{\eta_t}{2}\Exp{\norms{\nabla{f}(x_t)}^2} - \sum_{t=0}^m\frac{\eta_t}{2}\big(1 - L\eta_t\big)q_t^2 \vspace{1ex}\\
& + {~} \frac{1}{2}\left(\sum_{t=0}^m\omega_t\eta_t\right)a_0^2 + \frac{1}{2}\left(\sum_{t=0}^m\eta_tS_t\right) +  \frac{L^2}{2}\sum_{t=0}^m\eta_t\sum_{i=0}^{t-1}\omega_{i,t}\eta_i^2q_i^2.
\end{array}
\end{equation}
Let us define $T_m$ as in \eqref{eq:T_m}, i.e.:
\begin{equation*} 
\Tc_m :=  L^2\sum_{t=1}^m \eta_t \sum_{i=0}^{t-1}\omega_{i,t}\eta_i^2q_i^2 - \sum_{t=0}^m\eta_t\big(1 -  L\eta_t\big)q_t^2.
\end{equation*}
Then, we obtain from \eqref{eq:est5c} the estimate \eqref{eq:key_estimate_31}.
\end{proof}

\beforesubsec
\subsection{The proof of Theorem~\ref{th:singe_loop_const_step}: Single-loop with constant step-size}\label{apdx:th:singe_loop_const_step}
\aftersubsec
We analyze the case $\beta_t = \beta \in (0, 1)$ fixed and the step-size $\eta_t = \eta > 0$ fixed.
From Lemma~\ref{le:upper_bound_new}, we have $\omega_t = \beta^{2t}$, $\omega_{i,t} = \beta^{2(t-i)}$, and 
\begin{equation*}
\begin{array}{ll}
s_t &:= \sum_{i=0}^{t-1}\big(\prod_{j=i+2}^{t}\beta_{j-1}^2\big)(1-\beta_i)^2 \vspace{1ex}\\
& = (1-\beta)^2\big[1 + \beta^2 + \beta^{4} + \cdots + \beta^{2(t-1)}\big] \vspace{1ex}\\
&= (1-\beta)^2\Big[\frac{1-\beta^{2t}}{1-\beta^2}\Big] \vspace{1ex}\\
&< \frac{1-\beta}{1+\beta}.
\end{array}
\end{equation*}
In this case, by convention that $\omega_0 = 1$, we have
\begin{equation}\label{eq:bound_of_st_omegat}
\sum_{t=0}^ms_t < \frac{(1-\beta)(m+1)}{1+\beta}~~~\text{and}~~\sum_{t=0}^m\omega_t = 1 +  \frac{\beta^2(1-\beta^{2m})}{1-\beta^2} = \frac{1-\beta^{2(m+1)}}{1-\beta^2} < \frac{1}{1-\beta^2}.
\end{equation}
Now, to bound the quantity $\Tc_m$ defined by \eqref{eq:T_m}, we note that
\begin{equation*}
\begin{array}{ll}
\displaystyle\sum_{t=1}^m\sum_{i=0}^{t-1}\beta^{2(t-i)}q_i^2 &= \displaystyle\sum_{i=0}^{0}\beta^{2(1-i)}q_i^2 + \sum_{i=0}^{1}\beta^{2(2-i)}q_i^2 + \sum_{i=0}^{2}\beta^{2(3-i)}q_i^2 + \cdots + \sum_{i=0}^{m-1}\beta^{2(m-i)}q_i^2 \vspace{1ex}\\
&= \beta^2q_0^2 + \big[\beta^4q_0^2 + \beta^2q_1^2\big] + \big[\beta^6q_0^2 + \beta^4q_1^2 + \beta^2q_0^2\big] + \cdots \vspace{1ex}\\
& + {~} \big[\beta^{2m}q_0^2 + \beta^{2(m-1)}q_1 + \cdots + \beta^2q_{m-1}^2\big] \vspace{1ex}\\
&= \beta^2\big[1 + \beta^2 + \cdots + \beta^{2(m-1)}\big]q_0^2 + \beta^2\big[1 + \beta^2 + \cdots + \beta^{2(m-2)}\big]q_1^2 + \cdots \vspace{1ex}\\ &+ {~} \beta^2\big[1 + \beta^2\big]q_{m-2}^2 + \beta^2q_{m-1}^2\vspace{1ex}\\
&= \frac{\beta^2}{1-\beta^2}\Big[ (1-\beta^{2m})q_0^2 + (1-\beta^{2(m-1)})q_1^2 + \cdots + (1-\beta^2)q_{m-1}^2\Big].
\end{array}
\end{equation*}
Using this expression, we can write $\Tc_m$ from  \eqref{eq:T_m} as
\begin{equation}\label{eq:T_m_constant}
\begin{array}{ll}
\Tc_m &= \eta\Big[\frac{\beta^2(1-\beta^{2m})L^2\eta^2}{1-\beta^2} - (1-L\eta)\Big]q_0^2 +  \eta\Big[\frac{\beta^2(1-\beta^{2(m-1)})L^2\eta^2}{1-\beta^2} - (1-L\eta)\Big]q_1^2 + \cdots \vspace{1ex}\\
&+ {~}  \eta\Big[\frac{\beta^2(1-\beta^{2})L^2\eta^2}{1-\beta^2} - (1-L\eta)\Big]q_{m-1}^2 - \eta(1-L\eta)q_{m}^2.
\end{array}
\end{equation}
To guarantee $\Tc_m \leq 0$, from \eqref{eq:T_m_constant}, we need to choose
\begin{equation}\label{eq:cond_of_eta}
\left\{\begin{array}{ll}
\frac{L^2\eta^2\beta^2(1-\beta^{2m})}{1-\beta^2} - (1-L\eta) &\leq 0\vspace{1ex}\\
\frac{L^2\eta^2\beta^2(1-\beta^{2(m-1)})}{1-\beta^2} - (1-L\eta) &\leq 0\vspace{1ex}\\
\cdots & \cdots \vspace{1ex}\\
\frac{L^2\eta^2\beta^2(1-\beta^2)}{1-\beta^2} - (1-L\eta) &\leq 0\vspace{1ex}\\
- (1 - L\eta) &\leq 0.
\end{array}\right.
\end{equation}
Clearly, since $1 - \beta^{2(m-i)} \geq 1 - \beta^2$ for $i=0,\cdots, m-1$, if we define $\alpha_m^2 := \frac{\beta^2(1-\beta^{2m})}{1-\beta^2}$, then 
the condition \eqref{eq:cond_of_eta} holds if $L^2\eta^2\alpha^2_m - (1-L\eta) \leq 0$.
By tightening this condition, we obtain a quadratic equation $L^2\eta^2\alpha_m^2 - (1-L\eta) = 0$ in $\eta$, which leads to 
\begin{equation}\label{eq:step_size}
\eta := \frac{2}{L(\sqrt{1 + 4\alpha_m^2} + 1)}~~~\text{with}~~\alpha_m^2 :=  \frac{\beta^2(1-\beta^{2m})}{1-\beta^2}.
\end{equation}
Note that since $\alpha_m^2 \leq \frac{\beta^2}{1-\beta^2}$, we have $\eta \geq \underline{\eta}: = \frac{2\sqrt{1-\beta^2}}{L(\sqrt{1-\beta^2} + \sqrt{1 + 3\beta^2})}$.
In that case, by using \eqref{eq:bound_of_st_omegat} and \eqref{eq:cond_of_eta}, \eqref{eq:key_estimate_31} reduces to
\begin{equation}\label{eq:est5c_1}
\begin{array}{ll}
\Exp{f(x_{m+1})} &\overset{\tiny\eqref{eq:bound_of_st_omegat}}{\leq} \Exp{f(x_0)} - \frac{\eta}{2}\displaystyle\sum_{t=0}^m\Exp{\norms{\nabla{f}(x_t)}^2} \vspace{1ex}\\
& + {~} \frac{\eta(1-\beta^{2(m+1)})}{2(1-\beta^2)}\Exp{\norms{v_0 - \nabla{f}(x_0)}^2}  + \Big[\frac{(1-\beta)(m+1)}{1+\beta}\Big]\frac{\eta\sigma^2}{2}.
\end{array}
\end{equation}
Note that $\Exp{\norms{v_0 - \nabla{f}(x_0)}^2} \leq \frac{\sigma^2}{b}$  and $\Exp{f(x_{m+1})} \geq f^{\star}$, we can further bound \eqref{eq:est5c_1} as
\begin{equation*}
 \frac{\eta}{2}\displaystyle\sum_{t=0}^m\Exp{\norms{\nabla{f}(x_t)}^2}  \leq \Exp{f(x_0)} - f^{\star} + \frac{\eta\sigma^2}{2(1+\beta)}\Big[\frac{1}{(1-\beta)b} + (1-\beta)(m+1)\Big].
\end{equation*}
Multiplying both sides  of this inequality by $\frac{2}{\eta(m+1)}$, and then using the lower bound of $\eta$ from \eqref{eq:step_size}, we obtain
\begin{equation}\label{eq:main_estimate_for_single_loop}
\begin{array}{ll}
 \frac{1}{m+1}\displaystyle\sum_{t=0}^m\Exp{\norms{\nabla{f}(x_t)}^2}  &\leq \frac{2}{\eta(m+1)}\Big[\Exp{f(x_0)} - f^{\star}\Big] + \frac{\sigma^2}{(1+\beta)}\Big[\frac{1}{(1-\beta)b(m+1)} + (1-\beta)\Big] \vspace{1ex}\\
 &\leq  \frac{L}{(m+1)}\left(\frac{\sqrt{1-\beta^2} + \sqrt{1+3\beta^2}}{\sqrt{1-\beta^2}}\right)\Big[\Exp{f(x_0)} - f^{\star}\Big] \vspace{1ex}\\
 & + {~} \frac{\sigma^2}{(1+\beta)}\Big[\frac{1}{(1-\beta)b(m+1)} + (1-\beta)\Big].
 \end{array}
\end{equation}
Let us choose $\beta := 1 - \frac{c_1}{\sqrt{b(m+1)}}$ for some $0 < c_1 < \sqrt{b(m+1)}$.
In this case, the last two terms of the right-hand side of \eqref{eq:main_estimate_for_single_loop} become
\begin{equation*}
\frac{1}{(1-\beta)b(m+1)} + (1-\beta) = \left(c_1 + \frac{1}{c_1}\right)\frac{1}{\sqrt{b(m+1)}}.
\end{equation*}
With this choice of $\beta$, \eqref{eq:main_estimate_for_single_loop} leads to 
\begin{equation}\label{eq:single_loop} 
\begin{array}{ll}
 \frac{1}{m+1}\displaystyle\sum_{t=0}^m\Exp{\norms{\nabla{f}(x_t)}^2}  &\leq  \frac{L}{(m+1)}\left(\frac{\sqrt{1-\beta^2} + \sqrt{1+3\beta^2}}{\sqrt{1-\beta^2}}\right)\Big[\Exp{f(x_0)} - f^{\star}\Big]  \vspace{1ex}\\
 & + {~}  \left(c_1 + \frac{1}{c_1}\right)\frac{\sigma^2}{(1+\beta)\sqrt{b(m+1)}}.
 \end{array}
\end{equation}
(a)~Since $\beta = 1 - \frac{c_1}{\sqrt{b(m+1)}} < 1$ and $c_1 < \sqrt{b(m+1)}$, we have 
\begin{equation*}
1 - \beta^2 = 1 - \Big(1 - \frac{c_1}{\sqrt{b(m+1)}} \Big)^2 = \frac{2c_1}{\sqrt{b(m+1)}} - \frac{c_1^2}{b(m+1)} = \frac{2c_1\sqrt{b(m+1)} - c_1^2}{b(m+1)} > \frac{c_1}{\sqrt{b(m+1)}},
\end{equation*} 
and $\sqrt{1 - \beta^2} + \sqrt{1+3\beta^2} \leq 1 + \sqrt{1+3\beta^2} \leq 3$.
On the other hand, from \eqref{eq:step_size}, we have
\begin{equation}\label{eq:LR_lower_bound}
\eta \geq \underline{\eta} = \frac{2\sqrt{1-\beta^2}}{L(\sqrt{1-\beta^2} + \sqrt{1 + 3\beta^2})} \geq  \frac{2\sqrt{c_1}}{3\nhan{L}\big[ b(m+1)\big]^{1/4}}.
\end{equation}
This proves (a).

Let us define define $f^0 := f(x_0)$. 
Then, using $\beta < 1$ and \eqref{eq:LR_lower_bound} into \eqref{eq:single_loop}, we get
\begin{equation*} 
\frac{1}{m+1}\displaystyle\sum_{t=0}^m\Exp{\norms{\nabla{f}(x_t)}^2} \leq \frac{3 L b^{1/4}}{\sqrt{c_1}(m+1)^{3/4}}\big[f^0 - f^{\star}\big] + \left(c_1 + \frac{1}{c_1}\right)\frac{\sigma^2}{\sqrt{b(m+1)}}.
\end{equation*}
(b)~Let us choose $b := c_2\sigma^{8/3}(m+1)^{1/3}$ for some constant $c_2 > 0$.
Then the last estimate becomes
\begin{equation}\label{eq:single_loop_final2} 
\frac{1}{m+1}\displaystyle\sum_{t=0}^m\Exp{\norms{\nabla{f}(x_t)}^2} \leq \frac{\sigma^{2/3}}{(m+1)^{2/3}}\left[\frac{3 L c_2^{1/4}}{\sqrt{c_1}}\big[f^0 - f^{\star}\big] + \left(c_1 + \frac{1}{c_1}\right)\frac{1}{\sqrt{c_2}}\right].
\end{equation}
To guarantee $\frac{1}{m+1}\displaystyle\sum_{t=0}^m\Exp{\norms{\nabla{f}(x_t)}^2} \leq \varepsilon^2$, from \eqref{eq:single_loop_final2} we need to set 
\begin{equation*}
\frac{\sigma^{2/3}}{(m+1)^{2/3}}\left[\frac{3 L c_2^{1/4}}{\sqrt{c_1}}\big[f^0 - f^{\star}\big] + \left(c_1 + \frac{1}{c_1}\right)\frac{1}{\sqrt{c_2}}\right] \leq \varepsilon^2.
\end{equation*}
This leads to $m+1 \geq \frac{\sigma}{\varepsilon^3}\left[\frac{3 L c_2^{1/4}}{\sqrt{c_1}}\big[f^0 - f^{\star}\big] + \left(c_1 + \frac{1}{c_1}\right)\frac{1}{\sqrt{c_2}}\right]^{3/2}$.
Therefore, we can choose $m$ as shown in \eqref{eq:choice_of_m}.

Finally, if $c_1 = 1$, then the number of stochastic gradient evaluations is $\Tc_{ge}$ is
\begin{equation*}
\begin{array}{ll}
\Tc_{ge} &= b + 3m =  c_2\sigma^{8/3}(m + 1)^{1/3} +  \frac{3\sigma}{\varepsilon^3}\left[ 3 L c_2^{1/4}\big[f^0 - f^{\star}\big] + \frac{2}{\sqrt{c_2}}\right]^{3/2} \vspace{1ex}\\
& =  \frac{c_2\sigma^3}{\varepsilon}\left[ 3 L c_2^{1/4}\big[f^0 - f^{\star}\big] + \frac{2}{\sqrt{c_2}}\right]^{1/2} + \frac{3\sigma}{\varepsilon^3}\left[ 3 L c_2^{1/4}\big[f^0 - f^{\star}\big] + \frac{2}{\sqrt{c_2}}\right]^{3/2}  \vspace{1ex}\\
& =  \frac{\sigma^3}{\varepsilon}\left[ 3 L c_2^{9/4}\big[f^0 - f^{\star}\big] + c_2^{3/2}\right]^{1/2} + \frac{3\sigma}{\varepsilon^3}\left[ 3 L c_2^{1/4}\big[f^0 - f^{\star}\big] + \frac{2}{\sqrt{c_2}}\right]^{3/2}  \vspace{1ex}\\
& = \BigO{\frac{\sigma}{\varepsilon^3} + \frac{\sigma^3}{\varepsilon}},
\end{array}
\end{equation*}
which proves \eqref{eq:overall_complexity2}.
\Eproof

\beforesubsec
\subsection{The proof of Theorem~\ref{th:singe_loop_adapt_step}: Single-loop with adaptive step-size}\label{apdx:th:singe_loop_adapt_step}
\aftersubsec
First, from Lemma~\ref{le:key_estimate_of_convergence}, we have
\begin{equation}\label{eq:key_estimate_300}
\begin{array}{ll}
\Exp{f(x_{m+1})} &\leq \Exp{f(x_0)} - \displaystyle\frac{1}{2}\displaystyle\sum_{t=0}^m\eta_t\Exp{\norms{\nabla{f}(x_t)}^2} \vspace{1ex}\\
& + \dfrac{1}{2}\Big(\displaystyle\sum_{t=0}^m\eta_t\omega_t\Big)\Exp{\norms{v_0 - \nabla{f}(x_0)}^2}  + \frac{1}{2}\Big(\displaystyle\sum_{t=0}^m\eta_tS_t\Big) + \frac{1}{2}\Tc_m,
\end{array}
\end{equation}
where 
\begin{equation}\label{eq:T_m2}
\Tc_m :=   L^2\sum_{t=1}^m \eta_t \sum_{i=0}^{t-1}\omega_{i,t}\eta_i^2\Exp{\norms{v_i}^2} - \sum_{t=0}^m\big(\eta_t -  L\eta_t^2\big)\Exp{\norms{v_t}^2},
\end{equation}
and $\omega_t$, $\omega_{i,t}$, and $s_t$ are defined in Lemma~\ref{le:upper_bound_new}.

If we fix $\beta_t = \beta \in (0, 1)$, then we can show that $\omega_t = \beta^{2t}$, $\omega_{i,t} = \beta^{2(t-i)}$, and $s_t =  (1-\beta)^2\Big[\frac{1-\beta^{2t}}{1-\beta^2}\Big] < \frac{1-\beta}{1+\beta}$ as in the proof of Theorem~\ref{th:singe_loop_const_step}.

Now, let $u_i^2 := \Exp{\norms{v_i}^2}$.
To bound the quantity $\Tc_m$ defined by \eqref{eq:T_m}, we note that
\begin{equation*}
\begin{array}{ll}
\displaystyle\sum_{t=1}^m\eta_t\sum_{i=0}^{t-1}\beta^{2(t-i)}\eta_i^2u_i^2 &= \eta_1\displaystyle\sum_{i=0}^{0}\beta^{2(1-i)}\eta_i^2u_i^2 + \eta_2\sum_{i=0}^{1}\beta^{2(2-i)}\eta_i^2u_i^2 \vspace{1ex}\\
& +{~} \eta_3\displaystyle\sum_{i=0}^{2}\beta^{2(3-i)}\eta_i^2u_i^2 + \cdots + \eta_m\displaystyle\sum_{i=0}^{m-1}\beta^{2(m-i)}\eta_i^2u_i^2 \vspace{1ex}\\
&= \beta^2\eta_1\eta_0^2u_0^2 + \eta_2\big[\beta^4\eta_0^2u_0^2 + \beta^2\eta_1^2u_1^2\big] \vspace{1ex}\\
& + {~} \eta_3\big[\beta^6\eta_2^2u_0^2 + \beta^4\eta_1^2u_1^2 + \beta^2\eta_2^2u_2^2\big] + \cdots \vspace{1ex}\\
& + {~} \eta_m\big[\beta^{2m}\eta_0^2u_0^2 + \beta^{2(m-1)}\eta_1^2u_1^2 + \cdots + \beta^2\eta_{m-1}^2u_{m-1}^2\big] \vspace{1ex}\\
&= \beta^2\eta_0^2\big[\eta_1 + \beta^2\eta_2 + \cdots + \beta^{2(m-1)}\eta_m\big]u_0^2 \vspace{1ex}\\
& + {~} \beta^2\eta_1^2\big[\eta_2 + \beta^2\eta_3 + \cdots + \beta^{2(m-2)}\eta_{m}\big]u_1^2 + \cdots \vspace{1ex}\\ 
&+ {~} \beta^2\eta_{m-2}^2\big[\eta_{m-1} + \beta^2\eta_{m}\big]u_{m-2}^2 + \beta^2\eta_{m-1}^2\eta_m u_{m-1}^2. 
\end{array}
\end{equation*}
Using this expression, we can write $\Tc_m$ from  \eqref{eq:T_m} as
\begin{equation*} 
\begin{array}{ll}
\Tc_m &= \eta_0\Big[L^2\beta^2\eta_0\big[\eta_1 + \beta^2\eta_2 + \cdots + \beta^{2(m-1)}\eta_m\big]  - (1-L\eta_0)\Big]u_0^2 \vspace{1ex}\\
& + {~}  \eta_1\Big[ L^2\beta^2\eta_1\big[\eta_2 + \beta^2\eta_3 + \cdots + \beta^{2(m-2)}\eta_{m}\big] - (1-L\eta_1)\Big] + \cdots \vspace{1ex}\\
&+ {~}  \eta_{m-1}\Big[ L^2\beta^2\eta_{m-1}\eta_m - (1-L\eta_{m-1})\Big]u_{m-1}^2 - \eta_m(1-L\eta_m)u_{m}^2.
\end{array}
\end{equation*}
To guarantee $\Tc_m \leq 0$, from the last expression of $\Tc_m$, we can impose the following condition:
\begin{equation}\label{eq:cond_of_eta3}
\left\{\begin{array}{ll}
L^2\beta^2\eta_0\big[\eta_1 + \beta^2\eta_2 + \cdots + \beta^{2(m-1)}\eta_m\big]  - (1-L\eta_0) &= 0\vspace{1ex}\\
L^2\beta^2\eta_1\big[\eta_2 + \beta^2\eta_3 + \cdots + \beta^{2(m-2)}\eta_{m}\big] - (1-L\eta_1)&= 0\vspace{1ex}\\
\cdots & \cdots \vspace{1ex}\\
L^2\beta^2\eta_{m-1}\eta_m - (1-L\eta_{m-1}) &= 0\vspace{1ex}\\
- (1 - L\eta_{m}) &= 0.
\end{array}\right.
\end{equation}
The condition \eqref{eq:cond_of_eta3} leads to the following update of $\eta_t$:
\begin{equation*} 
\eta_m := \frac{1}{L},~~~\text{and}~~\eta_t := \frac{1}{L + L^2\big[\beta^2\eta_{t+1} + \beta^4\eta_{t+2} + \cdots + \beta^{2(m-t)}\eta_m\big]},~~~t=0,\cdots, m-1,
\end{equation*}
which is exactly \eqref{eq:update_of_eta_t}.

Next, note that $\beta^2 = \Big(1 - \frac{c_1}{\sqrt{b(m+1)}}\Big)^2 = 1 - \frac{2c_1}{\sqrt{b(m+1)}} + \frac{c_1^2}{b(m+1)}$.
Therefore, $1-\beta^2 = \frac{2c_1}{\sqrt{b(m+1)}} - \frac{c_1^2}{b(m+1)} \geq \frac{c_1}{\sqrt{b(m+1)}}$, which implies $\sqrt{1-\beta^2} \geq \frac{\sqrt{c_1}}{(b(m+1))^{1/4}}$.
Using $\sqrt{1-\omega} = \sqrt{1-\beta^2} \geq \frac{\sqrt{c_1}}{(b(m+1))^{1/4}}$ into \eqref{eq:stepsize_pros} of Lemma~\ref{le:adaptive_step_size}, we can show that $\Sigma_m \geq \frac{\sqrt{c_1}(m+1)^{3/4}}{2\nhan{L}b^{1/4}}$ as in the first statement (a) of Theorem~\ref{th:singe_loop_adapt_step}.

Note that $\omega_t = \beta^{2t}$, by the Chebyshev sum inequality, we have
\begin{equation*}
\sum_{t=0}^m\omega_t\eta_t = \sum_{t=0}^m\beta^{2t}\eta_t \leq \frac{\Sigma_m}{(m+1)}(1 + \beta^2 + \cdots + \beta^{2m}) \leq  \frac{\Sigma_m}{(m+1)(1-\beta^2)}.
\end{equation*}
Utilizing this estimate, $\Exp{\norms{v_0 - \nabla{f}(x_0)}^2}  \leq \frac{\sigma^2}{b}$, and $S_t \leq \frac{(1-\beta)\sigma^2}{1+\beta}$  into \eqref{eq:key_estimate_300}, and noting that $\Tc_m \leq 0$, we have
\begin{equation*} 
\displaystyle\frac{1}{2}\displaystyle\sum_{t=0}^m\eta_t\Exp{\norms{\nabla{f}(x_t)}^2} \leq f(x_0) -  \Exp{f(x_{m+1})} + \frac{\Sigma_m\sigma^2}{2(1-\beta^2)b(m+1)} +  \frac{(1-\beta)\sigma^2}{2(1+\beta)}\Sigma_m.
\end{equation*}
Since $\Exp{f(x_{m+1})} \geq f^{\star}$, using this into the last estimate, and multiplying the result by $\frac{2}{\Sigma_m}$, we obtain
\begin{equation}\label{eq:key_estimate_301}
\displaystyle\frac{1}{\Sigma_m}\displaystyle\sum_{t=0}^m\eta_t\Exp{\norms{\nabla{f}(x_t)}^2} \leq \frac{2}{\Sigma_m}[f(x_0) -  f^{\star}]  + \frac{\sigma^2}{(1+\beta)}\left[\frac{1}{b(m+1)(1-\beta)} +  (1-\beta)\right].
\end{equation}
Since $\left[\frac{1}{b(m+1)(1-\beta)} +  (1-\beta)\right] = \left(c_1 + \frac{1}{c_1}\right)\frac{1}{\sqrt{b(m+1)}}$ for $\beta = 1- \frac{c_1}{\sqrt{b(m+1)}}$, \eqref{eq:key_estimate_301} leads to
\begin{equation}\label{eq:key_estimate_302}
\displaystyle\frac{1}{\Sigma_m}\displaystyle\sum_{t=0}^m\eta_t\Exp{\norms{\nabla{f}(x_t)}^2} \leq \frac{4Lb^{1/4}}{\sqrt{c_1}(m+1)^{3/4}}\left[f(x_0) -  f^{\star}\right]  +  \left(c_1 + \frac{1}{c_1}\right)\frac{\sigma^2}{\sqrt{b(m+1)}}.
\end{equation}
The second statement (b) of Theorem~\ref{th:singe_loop_adapt_step} is proved similarly as in Theorem~\ref{th:singe_loop_const_step} using \eqref{eq:key_estimate_302}, and we omit the details.
\Eproof

\beforesubsec
\subsection{The proof of Theorem~\ref{th:double_loop_convergence}: Double-loop with constant step-size}\label{apdx:th:double_loop_convergence}
\aftersubsec
Similar to the proof of \eqref{eq:single_loop}  in Theorem \ref{th:singe_loop_const_step}, we have
\begin{equation}\label{eq:single_loop2} 
\displaystyle\sum_{t=0}^m\Exp{\norms{\nabla{f}(x_t^{(s)})}^2}  \leq  \frac{2}{\eta}\Big[\Exp{f(x_0^{(s)})} - \Exp{f(x_{m+1}^{(s)})}\Big]   + \frac{2(m+1)\sigma^2}{\sqrt{b(m+1)}},
\end{equation}
where we use the superscript $s$ to indicate the stage $s$ in Algorithm~\ref{alg:A2}.
Summing up this inequality from $s=1$ to $s = S$, and then multiplying the result by $\frac{1}{(m+1)S}$ and using $\Exp{f(x_{m+1}^{(S)})} \geq f^{\star} > -\infty$, we get
\begin{equation}\label{eq:est5d_1}
\begin{array}{ll}
\frac{1}{S(m+1)}\displaystyle\sum_{s=1}^S\sum_{t=0}^m\Exp{\norms{\nabla{f}(x_t^{(s)})}^2} & \leq \frac{2}{\eta S(m+1)}\big[f(\tilde{x}^0) - f^{\star}\big] + \frac{2 \sigma^2}{\sqrt{b (m+1)}} \vspace{1ex}\\
& \leq \frac{3 L b^{1/4}}{S (m+1)^{3/4}} \big[f(\tilde{x}^0) - f^{\star}\big] + \frac{2 \sigma^2}{\sqrt{b (m+1)}}.
\end{array}
\end{equation}
Here, we use the fact that $\eta \geq  \frac{2}{3\nhan{L}\big[ b(m+1)\big]^{1/4}}$ from \eqref{eq:LR_lower_bound} in the last inequality.

If we choose $b := \frac{c_1\sigma^2}{\varepsilon^2}$ and $m + 1 :=  \frac{c_2\sigma^2}{\varepsilon^2}$ for some constants $c_1 > 0$ and $c_2 > 0$ and $c_1c_2 > 4$, then, from \eqref{eq:est5d_1}, to guarantee $\frac{1}{S(m+1)}\displaystyle\sum_{s=1}^S\sum_{t=0}^m\Exp{\norms{\nabla{f}(x_t^{(s)})}^2} \leq \varepsilon^2$, we require
\begin{align*}
	\frac{3 L b^{1/4}}{S (m+1)^{3/4}} \big[f^0 - f^{\star}\big] + \frac{2 \sigma^2}{\sqrt{b (m+1)}} &= \frac{3 L c_1^{1/4} \sigma^{1/2}}{\varepsilon^{1/2}} \cdot \frac{\varepsilon^{3/2}}{S c_2^{3/4} \sigma^{3/2}} \big[f^0 - f^{\star}\big] + \frac{2 \sigma^2 \varepsilon^2} {\sigma^2 \sqrt{c_1 c_2}} = \varepsilon^2 \\
	&  \Leftrightarrow \frac{3 L c_1^{1/4} \varepsilon}{S c_2^{3/4} \sigma} \big[f^0 - f^{\star}\big] = \left( 1 - \frac{2}{\sqrt{c_1 c_2}} \right) \varepsilon^2 \\
	& \Leftrightarrow S = \frac{3 L c_1^{1/4} \big[f^0 - f^{\star}\big]}{c_2^{3/4} \sigma \left( 1 - \frac{2}{\sqrt{c_1 c_2}} \right) \varepsilon}.
\end{align*}
Consequently,  the total complexity is
\begin{equation*}
\begin{array}{ll}
\Tc_{ge} &:= (b + 3m)S = (c_1 + 3 c_2) \frac{\sigma^2}{\varepsilon^2} \frac{3 L c_1^{1/4} \big[f^0 - f^{\star}\big]}{c_2^{3/4} \sigma \left( 1 - \frac{2}{\sqrt{c_1 c_2}} \right) \varepsilon} \vspace{1ex}\\
&= \frac{3 L (c_1 + 3 c_2)c_1^{1/4} \big[f^0 - f^{\star}\big]\sigma}{c_2^{3/4}\left( 1 - \frac{2}{\sqrt{c_1 c_2}} \right) \varepsilon^3}  = \BigO{\frac{\sigma}{\varepsilon^3}}. 
\end{array}
\end{equation*}
Since we choose $b := \frac{c_1\sigma^2}{\varepsilon^2}$, the final complexity is $\BigO{\max\set{\frac{\sigma}{\varepsilon^3}, \frac{\sigma^2}{\varepsilon^2}}}$, where other constants independent of $\sigma$ and $\varepsilon$ are hidden.
\Eproof

\beforesec
\section{Appendix: The convergence analysis of the mini-batch variants}\label{apdx:sec:mini_batch}
\aftersec
In this supplementary document, we provide a full analysis of the mini-batch variants of Algorithm~\ref{alg:A1} and Algorithm~\ref{alg:A2}.

\beforesubsec
\subsection{Variance bound of mini-batch hybrid estimators}\label{apdx:le:upper_bound_new_batch}
\aftersubsec
For $\hat{v}_t$ defined by \eqref{eq:vhat_t}, we have the following property.

\begin{lemma}\label{le:key_pro_of_vhat_t}
The mini-batch gradient estimator $\hat{v}_t$ defined by \eqref{eq:vhat_t} satisfies
\begin{equation}\label{eq:key_pro_of_vhat_t}
{\!\!\!\!\!}\begin{array}{ll}
\Exps{(\Bc_t, \hat{\Bc}_t)}{\norms{\hat{v}_t - \nabla{f}(x_t)}^2} &=  \beta_{t-1}^2\norms{\hat{v}_{t-1}  - \nabla{f}(x_{t-1})}^2 - \rho\beta_{t-1}^2\norms{\nabla{f}(x_{t-1}) - \nabla{f}(x_t)}^2 \vspace{1ex}\\
& + {~} \rho\beta_{t-1}^2\Exps{\xi}{\norms{\nabla{f}(x_t;\xi) - \nabla{f}(x_{t-1};\xi)}^2} \vspace{1ex}\\
& + {~} (1-\beta_{t-1})^2\rho\sigma^2,
\end{array}{\!\!\!\!\!}
\end{equation}
where $\rho = \rho(\hat{b}) := \frac{n - \hat{b}}{(n-1)\hat{b}}$ if $n := \vert\Omega\vert$ is finite, and $\rho(\hat{b}) := \frac{1}{\hat{b}}$, otherwise.
\end{lemma}

\begin{proof}
Let $\hat{v}_t$ be defined by \eqref{eq:vhat_t}.
Let $z_t := \frac{1}{b_t}\sum_{i\in\Bc_t}(\nabla{f}_{\xi_i}(x_t) - \nabla{f}_{\xi_i}(x_{t-1}))$, $\bar{z} := \nabla{f}(x_t) - \nabla{f}(x_{t-1})$, $\Delta_t := \hat{v}_t - \nabla{f}(x_t)$, and $\Delta{u}_t := u_t - \nabla{f}(x_t)$.
Clearly, we have 
\begin{equation*}
\Exp{z_t} = \bar{z} ~~~\text{and}~~~\Exp{\Delta{u}_t} = 0.
\end{equation*}
Moreover, we can rewrite $\hat{v}_t$ in \eqref{eq:vhat_t} as
\begin{equation*}
\Delta_t  = \beta_{t-1}\Delta_{t-1} + \beta_{t-1}z_t + (1-\beta_{t-1})\Delta{u}_t - \beta_{t-1}\bar{z}.
\end{equation*}
Therefore, using these two expressions, we can derive
\begin{equation}\label{eq:bound_v_hat_t}
\begin{array}{ll}
\Exp{\norms{\Delta_t}^2}  &= \beta_{t-1}^2\norms{\Delta_{t-1}}^2 +  \beta_{t-1}^2\Exp{\norms{z_t}^2} + (1-\beta_{t-1})^2\Exp{\norms{\Delta{u}_t}^2}  + \beta_{t-1}^2\norms{\bar{z}}^2\vspace{1ex}\\
& + {~} 2\beta_{t-1}^2\iprods{\Delta_{t-1},\Exp{z_t}} + 2\beta_{t-1}(1-\beta_{t-1})\iprods{\Delta_{t-1}, \Exp{\Delta{u}_t}} - 2\beta_{t-1}^2\iprods{\Delta_{t-1},\bar{z}} \vspace{1ex}\\
& + {~} 2\beta_{t-1}(1-\beta_{t-1})\Exp{\iprods{z_t, \Delta{u}_t}} - 2\beta_{t-1}^2\iprods{\Exp{z_t}, \bar{z}} - 2\beta_{t-1}(1-\beta_{t-1})\iprods{\Exp{\Delta{u}_t}, \bar{z}} \vspace{1ex}\\
&=  \beta_{t-1}^2\norms{\Delta_{t-1}}^2 +  \beta_{t-1}^2\Exp{\norms{z_t}^2} + (1-\beta_{t-1})^2\Exp{\norms{\Delta{u}_t}^2}  - \beta_{t-1}^2\norms{\bar{z}}^2.
\end{array}
\end{equation}
For the finite-sum case, after a few elementary calculations, we can show that
\begin{equation*}
\Exp{\norms{z_t}^2} = \frac{n(b_t-1)}{(n-1)b_t}\norms{\bar{z}}^2 + \frac{(n-b_t)}{(n-1)b_t}\Exps{\xi}{\norms{\nabla{f}_{\xi}(x_t) - \nabla{f}_{\xi}(x_{t-1})}^2}. 
\end{equation*}
For the expectation case, we have
\begin{equation*}
\Exp{\norms{z_t}^2} =  \big(1- \frac{1}{b_t}\big)\norms{\bar{z}}^2 + \frac{1}{b_t}\Exps{\xi}{\norms{\nabla{f}_{\xi}(x_t) - \nabla{f}_{\xi}(x_{t-1})}^2}. 
\end{equation*}
In addition, under Assumption~\ref{as:A1}$($c$)$, we have $\Exp{\norms{\Delta{u}_t}^2} \leq \rho\sigma^2$.

Substituting one of the two last expressions and the bound of $\Exp{\norms{\Delta{u}_t}^2}$ into \eqref{eq:bound_v_hat_t}, we get \eqref{eq:key_pro_of_vhat_t}.
\end{proof}

The following analysis is given under fixed mini-batch sizes when we choose $\hat{b}_t = \tilde{b}_t = \hat{b}$.
Similar to Lemma \ref{le:upper_bound_new}, we can bound the variance $\Exp{\|\hat{v}_t - \nabla f(x_t)\|^2}$ of the mini-batch hybrid estimator $\hat{v}_t$ from \eqref{eq:vhat_t} in the following lemma.

\begin{lemma}\label{le:upper_bound_new_batch}
Assume that $f(\cdot,\cdot)$ is $L$-smooth and $u_t$ is an SGD estimator, $\hat{v}$ is given in \eqref{eq:vhat_t}, $\Bc_t$ and $\hat{\Bc}_t$ are mini-batches of the size $\hat{b}$.
Then, we have the following upper bound on the variance $\Exp{\norms{\hat{v}_t - \nabla{f}(x_t)}^2}$:
\begin{equation}\label{eq:vt_variance_bound_new_batch}
\Exp{\|\hat{v}_t - \nabla f(x_t)\|^2} \le \omega_t \Exp{\|\hat{v}_0 - \nabla f(x_0)\|^2} + L^2\rho \sum_{i=0}^{t-1}\omega_{i,t} \Exp{\|x_{i+1} - x_{i}\|^2} + \rho S_t,
\end{equation}
where the expectation is taking over all the randomness $\Fc_t := \sigma(v_0, v_1, \cdots, v_t)$, 
$\omega_{t} := \prod_{i=1}^{t}\beta_{i-1}^2$, $\omega_{i, t} := \prod_{j=i+1}^{t}\beta_{j-1}^2$ for $i=0,\cdots, t$, and $S_{t} := \sum_{i=0}^{t-1}\big(\prod_{j=i+2}^{t}\beta_{j-1}^2\big)(1-\beta_i)^2\sigma$ for $t \geq 0$. $\rho = \frac{n - \tilde{b}}{\tilde{b}(n-1)}$ if $|\Omega|$ is finite and $\rho = \frac{1}{\tilde{b}}$ otherwise.
\end{lemma}

\begin{proof}
From Lemma~\ref{le:key_pro_of_vhat_t}, taking the expectation with respect to $\mathcal{F}_t := \sigma(v_0,v_1,\cdots, v_t)$, we have
\begin{equation*}
\begin{array}{ll}
\Exp{\norms{\hat{v}_t - \nabla{f}(x_t)}^2} &\leq  \beta_{t-1}^2\Exp{\norms{\hat{v}_{t-1}  - \nabla{f}(x_{t-1})}^2} \vspace{1ex}\\
& + {~} L^2\rho\beta_{t-1}^2 \Exp{\|x_t - x_{t-1}\|^2} + \rho(1-\beta_{t-1})^2\sigma^2.
\end{array}
\end{equation*}
Let $a_t^2 := \Exp{\norms{\hat{v}_t - \nabla{f}(x_t)}^2} $ and $r_t^2 = \Exp{\|x_{t +1}- x_{t}\|^2}$.
By following inductive step as in the proof of Lemma \ref{le:upper_bound_new}, we obtain
\begin{align*}
a_t^2 &\le \left(\beta_{t-1}^2\cdots \beta_0^2 \right)a_0^2 + L^2\rho\left(\beta_{t-1}^2\cdots \beta_0^2\right)r_0^2 + \cdots + L^2\rho\beta_{t-1}^2r_{t-1}^2\\
&+ {~} \rho\left[\left(\beta_{t-1}^2\cdots\beta_1^2\right)(1-\beta_0)^2  + \cdots + (1-\beta_{t-1})^2 \right]\sigma^2.
\end{align*}
Using the definition of $\omega_{t}$, $\omega_{i, t}$, and $S_{t}$ in Lemma \ref{le:upper_bound_new}, the previous inequality becomes
\begin{align*}
a_t^2 &\le \omega_{t}a_0^2 + L^2\rho\sum_{i=0}^{t-1}\omega_{i,t}r_i^2 + \rho S_t,
\end{align*}
which is the same as \eqref{eq:vt_variance_bound_new_batch}.
\end{proof}

\beforesubsec
\subsection{The proof of Corollary~\ref{co:mini_batch}: Single loop with constant step-size and mini-batches}\label{apdx:co:mini_batch}
\aftersubsec
Using Lemma~\ref{le:upper_bound_new_batch} and following the same path of proof of Lemma~\ref{le:key_estimate_of_convergence}, we can show that
\begin{equation}\label{eq:co4_est1}
\begin{array}{ll}
\Exp{f(x_{m+1}} &\leq  \Exp {f(x_0)} - \displaystyle\sum_{t=0}^m\frac{\eta}{2}\Exp{\norms{\nabla{f}(x_t)}^2}  + \frac{\eta}{2} \left(\displaystyle\sum_{t=0}^m\omega_t\right)\Exp{\norms{\hat{v}_0 - \nabla{f}(x_0)}^2} \vspace{1ex}\\
& + {~} \frac{\rho\eta}{2}\displaystyle\sum_{t=0}^mS_t + \frac{1}{2}\hat{\Tc}_m, 
\end{array}
\end{equation}
where
\begin{equation*} 
\widehat{\mathcal{T}}_m := \rho L^2\eta^3\sum_{t=0}^m \sum_{i=0}^{t-1}\omega_{i,t}\Exp{\norms{\hat{v}_i}^2} -  \eta\sum_{t=0}^m\big(1 - L\eta\big)\Exp{\norms{\hat{v}_t}^2}.
\end{equation*}
Clearly, we can rewrite $\widehat{\mathcal{T}}_m$ as
\begin{equation*} 
\begin{array}{ll}
\widehat{\mathcal{T}}_m &= \eta\Big[\displaystyle\frac{\beta^2(1-\beta^{2m})L^2\eta^2\rho}{1-\beta^2} - (1-L\eta)\Big]q_0^2 \vspace{1ex}\\
& +  \eta\Big[\frac{\beta^2(1-\beta^{2(m-1)})L^2\eta^2\rho}{1-\beta^2} - (1-L\eta)\Big]q_1^2 + \cdots \vspace{1ex}\\
&+ {~}  \eta\Big[\displaystyle\frac{\beta^2(1-\beta^{2})L^2\eta^2\rho}{1-\beta^2} - (1-L\eta)\Big]q_{m-1}^2 - \eta(1-L\eta)q_{m}^2,
\end{array}
\end{equation*}
where $q_t^2 := \Exp{\norms{\hat{v}_t}^2}$.
To guarantee $\widehat{\mathcal{T}}_m \le 0$, we need to have
\begin{equation*}
\left\{\begin{array}{ll}
\displaystyle\frac{L^2\eta^2\rho\beta^2(1-\beta^{2m})}{1-\beta^2} - (1-L\eta) &\leq 0\vspace{1ex}\\
\displaystyle\frac{L^2\eta^2\rho\beta^2(1-\beta^{2(m-1)})}{1-\beta^2} - (1-L\eta) &\leq 0\vspace{1ex}\\
\cdots & \cdots \vspace{1ex}\\
\displaystyle\frac{L^2\eta^2\rho\beta^2(1-\beta^2)}{1-\beta^2} - (1-L\eta) &\leq 0\vspace{1ex}\\
- (1 - L\eta) &\leq 0.
\end{array}\right.
\end{equation*}
Let $\alpha^2_m := \frac{\beta^2(1-\beta^{2m})}{1 - \beta^2}$. 
Since $\alpha^2_1 < \alpha^2_2 < \dots < \alpha^2_m$,  the last condition holds if $L^2\eta^2\rho\alpha_m^2 - (1 - L\eta) \le 0$. 
By tightening this condition, we obtain
\begin{equation*}
\eta := \frac{2}{L\left(1 + \sqrt{1 + 4\rho\alpha_m^2}\right)}~\text{with}~\alpha^2_m := \frac{\beta^2(1-\beta^{2m})}{1 - \beta^2},
\end{equation*}
which is exactly \eqref{eq:step_size_batch}.
Since $\alpha_m^2 \le \frac{\beta^2}{1 - \beta^2}$, we have $\eta \geq \underline{\eta}: = \frac{2\sqrt{1-\beta^2}}{L(\sqrt{1-\beta^2} + \sqrt{1 + \beta^2(4\rho - 1)})}$.

Next, we can reuse the following estimates as in the proof of Theorem~\ref{th:singe_loop_const_step}:
\begin{equation*} 
\begin{array}{rl}
\displaystyle\sum_{t=0}^mS_t &\le \displaystyle\frac{\sigma^2(1-\beta)(m+1)}{1 +\beta}\\
\displaystyle\sum_{t=0}^m \omega_t &= \displaystyle\frac{1 - \beta^{2(m+1)}}{1 - \beta^2} \le \frac{1}{1 - \beta^2}.
\end{array}
\end{equation*}
Combining these estimate into \eqref{eq:co4_est1} and notting that $\widehat{\Tc}_m \leq 0$ and $\Exp{\norms{v_0 - \nabla{f}(x_0)}^2} \leq \frac{\sigma^2}{b}$, we can show that
\begin{equation}\label{eq:est5c_1_batch}
\begin{array}{ll}
\Exp{f(x_{m+1})} &\overset{\tiny\eqref{eq:bound_of_st_omegat}}{\leq} \Exp{f(x_0)} - \frac{\eta}{2}\displaystyle\sum_{t=0}^m\Exp{\norms{\nabla{f}(x_t)}^2} \vspace{1ex}\\
& + {~} \frac{\eta\sigma^2}{2(1+\beta)}\Big[\frac{1}{(1-\beta)b} + \rho(1-\beta)(m+1)\Big].
\end{array}
\end{equation}
Note that $\Exp{f(x_{m+1})} \ge f^{\star} > -\infty$, \eqref{eq:est5c_1_batch} can be rewritten as
\begin{equation}\label{eq:main_estimate_for_single_loop_batch1}
 \frac{\eta}{2}\displaystyle\sum_{t=0}^m\Exp{\norms{\nabla{f}(x_t)}^2}  \leq \Exp{f(x_0)} - f^{\star} + \frac{\eta\sigma^2}{2(1+\beta)}\Big[\frac{1}{(1-\beta)b} + \rho(1-\beta)(m+1)\Big].
\end{equation}
If we choose $\beta := 1 - \frac{c_1}{\sqrt{\hat{\rho}b(m+1)}}$ for any $0 < c_1 < \sqrt{b(m+1)}$ such that
\begin{equation*}
\frac{1}{(1-\beta)b(m+1)} + \rho(1-\beta) = \left(c_1 + \frac{1}{c_1}\right)\sqrt{\frac{\rho}{b(m+1)}},
\end{equation*} 
then \eqref{eq:main_estimate_for_single_loop_batch1} leads to
\begin{equation}\label{eq:single_loop_batch} 
\begin{array}{ll}
 \frac{1}{m+1}\displaystyle\sum_{t=0}^m\Exp{\norms{\nabla{f}(x_t)}^2}  &\leq  \frac{L}{(m+1)}\left(\frac{\sqrt{1-\beta^2} + \sqrt{1+\beta^2(4\rho - 1)}}{\sqrt{1-\beta^2}}\right)\Big[\Exp{f(x_0)} - f^{\star}\Big]  \vspace{1ex}\\
 & + {~}  \left(c_1 + \frac{1}{c_1}\right)\frac{\sigma^2}{(1+\beta)}\sqrt{\frac{\rho}{b(m+1)}}.
 \end{array}
\end{equation}
Since $\beta = 1 - \frac{c_1}{\sqrt{\rho b(m+1)}} < 1$ and if we choose $b$, $\hat{b}$, and $m$ such that $\rho b(m+1) > c_1^2$, we have 
\begin{equation*}
\begin{array}{ll}
1 - \beta^2 &= 1 - \left(1 - \frac{c_1}{\sqrt{\rho b(m+1)}} \right)^2 = \frac{2c_1}{\sqrt{\rho b(m+1)}} - \frac{c_1^2}{\rho b(m+1)} = \frac{2c_1\sqrt{\rho b(m+1)} - c_1^2}{\rho b(m+1)} > \frac{2c_1}{\sqrt{\rho b(m+1)}},
\end{array}
\end{equation*}
and $\sqrt{1 - \beta^2} + \sqrt{1+\beta^2(4\rho - 1)} \leq 1 + \sqrt{1+3\beta^2} \leq 3$ since $\rho \leq 1$. 
Therefore, we can bound $\eta$ as
\begin{equation*} 
\eta \geq \underline{\eta} \geq  \frac{2c_1}{3L\big[ \rho b(m+1)\big]^{1/4}}.
\end{equation*} 
Therefore, the inequality \eqref{eq:single_loop_batch} can be rewritten as
\begin{equation*} 
\begin{array}{ll}
\frac{1}{m+1}\displaystyle\sum_{t=0}^m\Exp{\norms{\nabla{f}(x_t)}^2} &\leq \frac{3 L(\rho b)^{1/4}}{2c_1(m+1)^{3/4}}\left(\Exp{f(x_0)} - f^{\star}\right) +  \left(c_1 + \frac{1}{c_1}\right)\frac{\sigma^2}{(1+\beta)}\sqrt{\frac{\rho}{b(m+1)}}.\\
&\leq \frac{3 L(\rho b)^{1/4}}{2c_1(m+1)^{3/4}}\left(\Exp{f(x_0)} - f^{\star}\right) +  \left(c_1 + \frac{1}{c_1}\right)\frac{\sigma^2}{2}\sqrt{\frac{\rho}{b(m+1)}}.
\end{array}
\end{equation*}
Let $f^0 := \Exp{f(x_0)}$.
We can write the  bound as
\begin{equation*} 
\frac{1}{m+1}\displaystyle\sum_{t=0}^m\Exp{\norms{\nabla{f}(x_t)}^2} \leq \frac{3 L(\rho b)^{1/4}}{2c_1(m+1)^{3/4}}\left( f^0 - f^{\star}\right) +  \left(c_1 + \frac{1}{c_1}\right)\frac{\sigma^2}{2}\sqrt{\frac{\rho}{b(m+1)}}.
\end{equation*}
Let us choose $b := c_2\sigma^{8/3}(\rho(m+1))^{1/3}$ for some $c_2 > 0$.
Then, the last inequality leads to 
\begin{equation}\label{eq:single_loop_final_batch} 
\frac{1}{m+1}\displaystyle\sum_{t=0}^m\Exp{\norms{\nabla{f}(x_t)}^2} \leq \frac{\rho^{1/3}\sigma^{2/3}}{(m+1)^{2/3}}\left[\frac{3 Lc_2^{1/4}}{2c_1}\left( f^0 - f^{\star}\right) +  \left(c_1 + \frac{1}{c_1}\right)\frac{1}{2\sqrt{c_2}}\right].
\end{equation}
From \eqref{eq:single_loop_final_batch}, to guarantee $\Exp{\norms{\nabla{f}(\widetilde{x}_m)}^2} \leq \varepsilon^2$, we need to choose 
\begin{equation*}
\frac{\rho^{1/3}\sigma^{2/3}}{(m+1)^{2/3}}\left[\frac{3 Lc_2^{1/4}}{2c_1}\left( f^0 - f^{\star}\right) +  \left(c_1 + \frac{1}{c_1}\right)\frac{1}{2\sqrt{c_2}}\right] \leq \varepsilon^2,
\end{equation*}
which leads to
\begin{equation*}
m + 1 \geq \frac{\rho^{1/2}\sigma}{\varepsilon^3}\left[\frac{3 Lc_2^{1/4}}{2{c_1}}\left( f^0 - f^{\star}\right) +  \left(c_1 + \frac{1}{c_1}\right)\frac{1}{2\sqrt{c_2}}\right]^{3/2}.
\end{equation*}
Hence, we can choose $m$ as in \eqref{eq:choice_of_m2}.

Finally, let $c_1 = 1$.
Then the number of stochastic gradient evaluations $\Tc_{ge}$ is
\begin{equation*}
\begin{array}{ll}
\Tc_{ge} &= b + 3\hat{b}m \leq b + \frac{3(m+1)}{\rho} \vspace{1ex}\\
&\leq c_2\sigma^{8/3} \left[\rho(m+1)\right]^{1/3} + \frac{3\sigma}{\rho^{1/2}\varepsilon^3}\left[\frac{3 Lc_2^{1/4}}{2{c_1}}\left( f^0 - f^{\star}\right) + \frac{1}{\sqrt{c_2}}\right]^{3/2}\vspace{1ex}\\
&\leq \frac{\rho^{1/2}\sigma^{3}}{\varepsilon}\left[\frac{3 Lc_2^{9/4}}{2{c_1}}\left( f^0 - f^{\star}\right) + c_2^{3/2}\right]^{1/2} + \frac{3\sigma}{\rho^{1/2}\varepsilon^3}\left[\frac{3 Lc_2^{1/4}}{2{c_1}}\left( f^0 - f^{\star}\right) + \frac{1}{\sqrt{c_2}}\right]^{3/2},
\end{array}
\end{equation*}
which proves \eqref{eq:T_ge3}, where $\rho \le \frac{1}{\hat{b}}$ if $\vert\Omega\vert$ is infinite and $\rho := \frac{n-\hat{b}}{\hat{b}(n-1)}$ if $\vert\Omega\vert$ is finite.
In particular, if $\vert\Omega\vert$ is infinite and we choose $\rho := \frac{c_3^2}{\sigma^2\varepsilon^2}$ for some $c_3 \leq \sigma\varepsilon$, then 
\begin{align*}
\Tc_{ge} &= \frac{c_3\sigma^{2}}{\varepsilon^2}\left[\frac{3 Lc_2^{9/4}}{2{c_1}}\left( f^0 - f^{\star}\right) + c_2^{3/2}\right]^{1/2} + \frac{3\sigma^2}{c_3\varepsilon^2}\left[\frac{3 Lc_2^{1/4}}{2{c_1}}\left( f^0 - f^{\star}\right) + \frac{1}{\sqrt{c_2}}\right]^{3/2}.
\end{align*}
Hence, we obtain $\Tc_{ge} = \BigO{\big(c_3 + \frac{1}{c_3}\big)\frac{\sigma^2}{\varepsilon^2}}$.
\Eproof

\beforesubsec
\subsection{The mini-batch variant of Algorithm~\ref{alg:A2} and its complexity}\label{apdx:th:double_loop_convergence_mini_batch}
\aftersubsec
Let us consider a mini-batch variant of Algorithm~\ref{alg:A2}.
Similar to Theorem~\ref{th:double_loop_convergence}, we can prove the following result.

\begin{corollary}\label{co:double_loop_convergence_minibatch}
Let $\sets{x^{(s)}_t}_{t=0\to m}^{s=1\to S}$ be the sequence generated by the mini-batch variant of Algorithm~\ref{alg:A2} using constant step-size $\eta$ defined in \eqref{eq:step_size_batch} with $c_1 := 1$.
Then, the following estimate holds
\begin{equation}\label{eq:double_loop_est2} 
\frac{1}{S(m+1)}\displaystyle\sum_{s=1}^S\sum_{t=0}^m\Exp{\norms{\nabla{f}(x_t^{(s)})}^2} \leq  \frac{3 L \rho(\hat{b}) b^{1/4}}{S (m+1)^{3/4}} \big[f(\tilde{x}^0) - f^{\star}\big] + \frac{2 \sigma^2\sqrt{\rho(\hat{b})}}{\sqrt{b (m+1)}}.
\end{equation}
Let $\widetilde{x}_T \sim \Ub(\sets{x^{(s)}_t}_{t=0\to m}^{s=1\to S})$.
If we choose $b := \frac{c_1\sigma^2}{\varepsilon^2}$ and $\frac{m + 1}{\hat{b}} :=  \frac{c_2\sigma^2}{\hat{b}^2\varepsilon^2}$ for some constants $c_1 > 0$ and $c_2 > 0$ and $c_1c_2 > 4$, then, to guarantee $\Exp{\norms{\nabla{f}(\widetilde{x}_T)}^2} \leq \varepsilon^2$, we require
\begin{equation}\label{eq:S_iterations2}
S := \frac{3 L c_1^{1/4} \big[f(\widetilde{x}^0) - f^{\star}\big]}{c_2^{3/4} \hat{b}^{3/2}\sigma \left( 1 - \frac{2}{\sqrt{c_1 c_2}} \right) \varepsilon}.
\end{equation}
Consequently,  the total number of stochastic gradient evaluations $\Tc_{ge}$ does not exceed 
\begin{equation}\label{eq:Toc_double_loop2}
\Tc_{ge} := \left(b + 3\lfloor\tfrac{m}{\hat{b}}\rfloor\right)S  = \frac{3 L (c_1 + 3 c_2)c_1^{1/4} \big[f(\widetilde{x}^0) - f^{\star}\big]\sigma}{c_2^{3/4}\hat{b}^{3/2}\left( 1 - \frac{2}{\sqrt{c_1 c_2}} \right) \varepsilon^3}  = \BigO{\frac{\sigma}{\varepsilon^3}}. 
\end{equation}
\end{corollary}

\begin{proof}
First, similar to the proof of \eqref{eq:single_loop_final_batch}, we have
\begin{equation*} 
\frac{1}{m+1}\displaystyle\sum_{t=0}^m\Exp{\norms{\nabla{f}(x_t^{(s)})}^2} \leq \frac{3 L(\rho b)^{1/4}}{(m+1)^{3/4}}\left( \Exp{f(x_0^{(s)})} - \Exp{f(x_{m+1}^{(s)})}\right) + 2\sigma^2\sqrt{\frac{\rho}{b(m+1)}}.
\end{equation*}
Summing up this inequality from $s = 1$ to $s = S$ and then using $\Exp{f(x_{m+1}^{(S)})} \geq f^{\star}$ and $\widetilde{x}_0 := x_0^{(1)}$, we can show that
\begin{equation*} 
\frac{1}{(m+1)S}\displaystyle\sum_{s=1}^S\sum_{t=0}^m\Exp{\norms{\nabla{f}(x_t^{(s)})}^2} \leq \frac{3 L(\rho b)^{1/4}}{S(m+1)^{3/4}}\left( \Exp{f(\widetilde{x}_0} - f^{\star}\right) + 2\sigma^2\sqrt{\frac{\rho}{b(m+1)}}.
\end{equation*}
If we choose $b := \frac{c_1\sigma^2}{\hat{b}^2\varepsilon^2}$ and $\frac{m + 1}{\hat{b}} :=  \frac{c_2\sigma^2}{\varepsilon^2}$ for some constants $c_1 > 0$ and $c_2 > 0$ and $c_1c_2 > 4$, then, $\rho(m+1) =  \frac{c_2\sigma^2}{\varepsilon^2}$, $b = \frac{c_1\rho^2\sigma^2}{\varepsilon^2}$, and 
from \eqref{eq:est5d_1}, to guarantee $\frac{1}{S(m+1)}\displaystyle\sum_{s=1}^S\sum_{t=0}^m\Exp{\norms{\nabla{f}(x_t^{(s)})}^2} \leq \varepsilon^2$, we require
\begin{align*}
	\frac{3 L (\rho b)^{1/4}}{S (m+1)^{3/4}} \big[f^0 - f^{\star}\big] + \frac{2 \sigma^2\sqrt{\rho}}{\sqrt{b (m+1)}} &= \frac{3 L c_1^{1/4}\rho^{3/4} \sigma^{1/2}}{\varepsilon^{1/2}} \cdot \frac{\rho^{3/4}\varepsilon^{3/2}}{S c_2^{3/4} \sigma^{3/2}} \big[f^0 - f^{\star}\big] + \frac{2\sigma^2 \varepsilon^2} {\sigma^2 \sqrt{c_1 c_2}} = \varepsilon^2 \\
	&  \Leftrightarrow \frac{3 L c_1^{1/4}\rho^{3/2} \varepsilon}{S c_2^{3/4} \sigma} \big[f^0 - f^{\star}\big] = \left( 1 - \frac{2}{\sqrt{c_1 c_2}} \right) \varepsilon^2 \\
	& \Leftrightarrow S = \frac{3 L \rho^{3/2} c_1^{3/4} \big[f^0 - f^{\star}\big]}{c_2^{3/4} \sigma \left( 1 - \frac{2}{\sqrt{c_1 c_2}} \right) \varepsilon}.
\end{align*}
Consequently,  the total complexity is
\begin{equation*}
\begin{array}{ll}
\Tc_{ge} &:= (b + 3\frac{m}{\hat{b}})S \leq (c_1 + 3 c_2) \frac{\sigma^2}{\varepsilon^2} \frac{3 L c_1^{1/4}\rho^{3/2} \big[f^0 - f^{\star}\big]}{c_2^{3/4} \sigma \left( 1 - \frac{2}{\sqrt{c_1 c_2}} \right) \varepsilon} \vspace{1ex}\\
&= \frac{3 L (c_1 + 3 c_2)c_1^{1/4} \big[f^0 - f^{\star}\big]\sigma}{c_2^{3/4}\hat{b}^{3/2}\left( 1 - \frac{2}{\sqrt{c_1 c_2}} \right) \varepsilon^3}  = \BigO{\frac{\sigma}{\varepsilon^3}}. 
\end{array}
\end{equation*}
Since we choose $b\hat{b}^2 := \frac{c_1\sigma^2}{\varepsilon^2}$ which shows that $b \leq \frac{c_1\sigma^2}{\varepsilon^2}$, the final complexity is $\BigO{\max\set{\frac{\sigma}{\varepsilon^3}, \frac{\sigma^2}{\varepsilon^2}}}$, where other constants independent of $\sigma$ and $\varepsilon$ are hidden.
\end{proof}

\beforesec
\section{Appendix: Additional numerical experiments}\label{apdx:subsec:experiments}
\aftersec
In this subsection, we provide more numerical examples on two examples we tested in the main text.

\beforesubsec
\subsection{Experiment setup}
\aftersubsec
\textbf{Our algorithms: } 
We implement the following variants of Algorithm~\ref{alg:A1} and Algorithm~\ref{alg:A2} in Python:
\begin{itemize}
\vspace{-1ex}
\item \textbf{Single-loop algorithms}: 
We consider different variants of the single-loop algorithm, Algorithm~\ref{alg:A1}.
We denote them by \texttt{Hybrid-SGD-SL} for constant step-size variants, and \texttt{Hybrid-SGD-ASL} for adaptive step-size variants.

\item \textbf{Double-loop algorithms}: 
These are variants of Algorithm~\ref{alg:A2}. 
We denote them by \texttt{Hybrid-SGD-DL[1-3]} the variants corresponding to different snapshot gradient batch-sizes of $b = n^{2/3}$, $b = 0.1n$, and $b = n$. 
We also denote \texttt{Hybrid-SGD-DL} as the best variants among these three choices of the batch-size for snapshot gradient.
\vspace{-1ex}
\end{itemize}
\textbf{Competitors:} 
We also compare our methods with the most state-of-the-art candidates from the literature.
We ignore other variants since their complexity bound is worse than ours and they use complicated routines for hyper-parameter selection.
\begin{itemize}
\vspace{-1ex}
\item Stochastic gradient descent (SGD): We test two variants of SGD. 
The first one, called SGD1, is with constant step-size $\eta_t := \frac{0.1}{L}$.
The second variant, called SGD2, is with an adaptive step-size of the form $\eta_t := \frac{\eta_0}{1 + \eta'\lfloor t/n\rfloor}$, where $\eta_0 > 0$ and $\eta' \geq 0$ are carefully tuned to obtain the best performance.
In our tests, we use $\eta_0 := \frac{0.1}{L}$ and $\eta' := 1$.
\item SVRG: This algorithmic variant is from \cite{nonconvexSVRG}, where its theoretical step-size in the single sample case is $\eta_t := \frac{1}{3nL}$, and in the mini-batch case is $\eta_t :=\frac{1}{3L} $.
\item SVRG+: This is a variant of SVRG studied in \cite{li2018simple}. Its theoretical step-size in the single sample case is $\eta_t := \frac{1}{6nL}$, and in the mini-batch case is $\eta_t := \frac{1}{6L}$.
\item SPIDER: SPIDER \cite{fang2018spider} is a stochastic gradient method using SARAH estimator (also called Stochastic Path-Integrated Differential EstimatoR). 
This method achieves the best-known complexity as Algorithm~\ref{alg:A2} but uses very different step-size $\eta_t := \min\set{\frac{\epsilon}{Ln_0\norm{v^k}},\frac{1}{2Ln_0}}$ where $n_0 = \frac{n^{1/2}}{\hat{b}}$ with $\hat{b}$ is a given mini-batch size in the range $[1,\sqrt{n}]$.
\item SpiderBoost: SpiderBoost \cite{wang2018spiderboost} is a modification of SPIDER by using a large constant step-size $\eta_t := \frac{1}{2L}$, but requires to set very specific mini-batch $\hat{b} = \lfloor\sqrt{n}\rfloor$ to achieve the best-known complexity as in Algorithm~\ref{alg:A2}.
\vspace{-1.5ex}
\end{itemize}
\textbf{Problems:}
We consider three examples: 
The first one is the logistic regression with non-convex regularizer as in \eqref{eq:exam1}.
The second example is a binary classification with non-convex loss as in \eqref{eq:exam2}.

\textbf{Datasets:}
All the datasets used in this paper are downloaded from LibSVM \cite{CC01a} at \\
\href{https://www.csie.ntu.edu.tw/~cjlin/libsvm/}{https://www.csie.ntu.edu.tw/~cjlin/libsvm/}.
We select 6 datasets:  \texttt{w8a} ($n=49,749, p=300$), \texttt{rcv1.binary} ($n=20,242, p = 47,236$),  \texttt{real-sim} $(n=72,309, p = 20,958$), \texttt{news20.binary} ($n=19,996, p = 1,355,191$), \texttt{url\_combined} ($n=2,396,130, p = 3,231,961$), and \texttt{epsilon} ($n=400,000, p = 2,000$).

\beforesubsec
\subsection{Logistic regression with non-convex regularizer}
\aftersubsec
In this section, we add more numerical examples to solve problem \eqref{eq:exam1}.
Together with the convergence of the trainning loss and gradient norms in Fig. \ref{fig:logistic_reg}, the training and test accuracies are also plotted in Fig. \ref{fig:logistic_acc1} for three datasets: \texttt{w8a}, \texttt{rcv1.binary}, and \texttt{real-sim}.

\begin{figure}[htp!]
\begin{center}
\includegraphics[width = 1\textwidth]{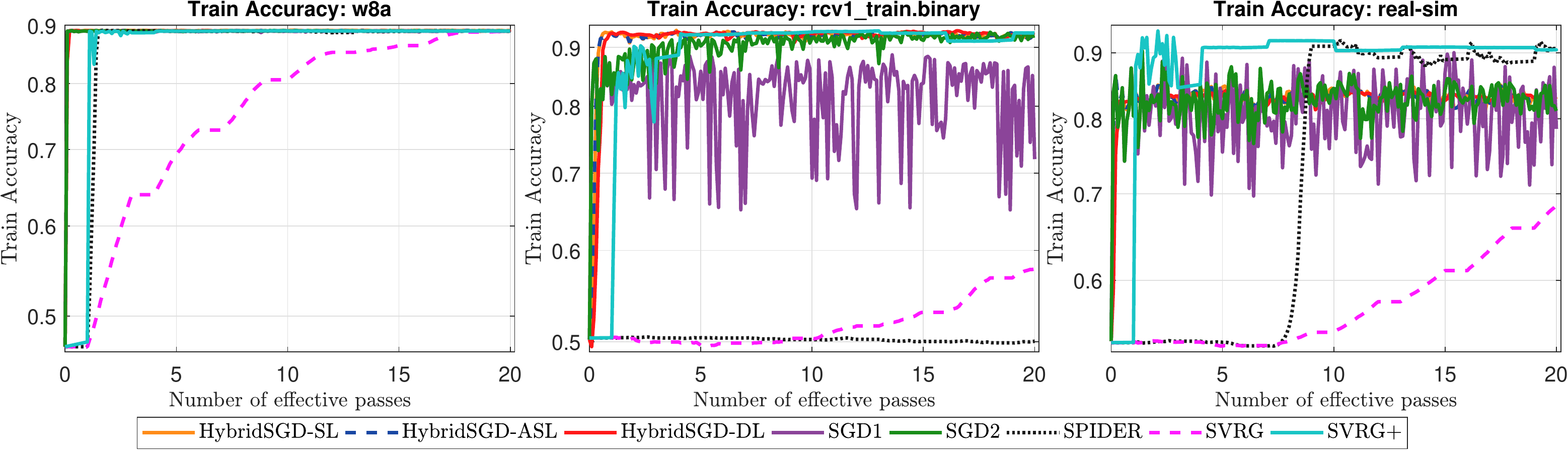}
\includegraphics[width = 1\textwidth]{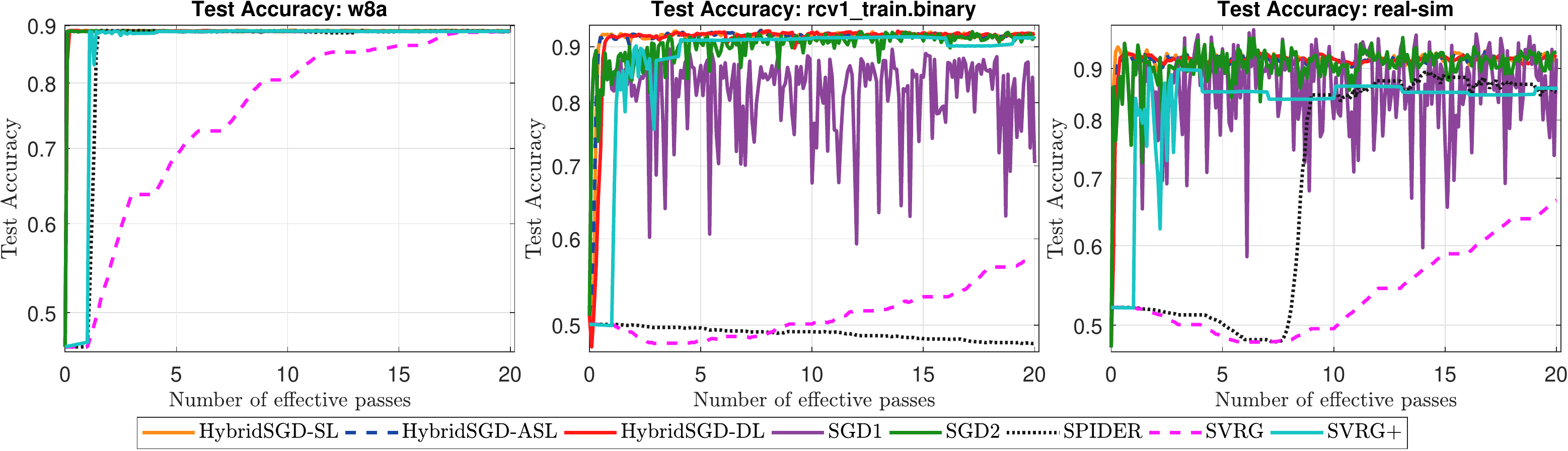}
\vspace{-1.5ex}
\caption{\done{The training and test accuracies of \eqref{eq:exam2} on three datasets: Single-sample case}.}\label{fig:logistic_acc1}
\end{center}
\vspace{-2ex}
\end{figure}

As we can observe from Fig.~\ref{fig:logistic_acc1}, for \texttt{w8a}, all the algorithms except for SVRG achieve similar training accuracy as well as test accuracy.
SVRG eventually reaches the same accuracy after around 17 epochs.
For \texttt{rcv1.binary}, HybridSGD variants, SGD2, and SVRG+ have similar training and test accuracies, but SGD2 is more oscillated than the other methods.
SGD1 performs worse than our methods in this case. Both SPIDER and SVRG still perform poorly.
For \texttt{real-sim}, although our methods, SGD1, and SGD2 achieve lower training accuracy, they are able to reach better test accuracy than SVRG+. 

In addition, the training and testing accuracies of the mini-batch case are presented in Fig. \ref{fig:logistic_acc2}, where  the relative residual of the train loss and the gradient norms are shown in Fig.~\ref{fig:logistic_reg2}. 
Again, our methods achieve training and test accuracies consistently with SGD2 in \texttt{w8a} and \texttt{real-sim}, while having better accuracy in \texttt{rcv1.binary}. 

\begin{figure}[htp!]
\begin{center}
\includegraphics[width = 1\textwidth]{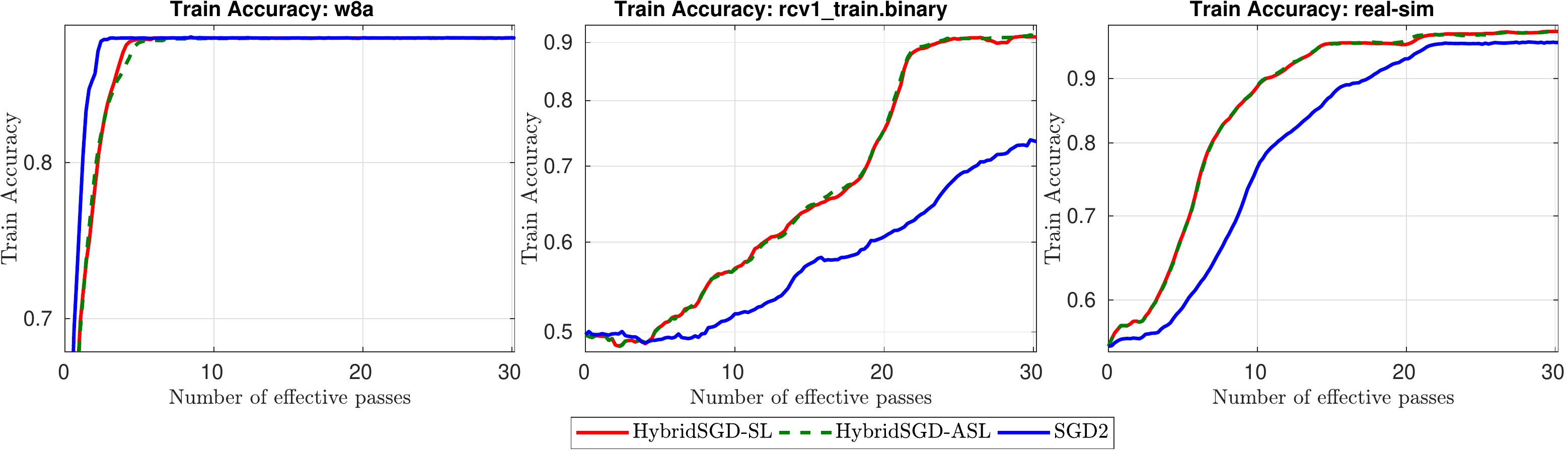}
\includegraphics[width = 1\textwidth]{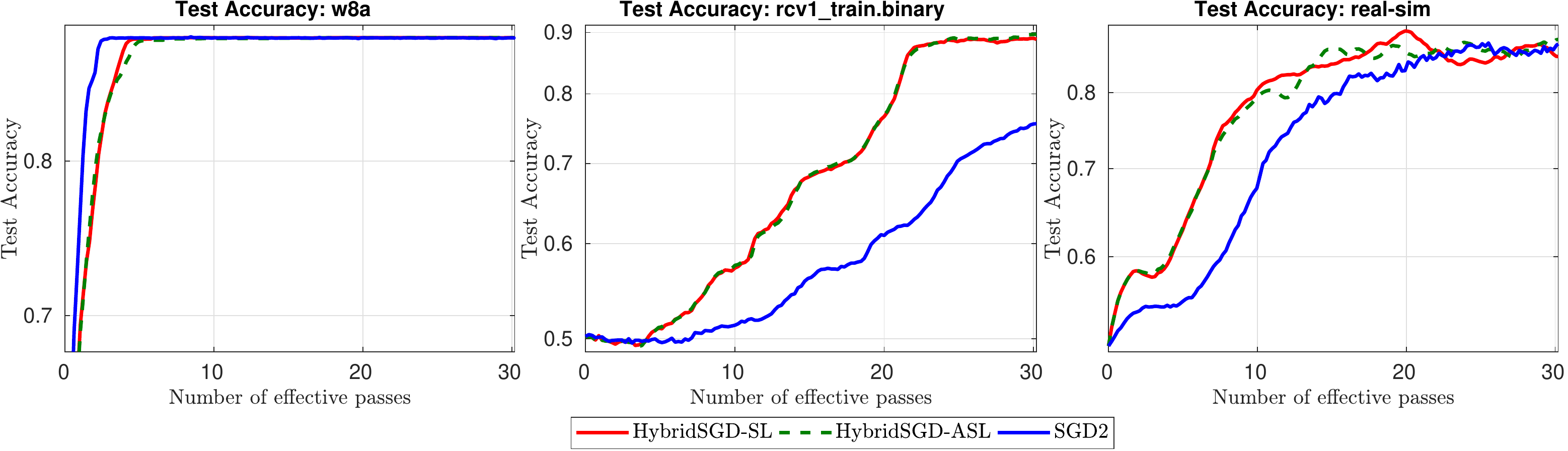}
\vspace{-1.5ex}
\caption{The training and test accuracies of \eqref{eq:exam2} on three datasets: Mini-batch case.}\label{fig:logistic_acc2}
\end{center}
\vspace{-2ex}
\end{figure}

We also run SVRG, SVRG+, SpiderBoost, and our double-loop variant (Algorithm~\ref{alg:A2}) on three datasets: \texttt{w8a}, \texttt{rcv1.binary}, and \texttt{real-sim}.
The results are plotted in Fig.~\ref{fig:logistic_reg4}.

\begin{figure}[htp!]
\begin{center}
\includegraphics[width = 1\textwidth]{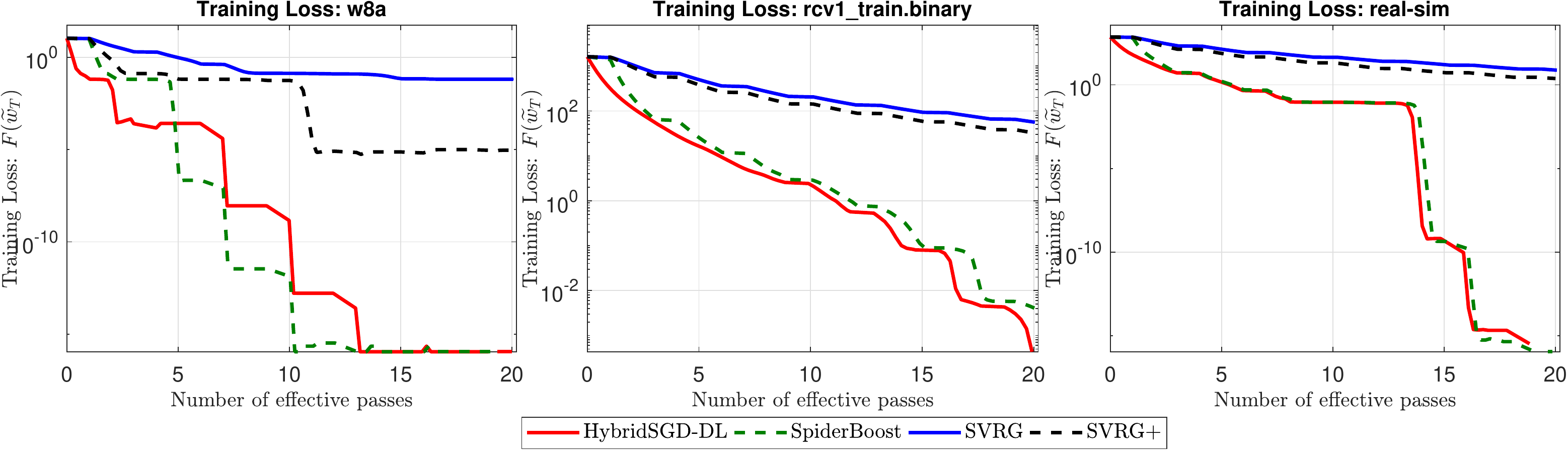}
\includegraphics[width = 1\textwidth]{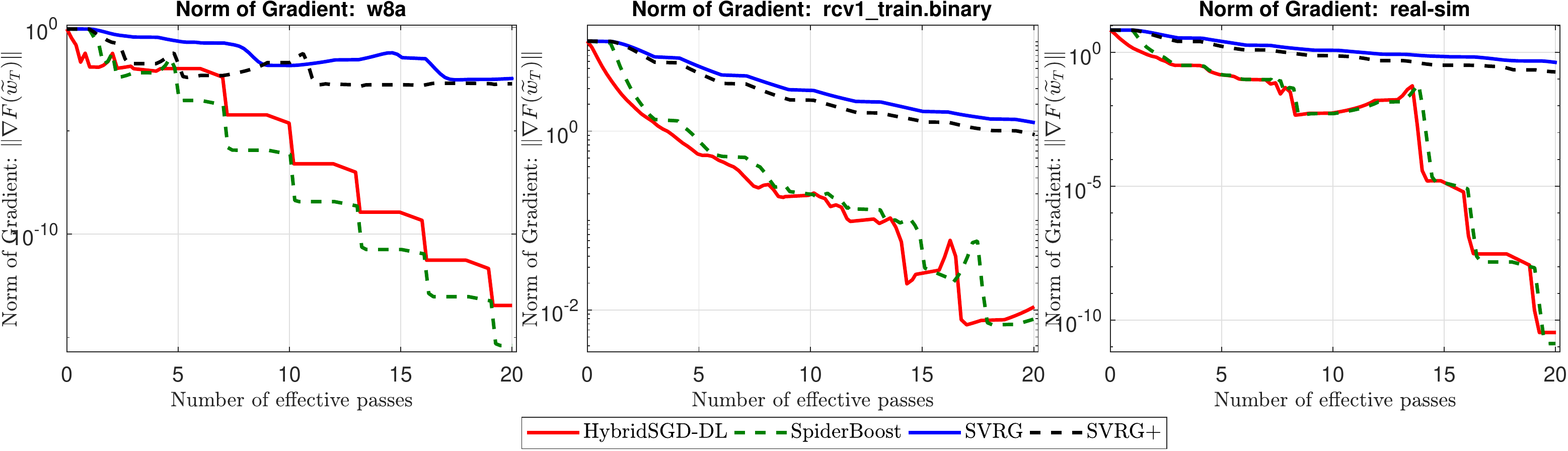}
\includegraphics[width = 1\textwidth]{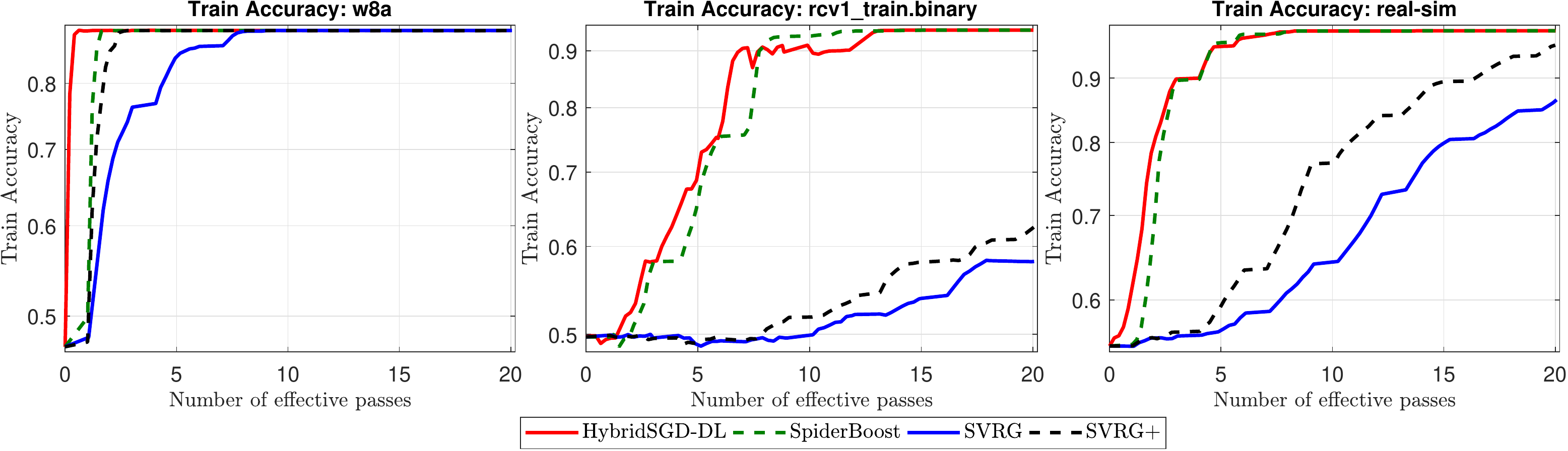}
\includegraphics[width = 1\textwidth]{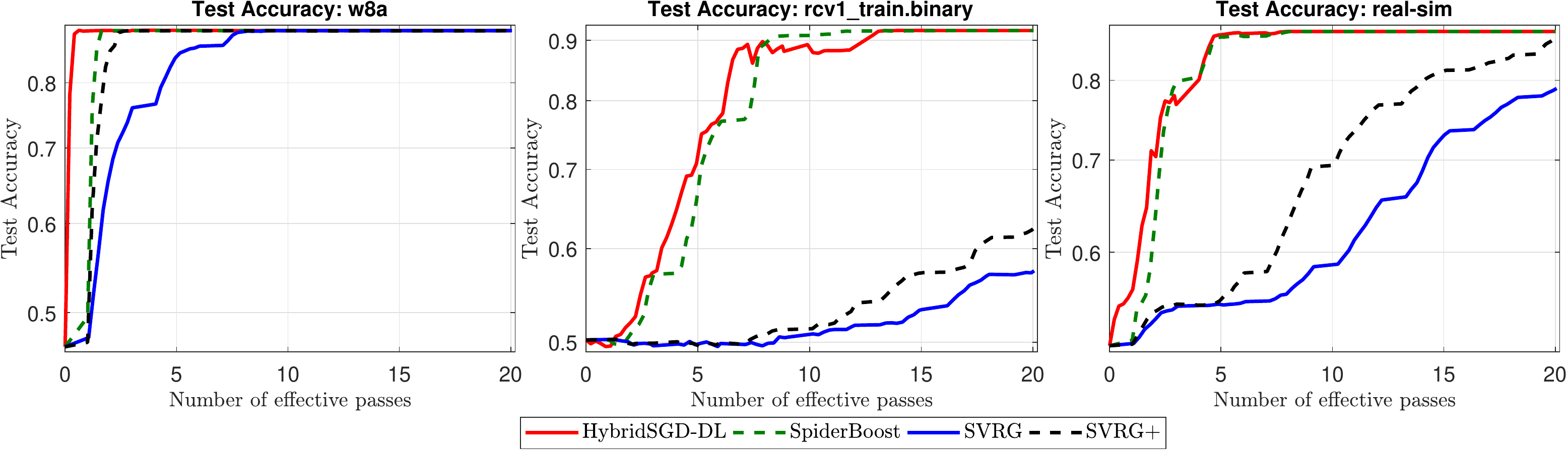}
\vspace{-1ex}
\caption{\done{The results of 4 algorithms for solving \eqref{eq:exam1}: Mini-batch case.}}\label{fig:logistic_reg4}
\end{center}
\vspace{-3ex}
\end{figure}

In this experiment, our double-loop variant and SpiderBoost outperform SVRG and SVRG+. 
Although the step-size of SVRG+ is $\eta = \frac{1}{6L}$ which is smaller than $\frac{1}{3L}$ in SVRG, SVRG+ still performs better than SVRG.
SpiderBoost uses a large step-size $\eta = \frac{1}{2L}$ and it indeed performs slightly better than ours in the \texttt{w8a} dataset, but is comparable in other two.
Note that our step-size $\eta$ is selected based on our theory in Theorem~\ref{th:double_loop_convergence}.

Finally, we conduct experiment on three larger datasets: \texttt{epsilon}, \texttt{url\_combined}, and \texttt{news20.binary}. Since the the sample sizes are large, we only run mini-batch variants.
The results of the single-loop variants are shown in Fig.~\ref{fig:larger_datasets1}.

\begin{figure}[htp!]
\begin{center}
\includegraphics[width = 1\textwidth]{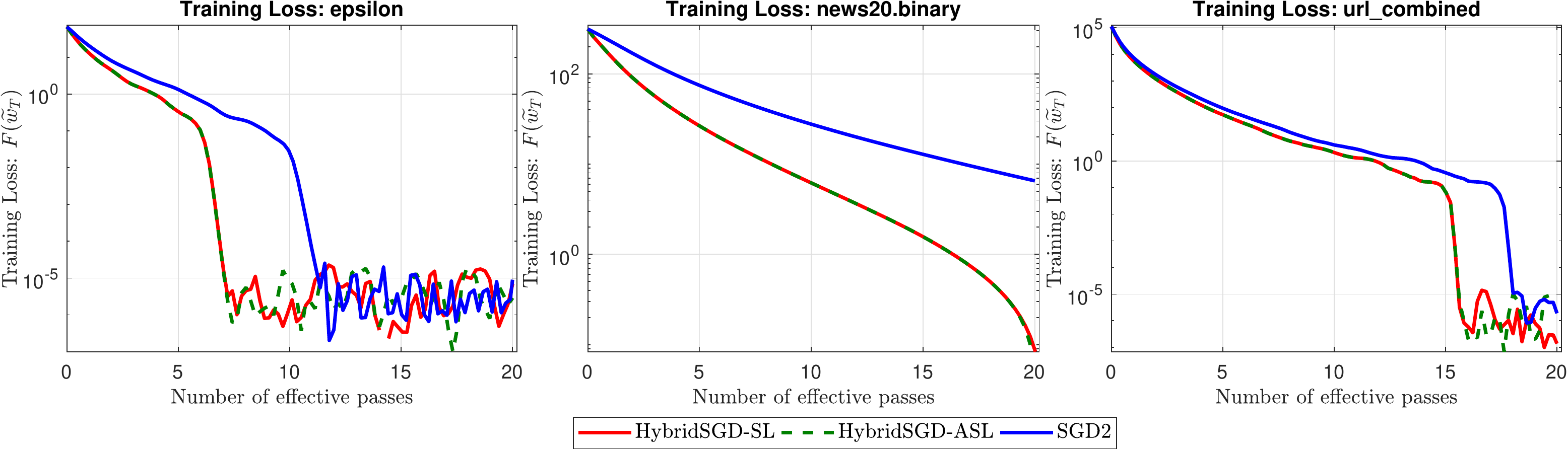}
\includegraphics[width = 1\textwidth]{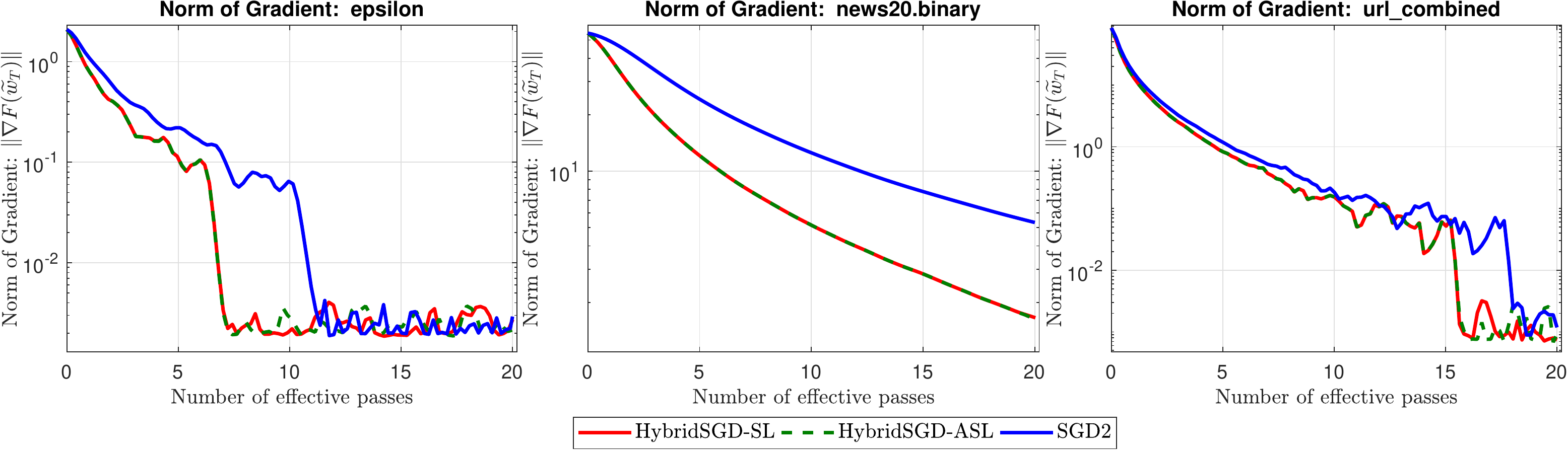}
\includegraphics[width = 1\textwidth]{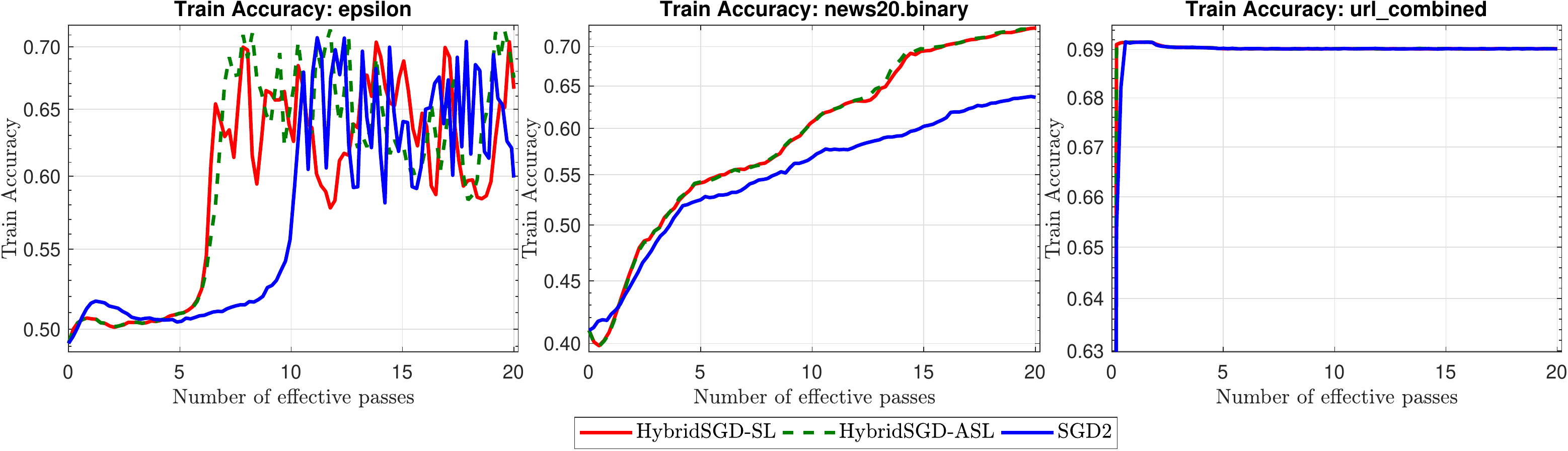}
\includegraphics[width = 1\textwidth]{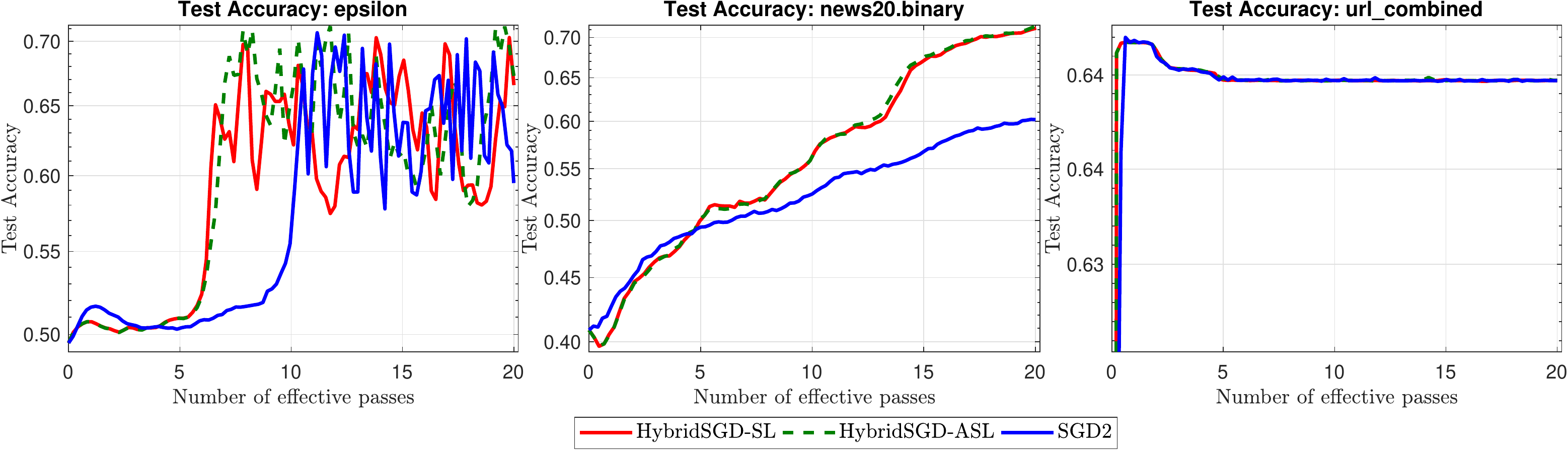}
\vspace{-1ex}
\caption{\done{The results of 3 single-loop algorithms for solving \eqref{eq:exam1} on large datasets: Mini-batch case.}}\label{fig:larger_datasets1}
\end{center}
\vspace{-3ex}
\end{figure}

We can observe from Fig.~\ref{fig:larger_datasets1} that our single loop variants outperform SGD in all three datasets. Note that the performance of the adaptive step-size variant is similar to its fixed step-size one.

The results of the double-loop variants are also shown in Fig.~\ref{fig:larger_datasets2}.

\begin{figure}[htp!]
\begin{center}
\includegraphics[width = 1\textwidth]{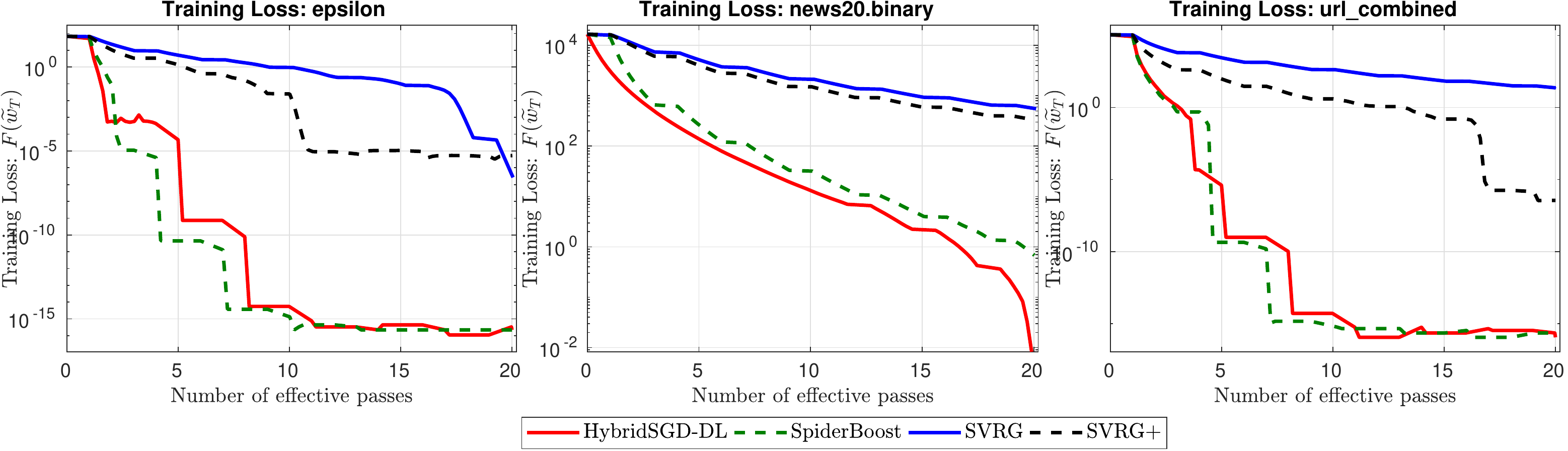}
\includegraphics[width = 1\textwidth]{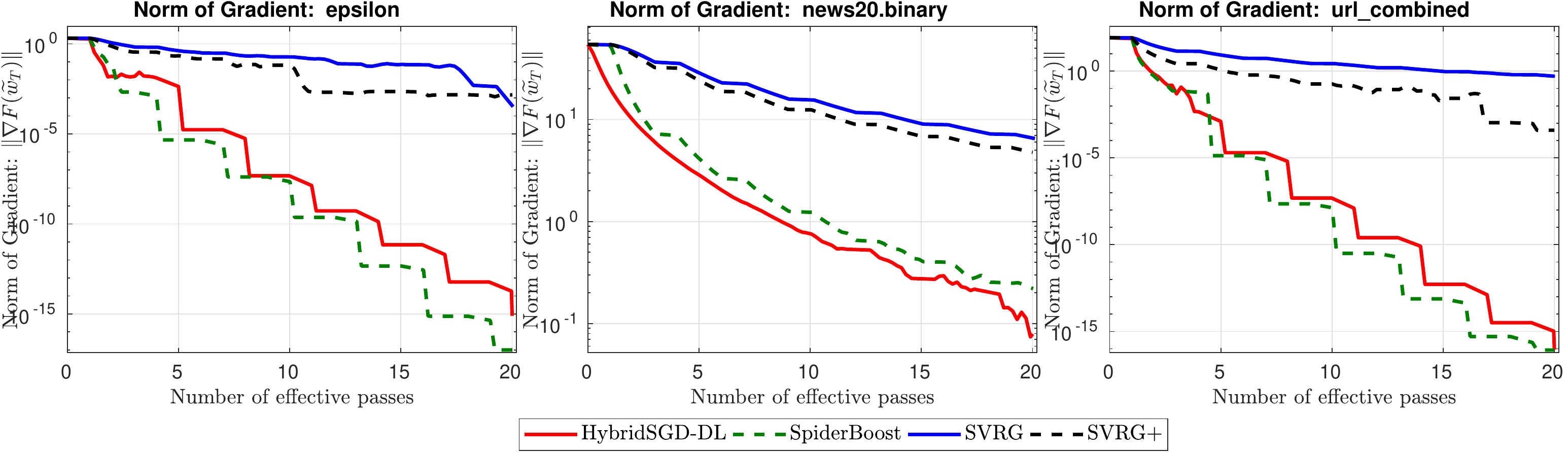}
\includegraphics[width = 1\textwidth]{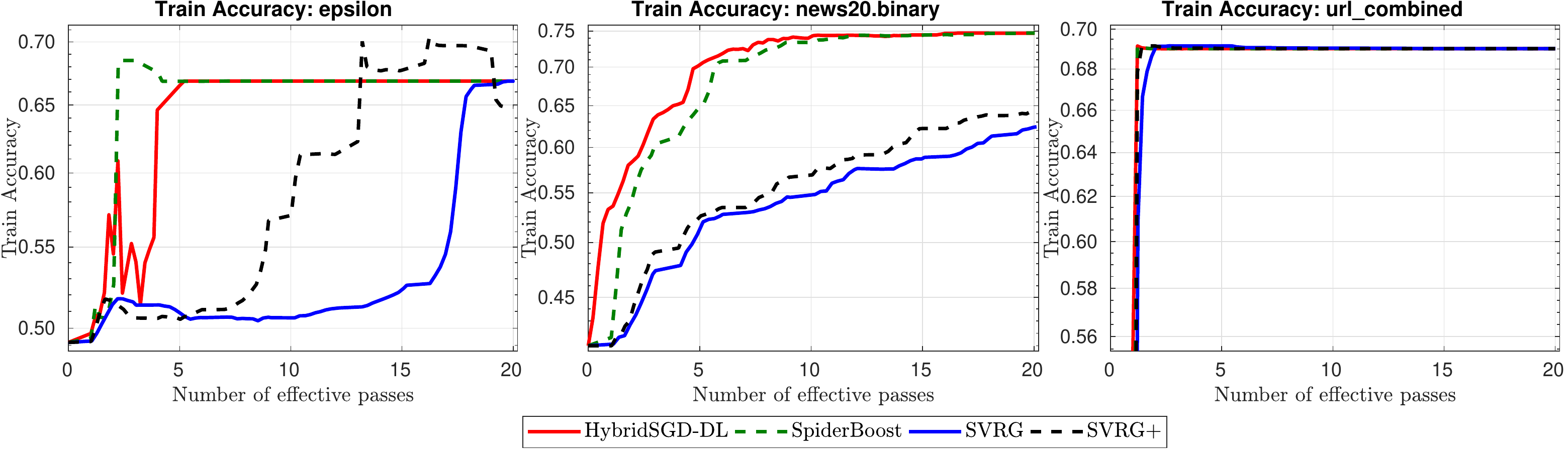}
\includegraphics[width = 1\textwidth]{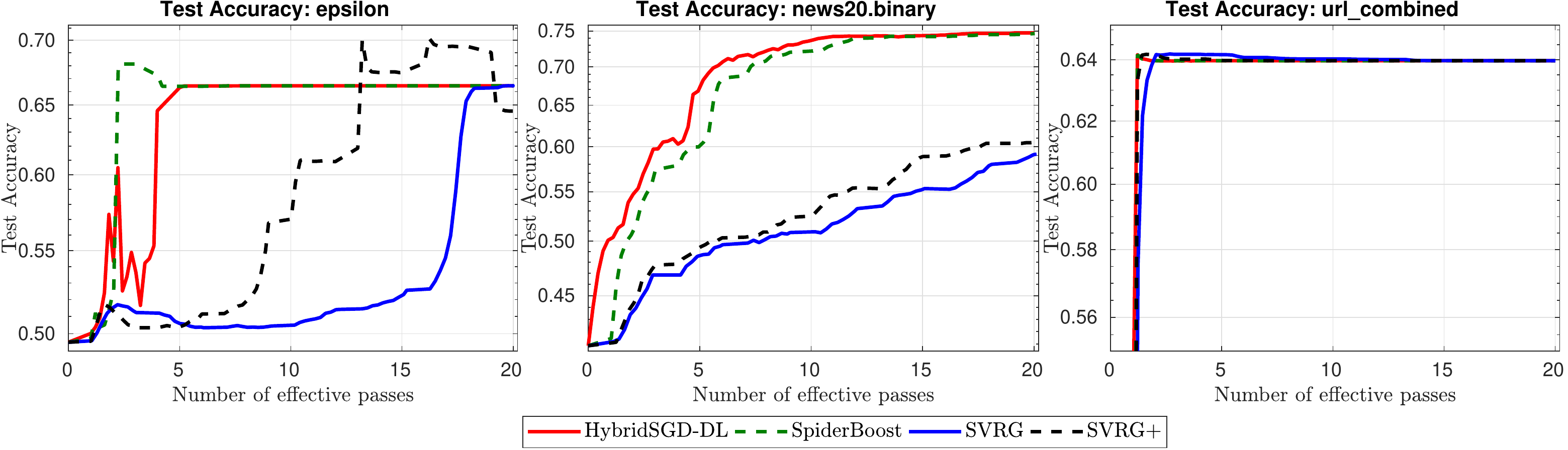}
\vspace{-1ex}
\caption{The results of $4$ double-loop algorithms for solving \eqref{eq:exam1} on large datsets: Mini-batch case.}\label{fig:larger_datasets2}
\end{center}
\vspace{-2ex}
\end{figure}

Clearly, our double-loop variants achieve better performance than SVRG and SVRG+ due to better convergence rate. SpiderBoost is slightly better than ours in the dataset \texttt{epsilon} while they are comparable in the last two datasets since we have the same best-known convergence rate as SpiderBoost.

\beforesubsec
\subsection{Binary classification involving non-convex loss and Tikhonov's regularizer}\label{subsec:example2}
\aftersubsec
We also conduct additional experiments to test our algorithms for solving \eqref{eq:exam2}.
We use two different non-convex loss functions as in \cite{zhao2010convex} apart from the one used in the main text, which are:
\begin{itemize}
\vspace{-1ex}
\item \textit{Nonconvex loss in two-layer neural networks}: $\ell_1(\tau, s) = \left(1 - \frac{1}{1+\exp(-\tau s)}\right)^2$. 
\item \textit{Logistic difference loss}: $\ell_2(\tau, s) = \log(1 + \exp(\tau s)) - \log(1 + \exp(-\tau s - 1))$.
\vspace{-1ex}
\end{itemize}
These functions are smooth and satisfy Assumption~\ref{as:A1}.

Let us first test our algorithms and other methods on three datasets: \texttt{w8a}, \texttt{rcv1.binary}, and \texttt{real-sim} using single-sample setting. The results are plotted in Fig.~\ref{fig:binary_single2} and Fig.~\ref{fig:binary_single3}. In this test, HybridSGD-DL achieves the best performance followed by HybridSGD-SL and HybridSGD-ASL. SPIDER has decent performance in the last two datasets. SGD variants also have good performance in all datasets while SGD2 is better than its fixed step-size variant. SVRG+ also has comparable performance with SGD2 whereas SVRG cannot achieve fast convergence due to its small step-size.

\begin{figure}[htp!]
\begin{center}
\includegraphics[width = 1\textwidth]{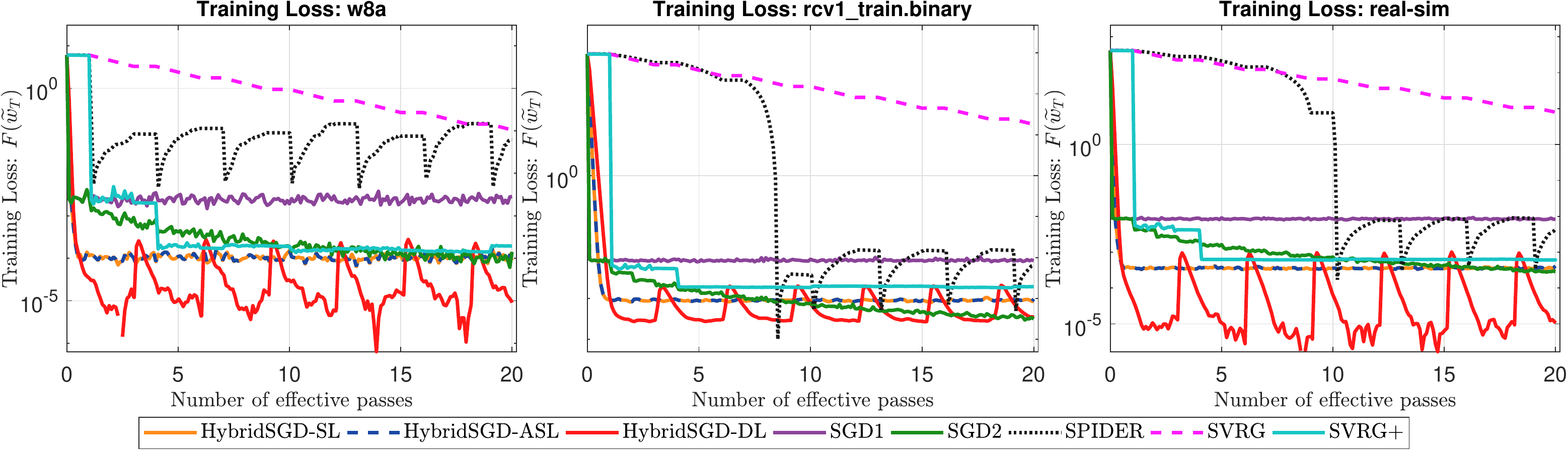}
\includegraphics[width = 1\textwidth]{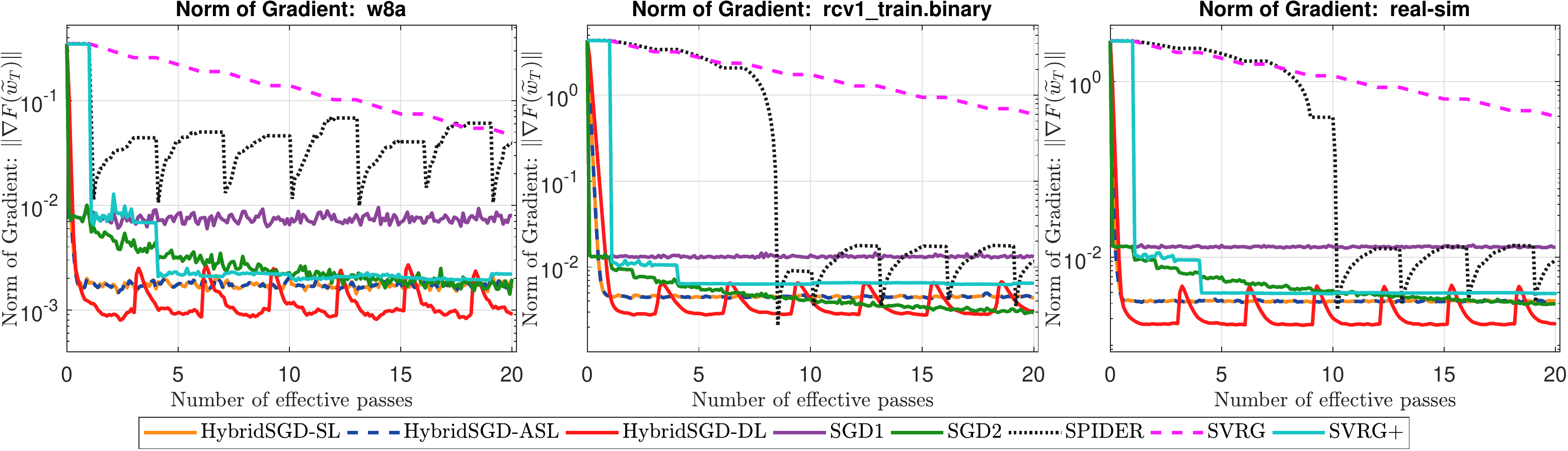}
\includegraphics[width = 1\textwidth]{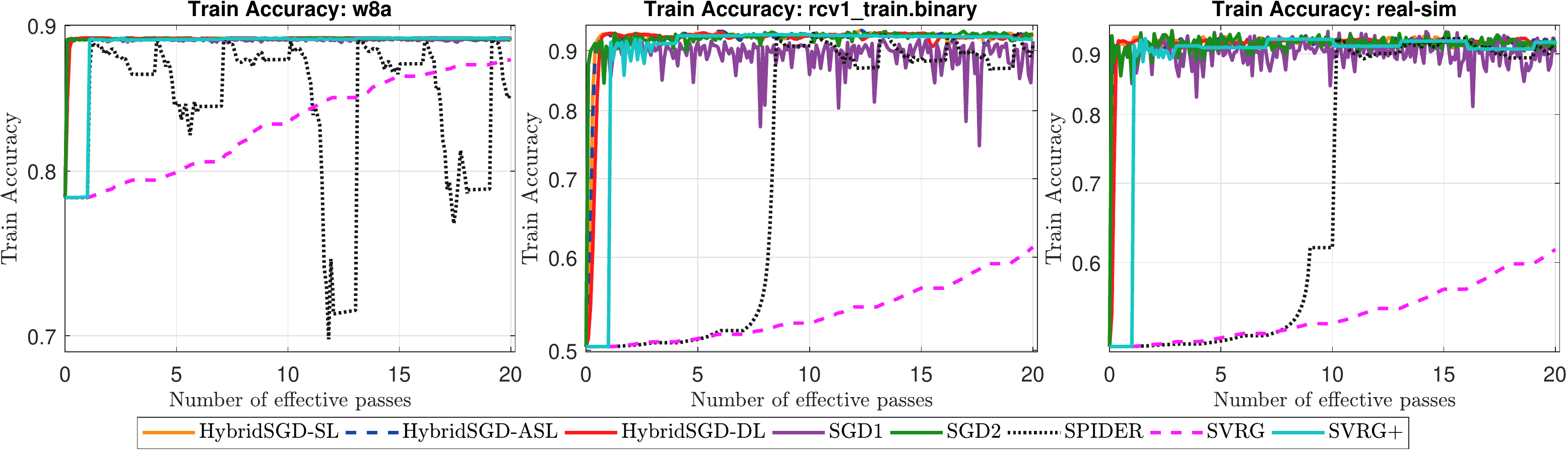}
\includegraphics[width = 1\textwidth]{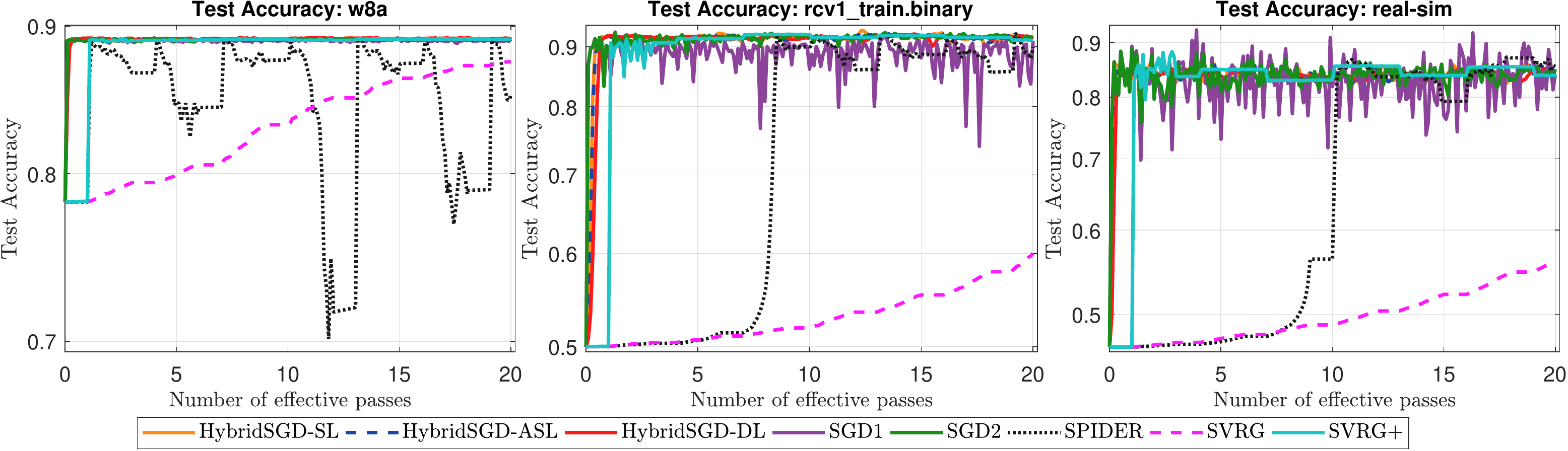}
\vspace{-1ex}
\caption{\done{The training loss and gradient  norms of \eqref{eq:exam2}  with loss $\ell_1$: Single-sample.}}\label{fig:binary_single2}
\end{center}
\vspace{-2ex}
\end{figure}

\begin{figure}[htp!]
\begin{center}
\includegraphics[width = 1\textwidth]{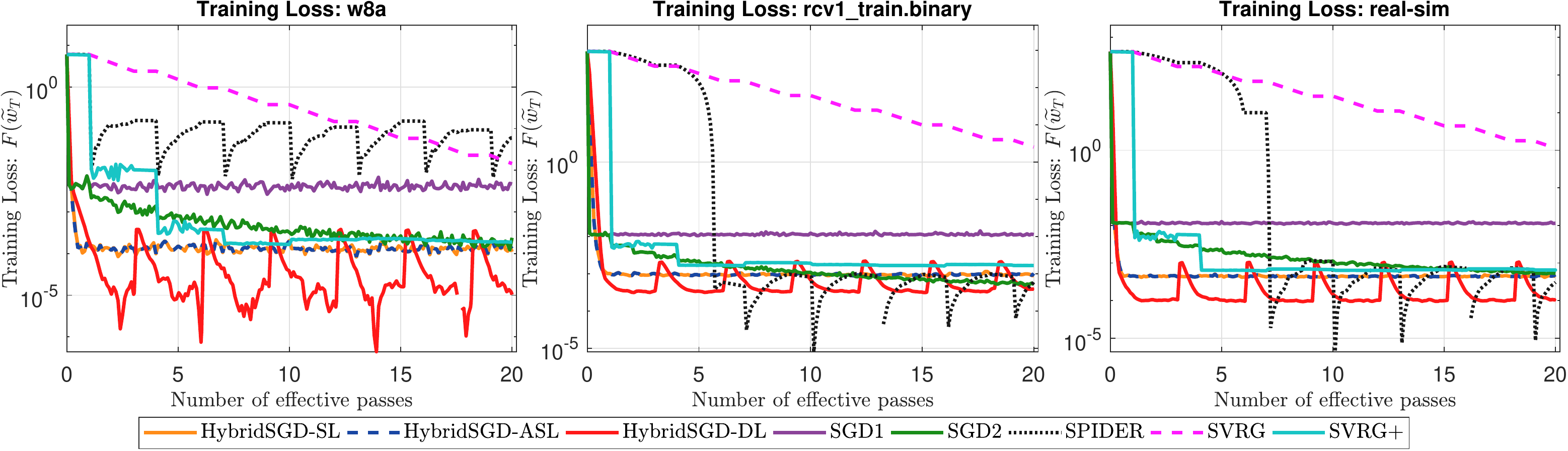}
\includegraphics[width = 1\textwidth]{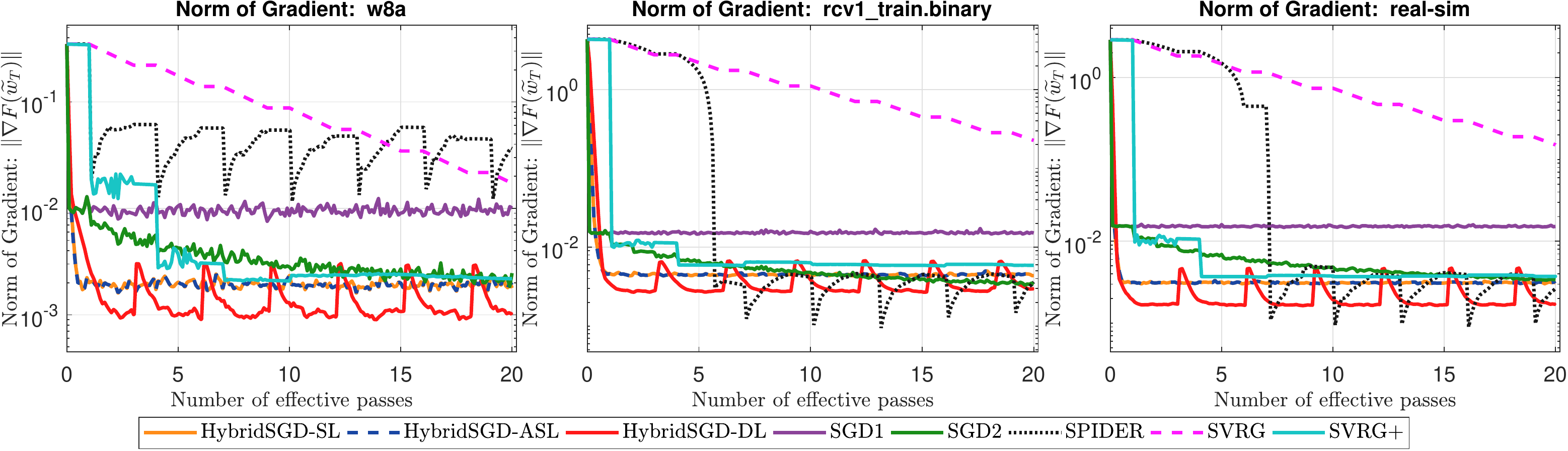}
\includegraphics[width = 1\textwidth]{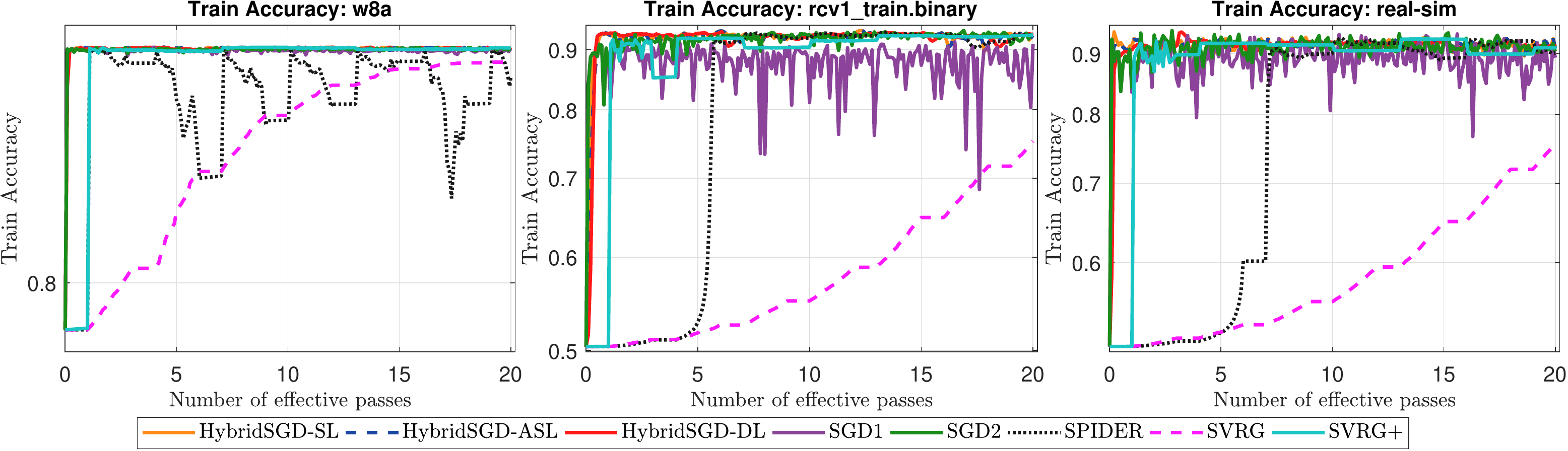}
\includegraphics[width = 1\textwidth]{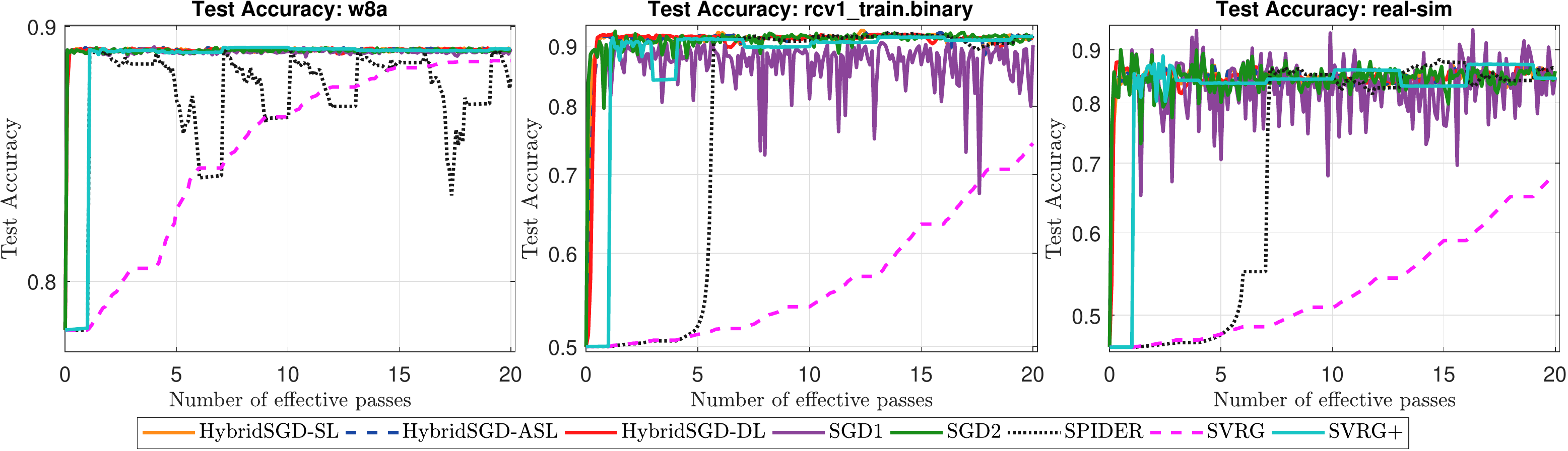}
\vspace{-1ex}
\caption{\done{The training loss and gradient norms of \eqref{eq:exam2} with loss $\ell_2$: Single-sample.}}\label{fig:binary_single3}
\end{center}
\vspace{-2ex}
\end{figure}

Next, we test mini-batch variants. 
On the one hand, we compare our single-loop variants HybridSGD-SL and HybridSGD-ASL with SGD. 
On the other hand, we compare our double-loop variants with SVRG, SVRG+, and SpiderBoost. The results for solving \eqref{eq:exam2} with loss $\ell_1$ are shown in Fig.~\ref{fig:binary_batch2_21} and Fig.~\ref{fig:binary_batch2_22} whereas Fig.~\ref{fig:binary_batch2_31} and Fig.~\ref{fig:binary_batch2_32} present the results when using loss $\ell_2$.

\begin{figure}[htp!]
\begin{center}
\includegraphics[width = 1\textwidth]{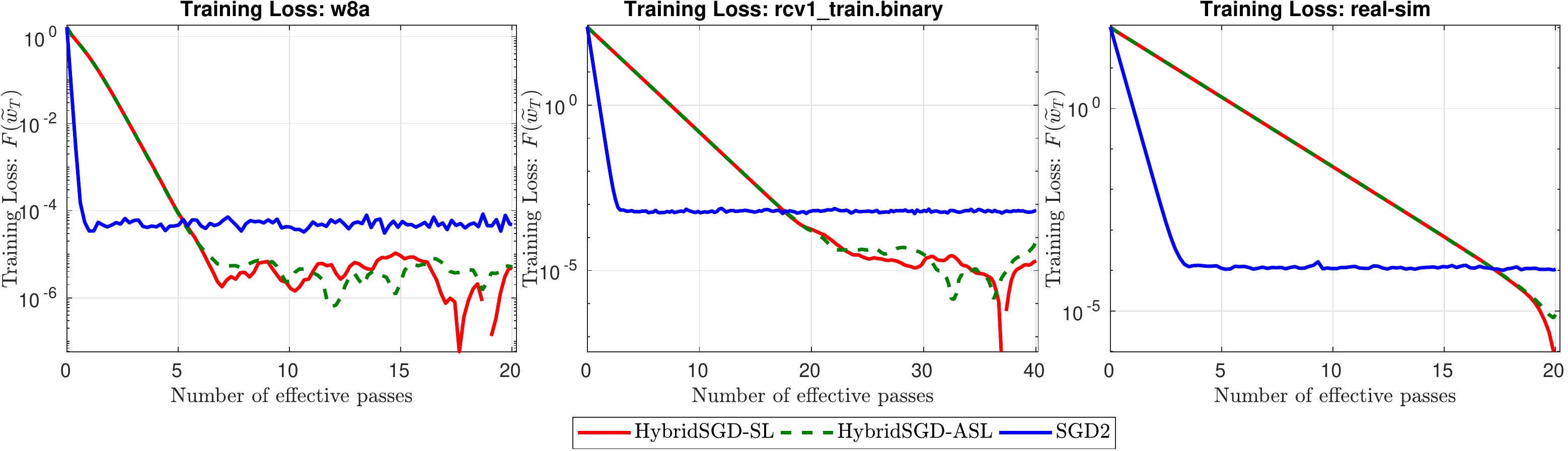}
\includegraphics[width = 1\textwidth]{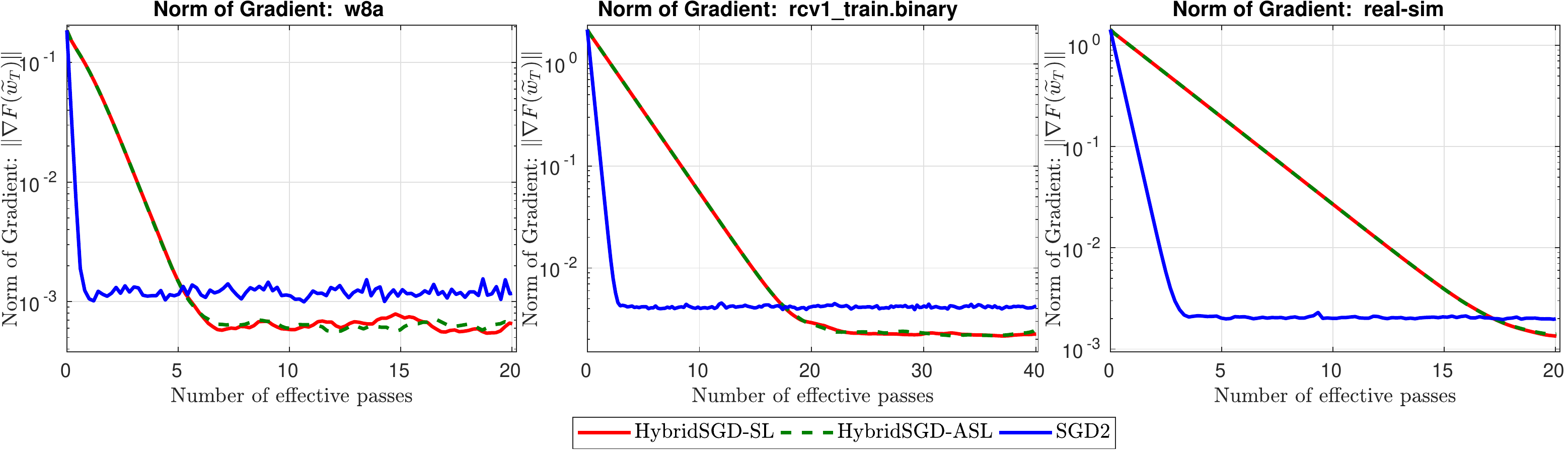}
\includegraphics[width = 1\textwidth]{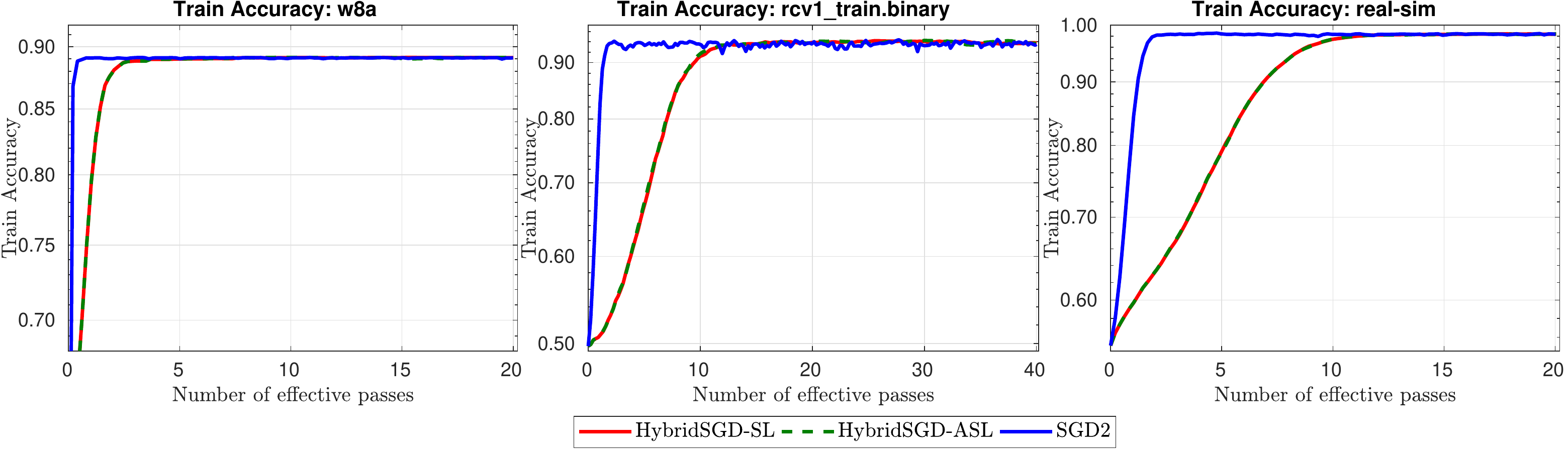}
\includegraphics[width = 1\textwidth]{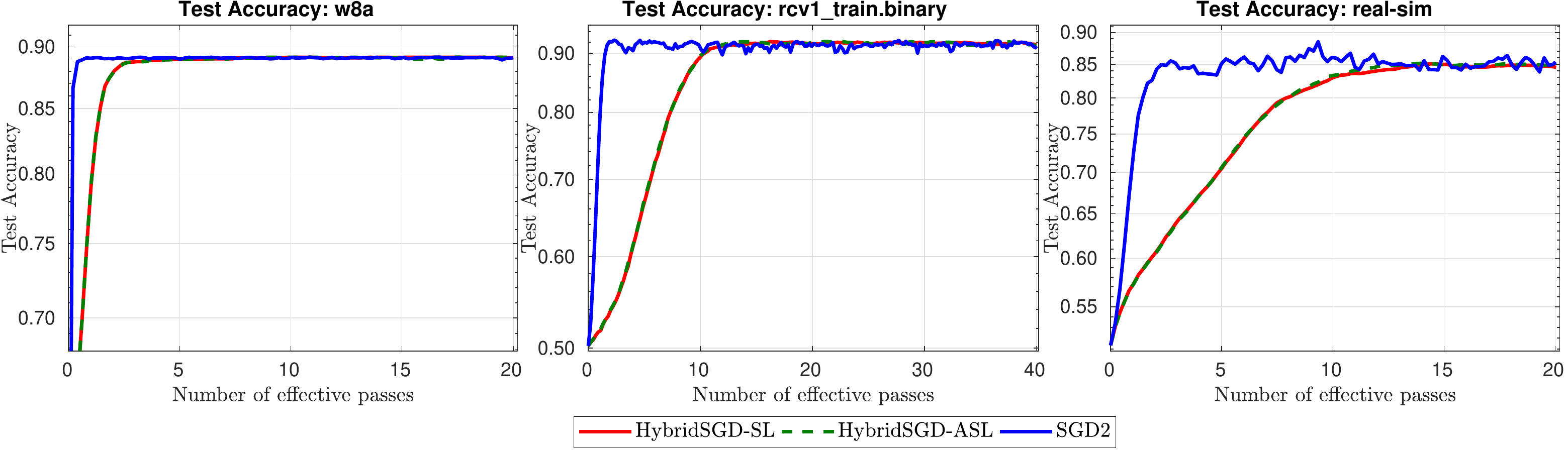}
\vspace{-1ex}
\caption{\done{The training loss and gradient norms of \eqref{eq:exam2} with loss $\ell_1$: Mini-batch.}}\label{fig:binary_batch2_21}
\end{center}
\vspace{-2ex}
\end{figure}

\begin{figure}[htp!]
\begin{center}
\includegraphics[width = 1\textwidth]{figs3/binary_loss2_batch_TrainLoss_2}
\includegraphics[width = 1\textwidth]{figs3/binary_loss2_batch_GradNorm_2}
\includegraphics[width = 1\textwidth]{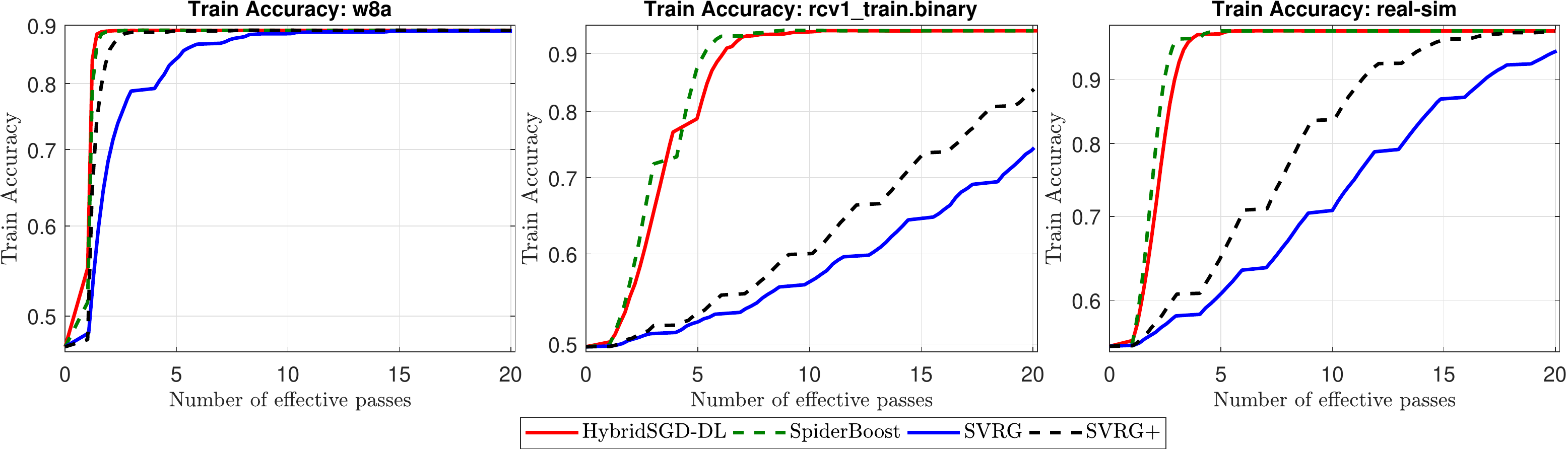}
\includegraphics[width = 1\textwidth]{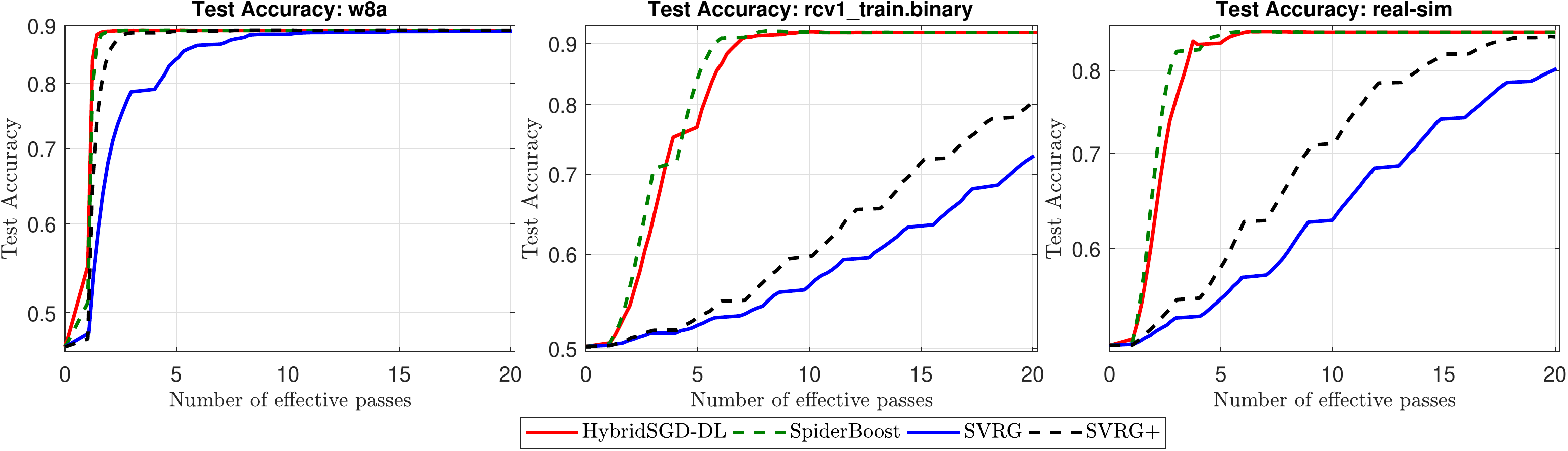}
\vspace{-1ex}
\caption{\done{The training loss and gradient norms of \eqref{eq:exam2}  with loss $\ell_1$: Mini-batch.}}\label{fig:binary_batch2_22}
\end{center}
\vspace{-2ex}
\end{figure}

\begin{figure}[htp!]
\begin{center}
\includegraphics[width = 1\textwidth]{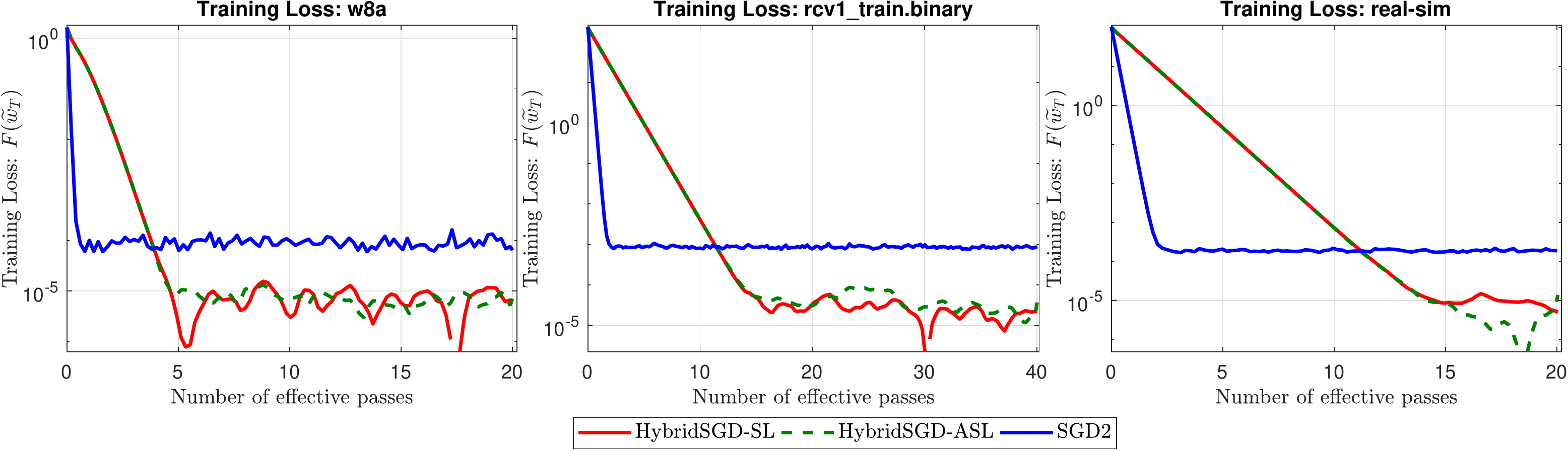}
\includegraphics[width = 1\textwidth]{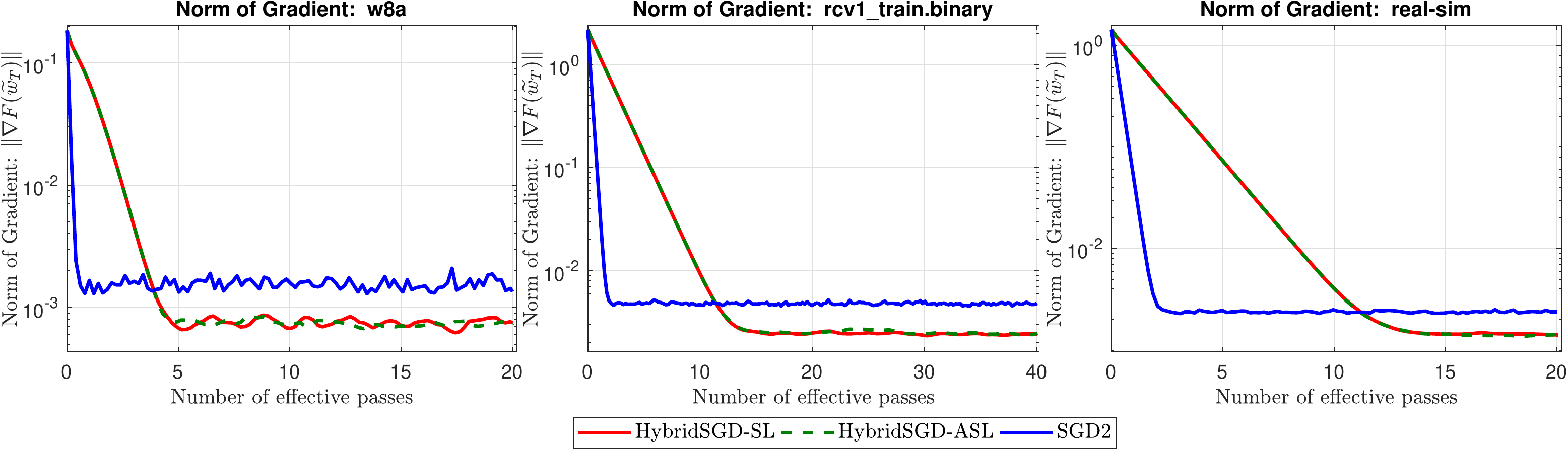}
\includegraphics[width = 1\textwidth]{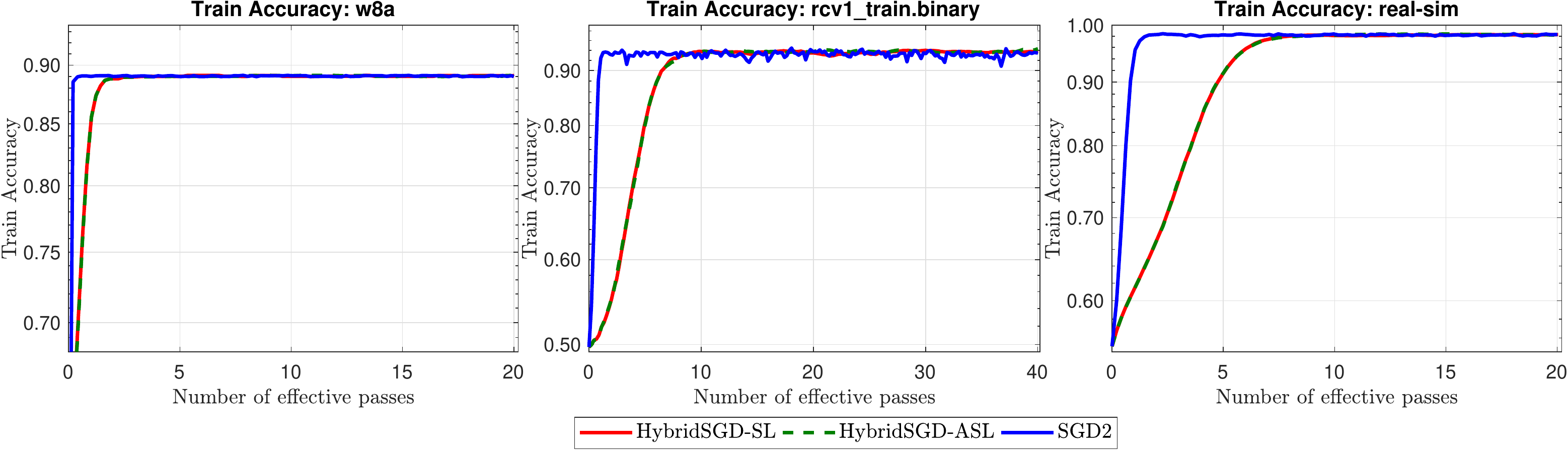}
\includegraphics[width = 1\textwidth]{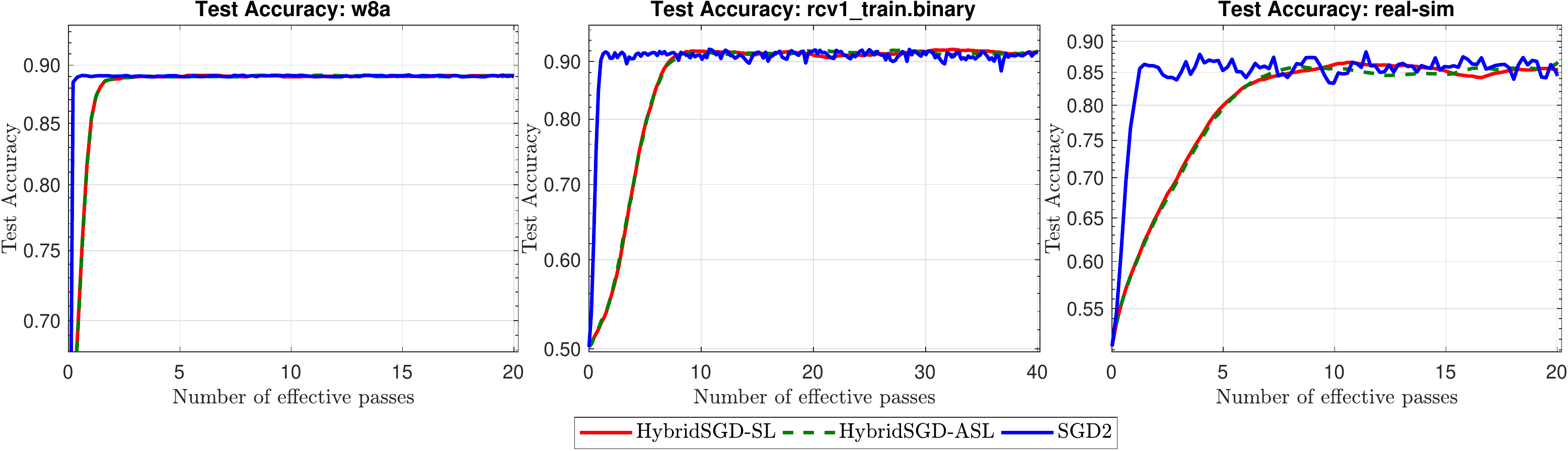}
\vspace{-1ex}
\caption{\done{The training loss and gradient norms of \eqref{eq:exam2}  with loss $\ell_2$: Mini-batch.}}\label{fig:binary_batch2_31}
\end{center}
\vspace{-2ex}
\end{figure}

\begin{figure}[htp!]
\begin{center}
\includegraphics[width = 1\textwidth]{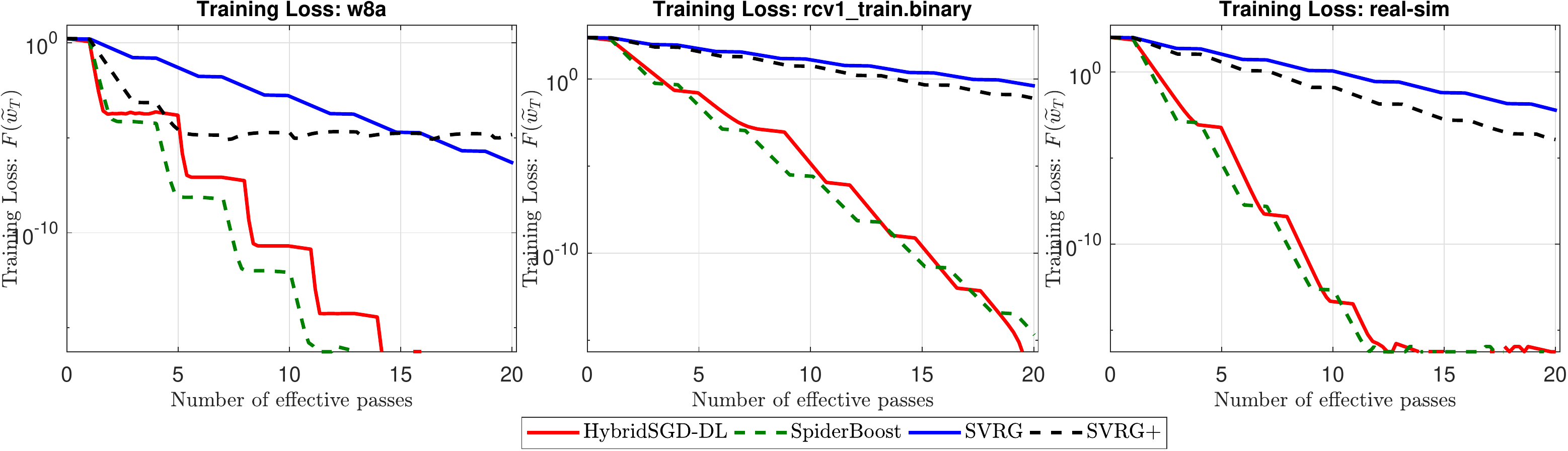}
\includegraphics[width = 1\textwidth]{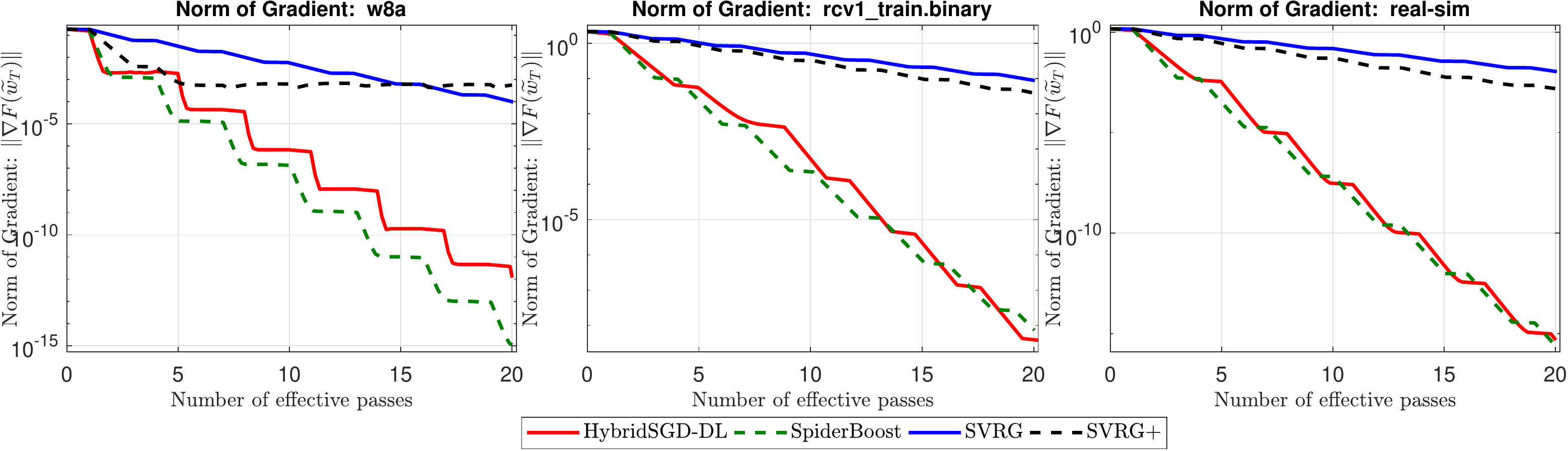}
\includegraphics[width = 1\textwidth]{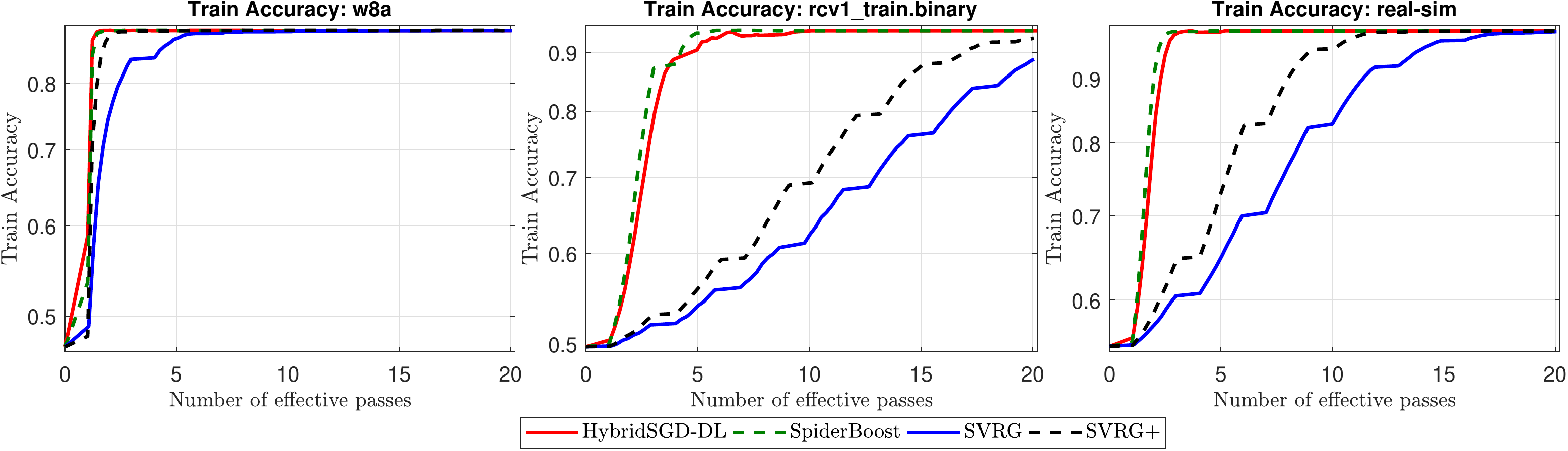}
\includegraphics[width = 1\textwidth]{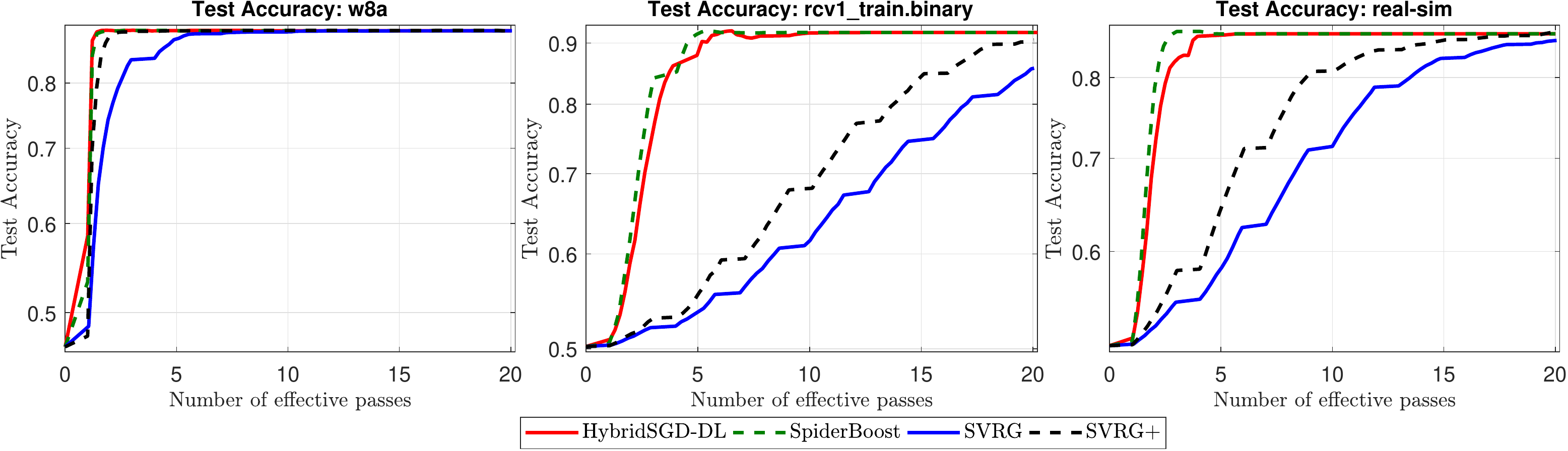}
\vspace{-1ex}
\caption{\done{The training loss and gradient norms of \eqref{eq:exam2}  with loss $\ell_2$: Mini-batch.}}\label{fig:binary_batch2_32}
\end{center}
\vspace{-2ex}
\end{figure}

Additionally, we repeat the experiments on three larger datasets: \texttt{epsilon}, \texttt{news20.binary}, and \texttt{ulr\_combined}. 
The results are shown in Fig.~\ref{fig:binary_batch4_21}, \ref{fig:binary_batch4_22}, \ref{fig:binary_batch4_31}, and \ref{fig:binary_batch4_32}.

\begin{figure}[H]
\begin{center}
\includegraphics[width = 1\textwidth]{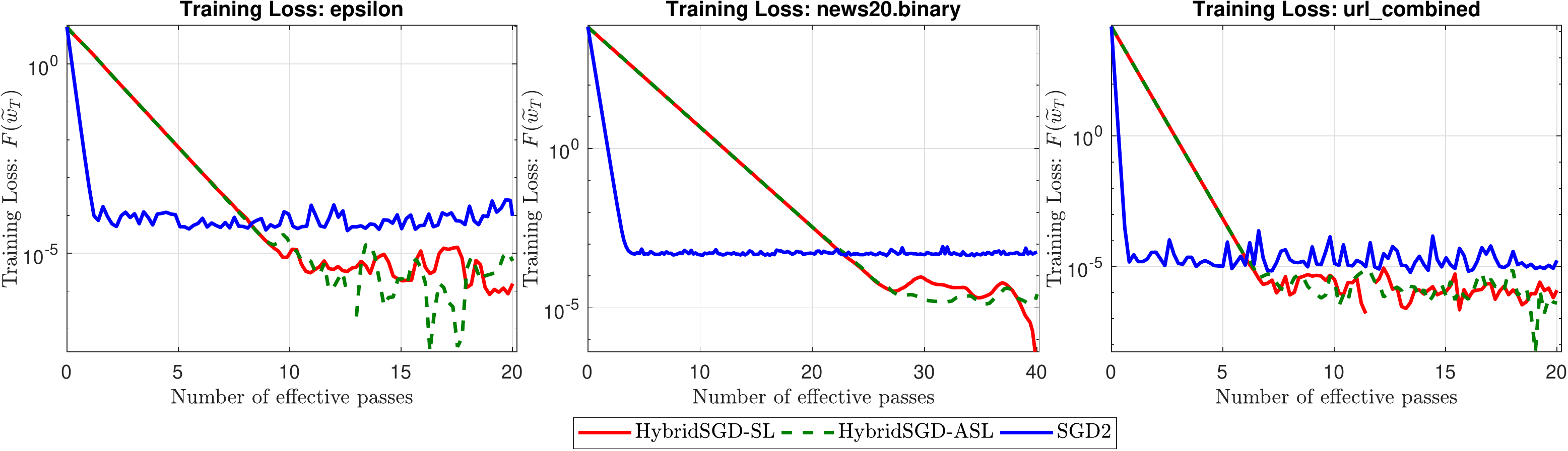}
\includegraphics[width = 1\textwidth]{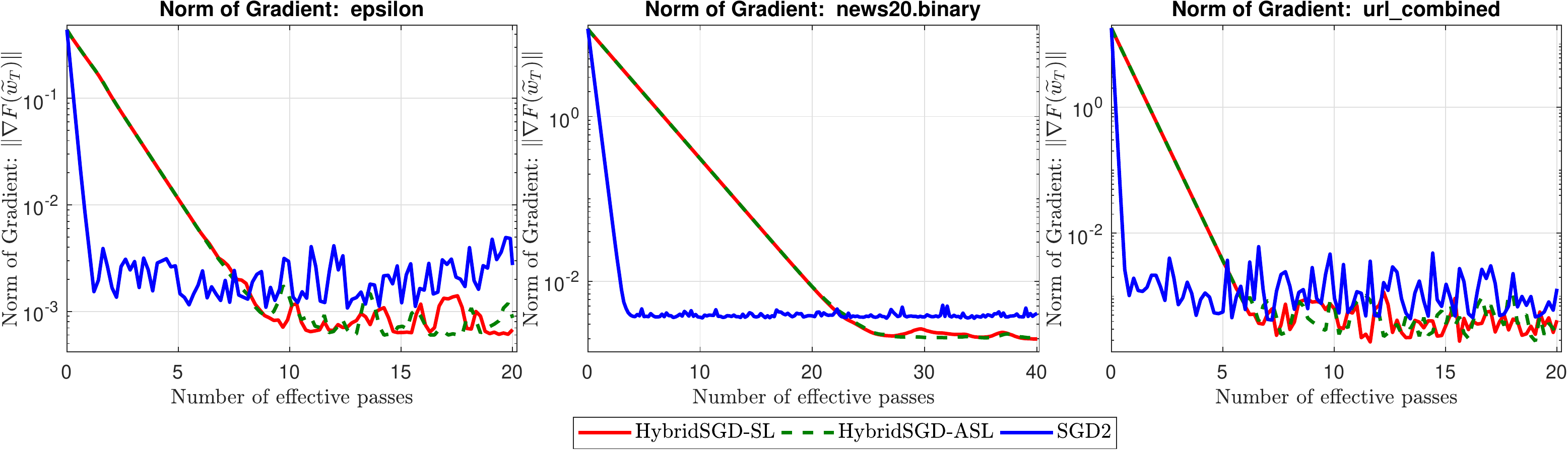}
\includegraphics[width = 1\textwidth]{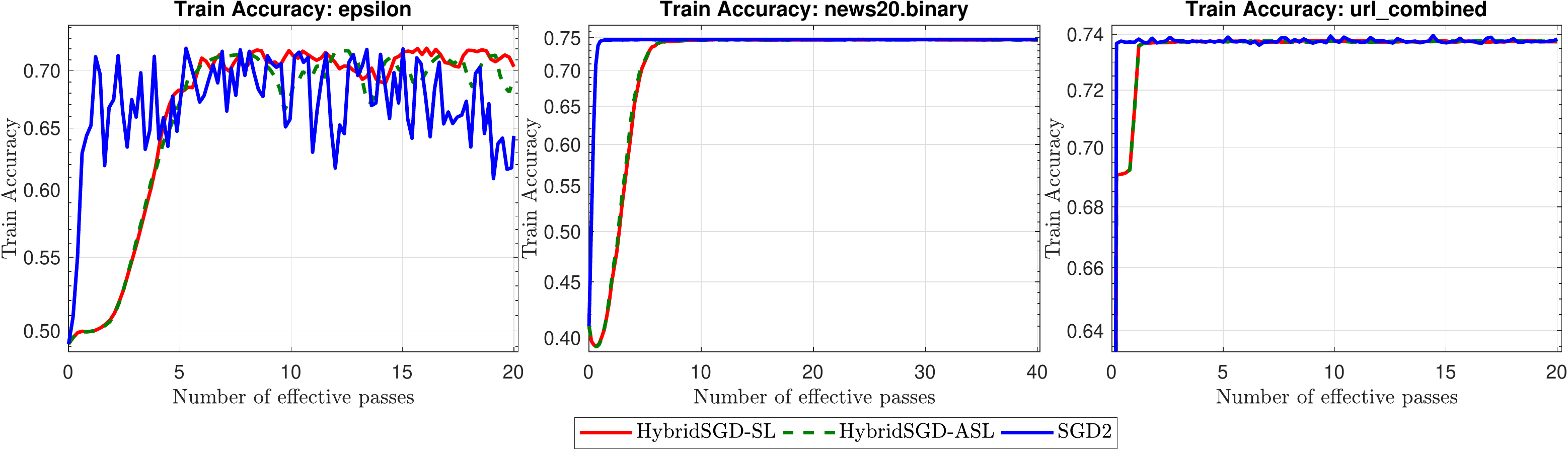}
\includegraphics[width = 1\textwidth]{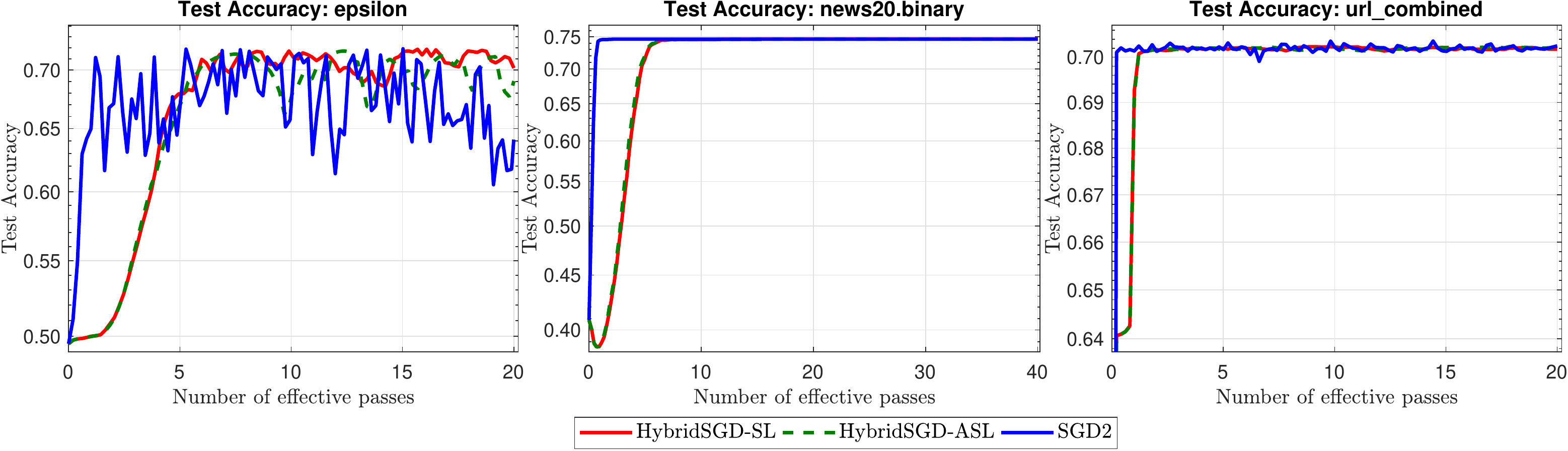}
\vspace{-1ex}
\caption{\done{The training loss and gradient norms of \eqref{eq:exam2}  with loss $\ell_1$: Mini-batch.}}\label{fig:binary_batch4_21}
\end{center}
\vspace{-2ex}
\end{figure}

\begin{figure}[H]
\begin{center}
\includegraphics[width = 1\textwidth]{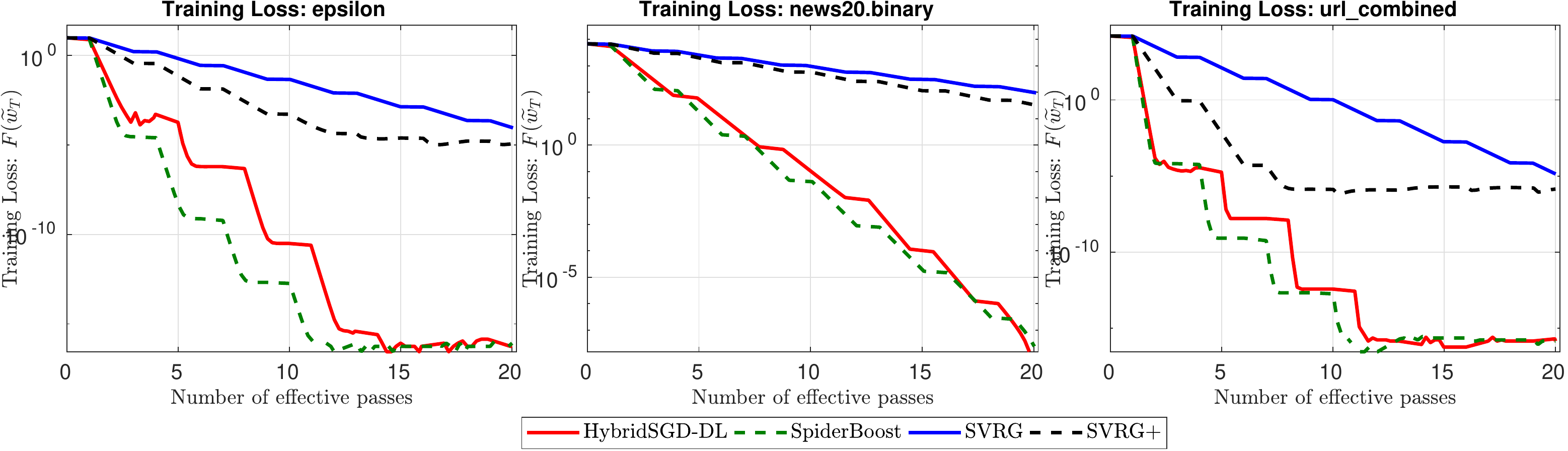}
\includegraphics[width = 1\textwidth]{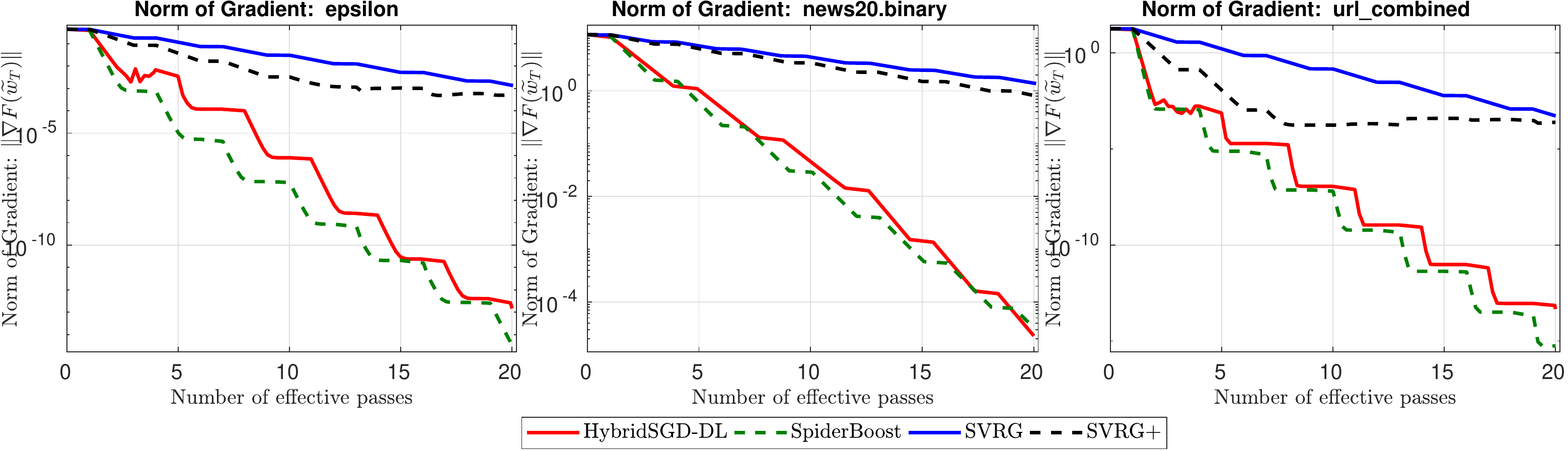}
\includegraphics[width = 1\textwidth]{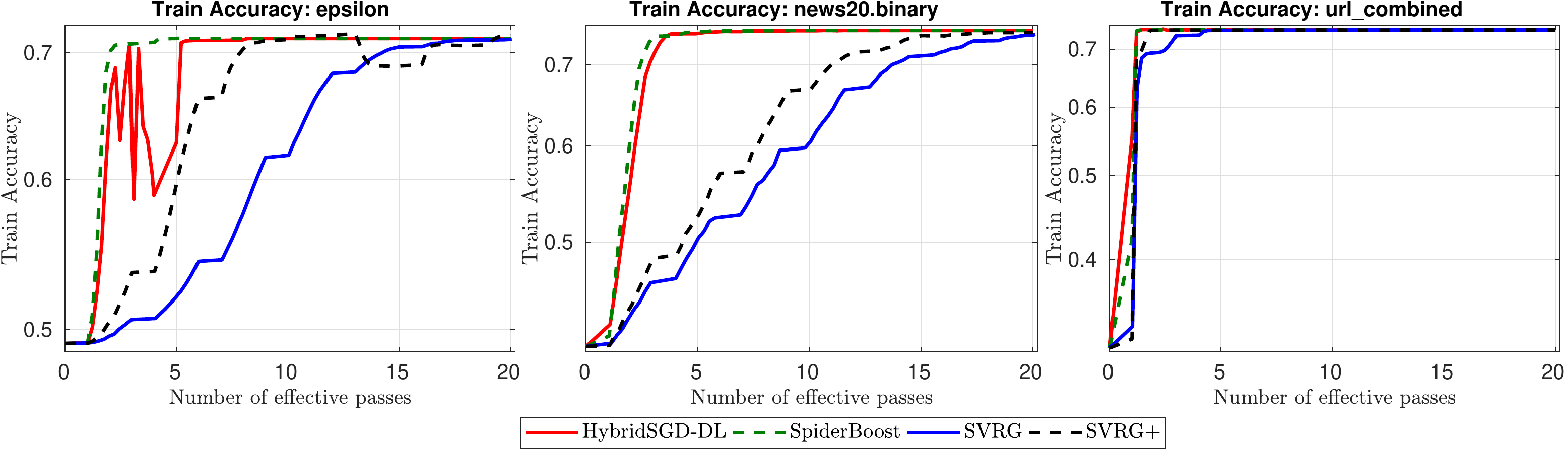}
\includegraphics[width = 1\textwidth]{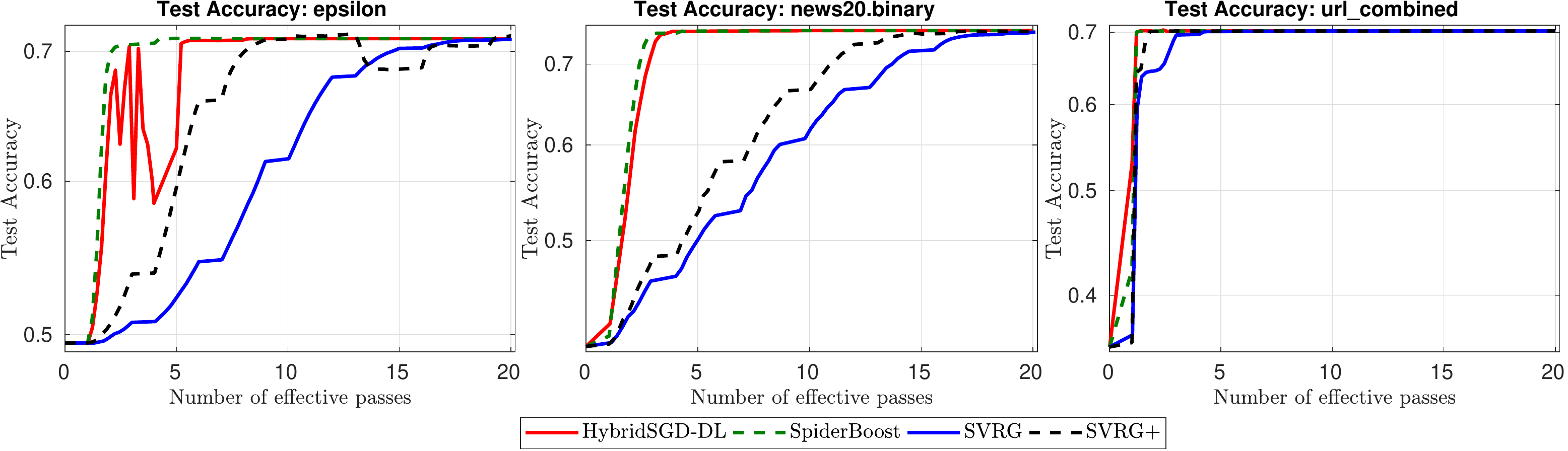}
\vspace{-1ex}
\caption{\done{The training loss and gradient norms of \eqref{eq:exam2}  with loss $\ell_1$: Mini-batch.}}\label{fig:binary_batch4_22}
\end{center}
\vspace{-2ex}
\end{figure}

\begin{figure}[htp!]
\begin{center}
\includegraphics[width = 1\textwidth]{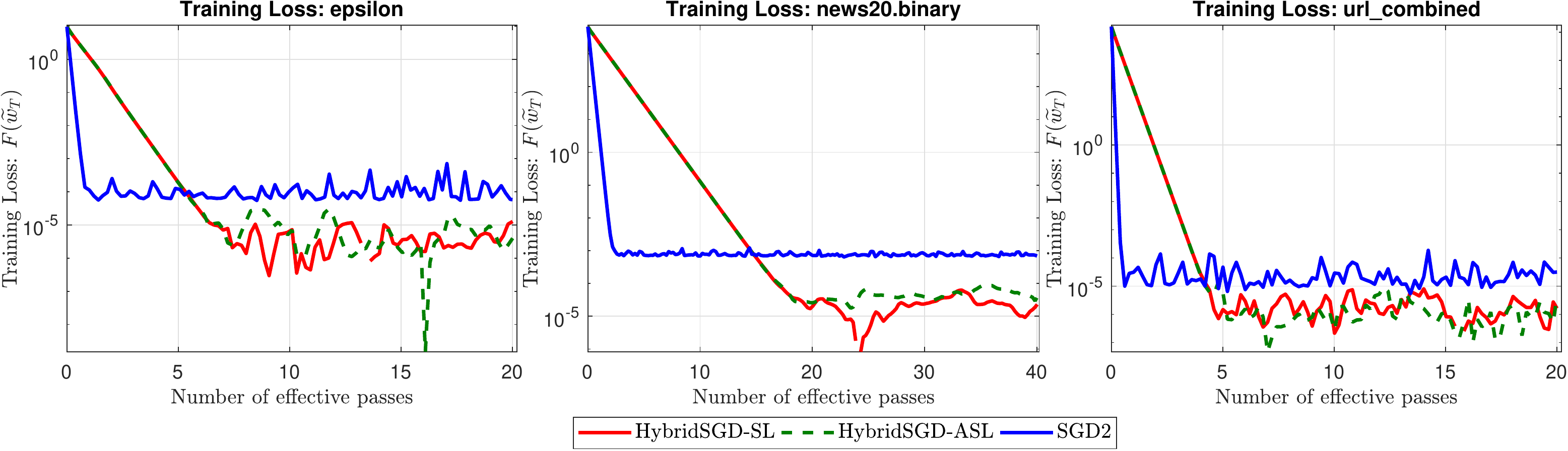}
\includegraphics[width = 1\textwidth]{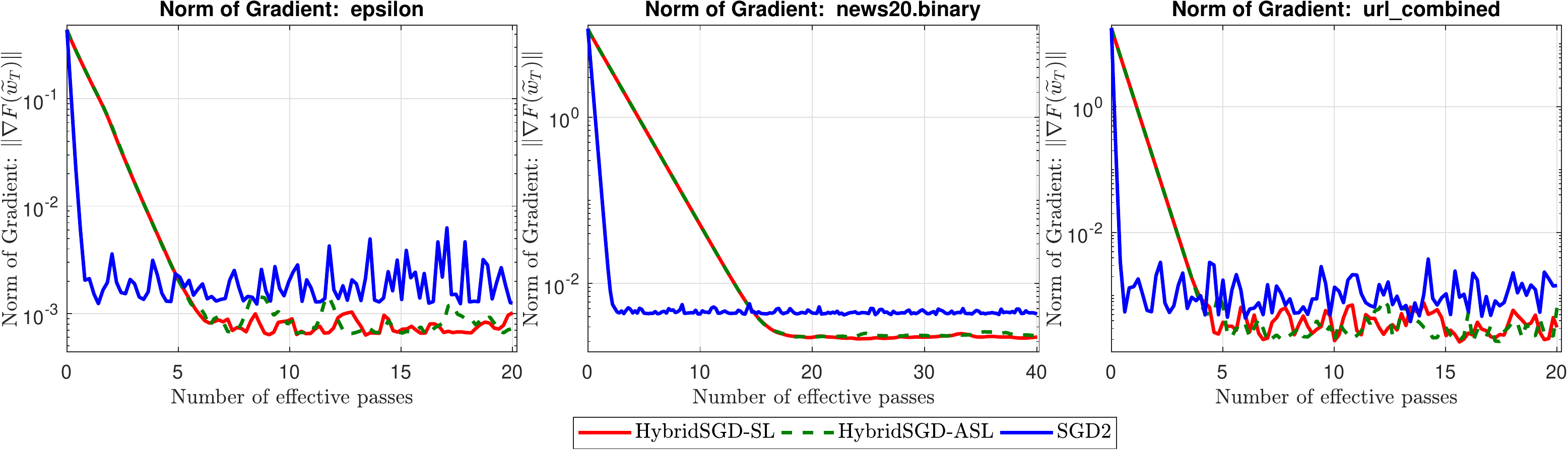}
\includegraphics[width = 1\textwidth]{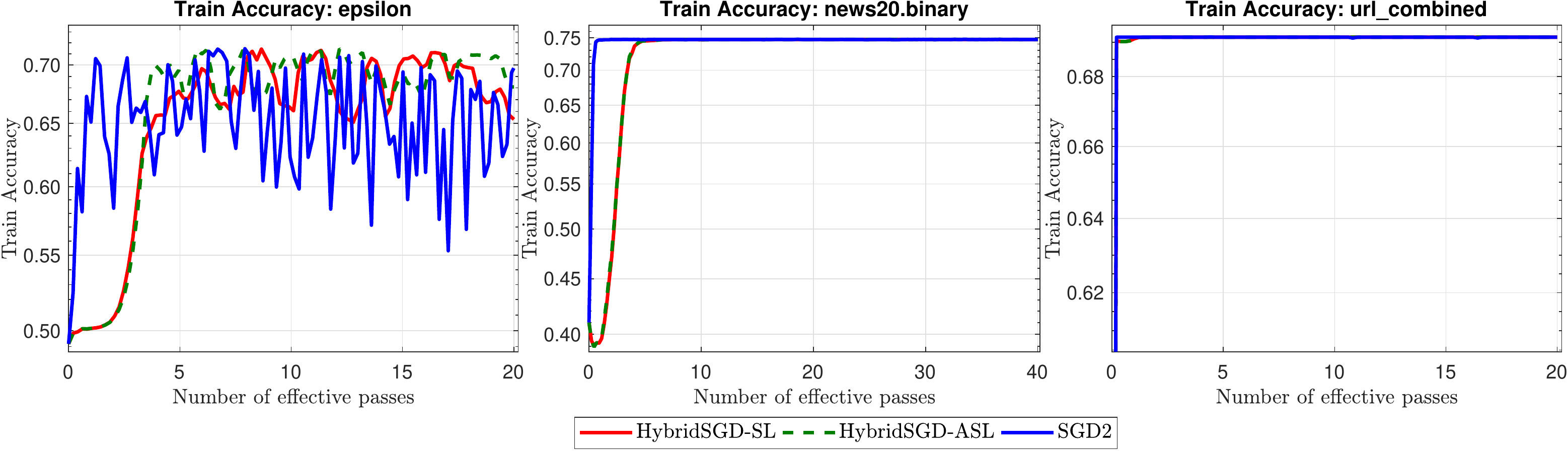}
\includegraphics[width = 1\textwidth]{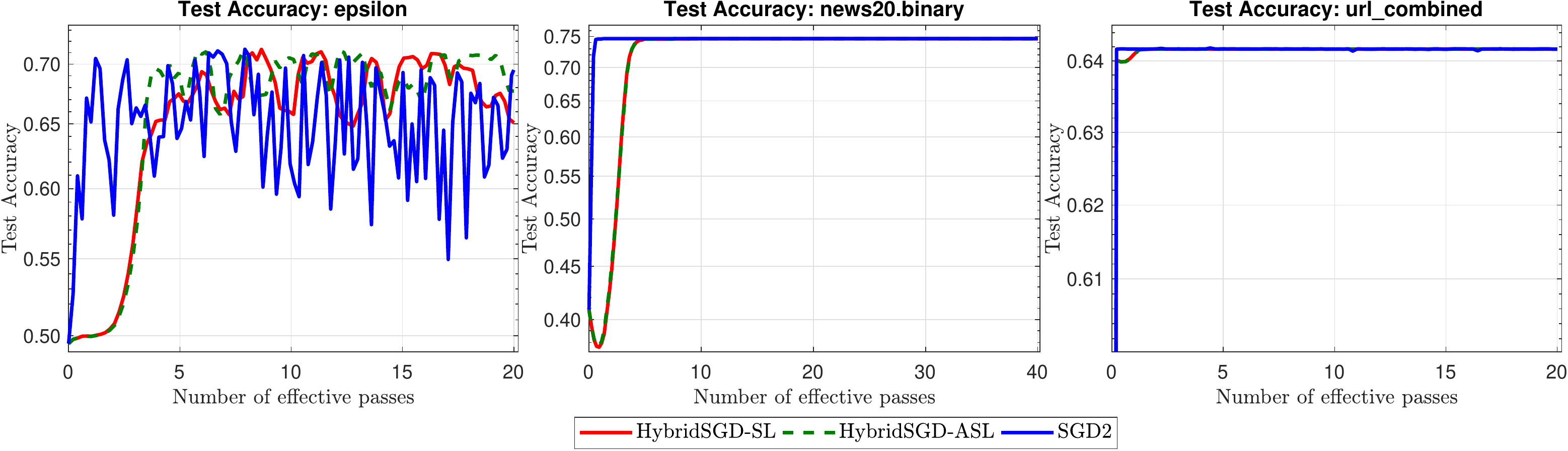}
\vspace{-1ex}
\caption{\done{The training loss and gradient norms of \eqref{eq:exam2}  with loss $\ell_2$: Mini-batch.}}\label{fig:binary_batch4_31}
\end{center}
\vspace{-2ex}
\end{figure}

\begin{figure}[htp!]
\begin{center}
\includegraphics[width = 1\textwidth]{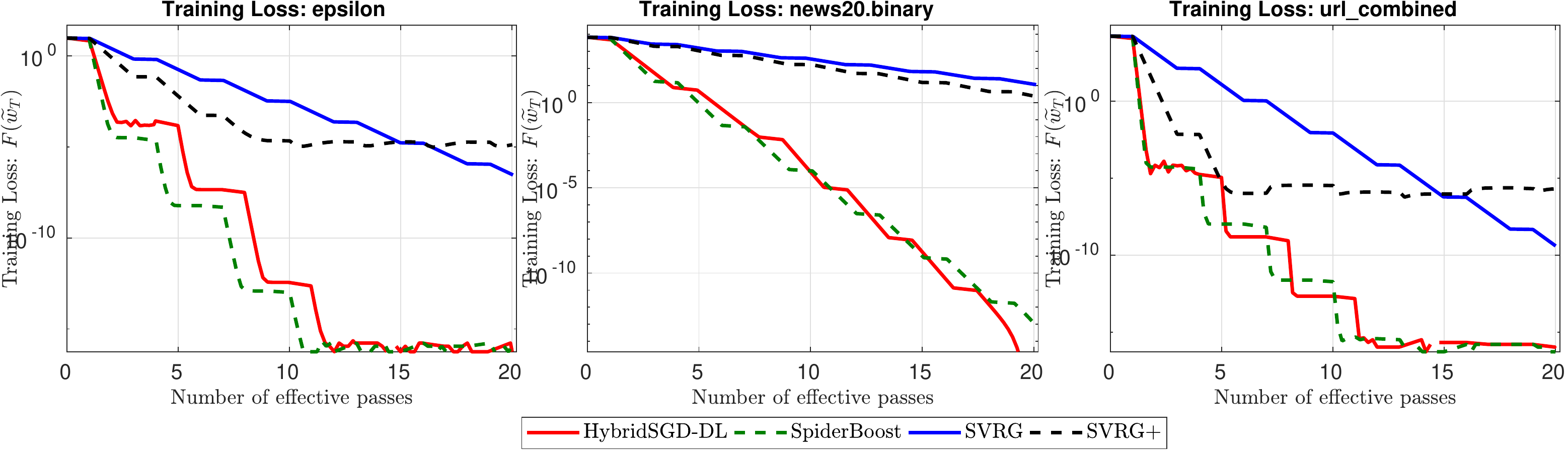}
\includegraphics[width = 1\textwidth]{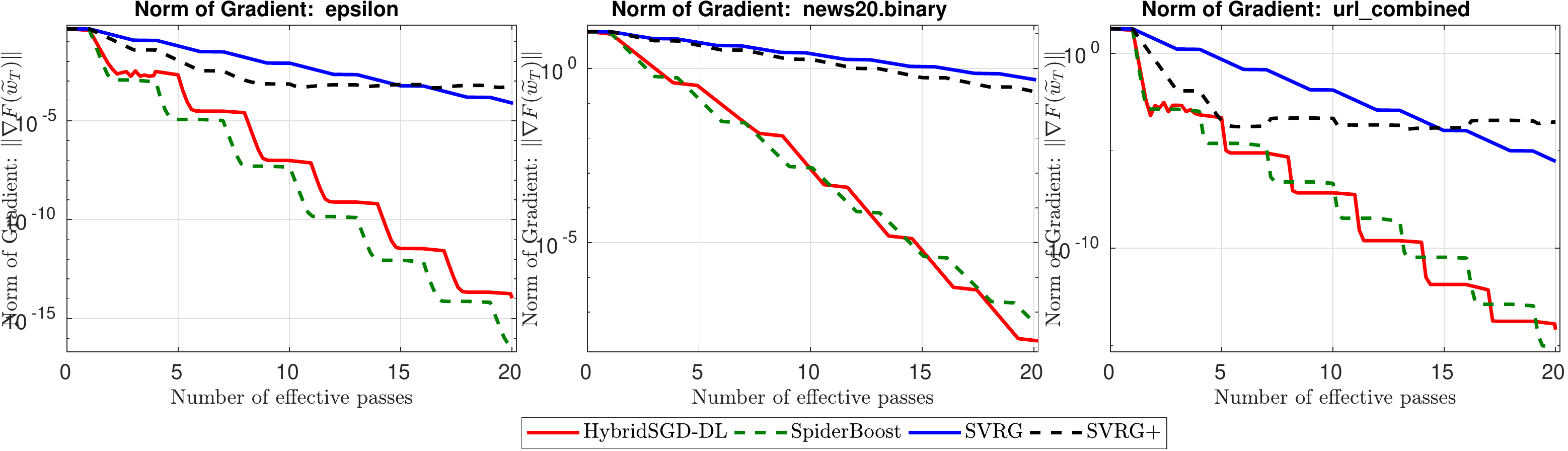}
\includegraphics[width = 1\textwidth]{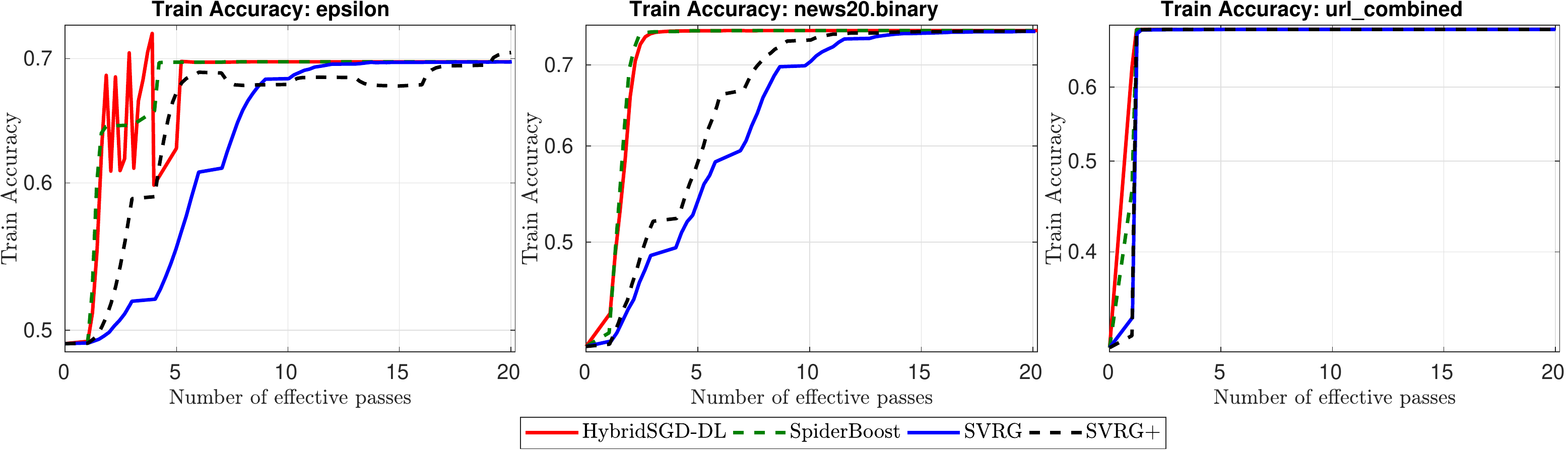}
\includegraphics[width = 1\textwidth]{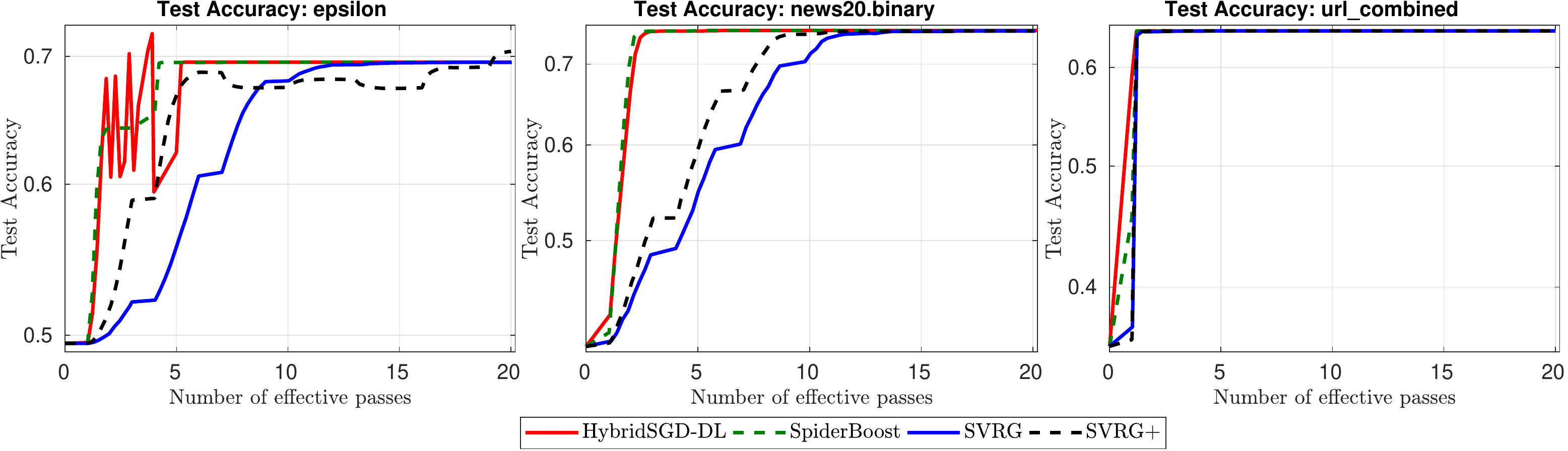}
\vspace{-1ex}
\caption{\done{The training loss and gradient norms of \eqref{eq:exam2}  with loss $\ell_2$: Mini-batch.}}\label{fig:binary_batch4_32}
\end{center}
\vspace{-2ex}
\end{figure}

In this experiment, although SGD2 has faster decrease during the first few epochs, our HybridSGD-SL and HybridSGD-ASL eventually achieve lower training loss and gradient norm in all datasets while reaching similar  training and testing accuracies as SGD2.

Regarding the double-loop variants, our HybridSGD-DL once again has better performance than SVRG and SVRG+ while having comparable performance with SpiderBoost in terms of training loss, gradient norm, and accuracies.

%

\bibliographystyle{plain}
\newcommand{\beforebib}{\vspace{0ex}}


\end{document}